\documentclass[a4paper,11pt]{article}
\usepackage[utf8]{inputenc}
\usepackage{latexsym}
\usepackage{amsfonts}
\usepackage{amssymb}

\usepackage{enumitem} 
\usepackage{braket}
\usepackage{hyperref}

\usepackage{default}
\usepackage[T1]{fontenc}
\usepackage{amsmath}
\usepackage{amssymb}
\usepackage{amsfonts}
\usepackage{amsthm}
\usepackage{indentfirst}
\usepackage{psfrag}
\usepackage{amscd,stmaryrd,latexsym}
\usepackage[margin=1.3in]{geometry}
\usepackage[pdftex]{graphicx}
\newtheorem{thm}{Theorem}[section]
\newtheorem{lem}{Lemma}[section]

\newtheorem{Prop}{Proposition}[section]
\newtheorem*{St*}{Statement}

\newtheorem*{theorem*}{Theorem}
\theoremstyle{remark}

\theoremstyle{definition}

\theoremstyle{remark}
\newtheorem{oss}{Remark}[section]

\newcommand{\be}{\begin{equation}}
\newcommand{\ee}{\end{equation}}
\newcommand{\R}{\mathbb{R}}
\newcommand{\RP}{\mathbb{RP}}
\newcommand{\N}{\mathbb{N}}

\newcommand{\Het}{\mathbb{H}}
\newcommand{\OHet}{\overline{\mathbb{H}}}

\newcommand{\He}{\mathbb{H}_{\varepsilon}}

\newcommand{\ca}[2]{\mathcal{#1}_{#2}}

\newcommand{\spt}[1]{\text{spt}\,\|#1\|}

\newcommand\res{\mathop{\hbox{\vrule height 7pt width .5pt depth 0pt
\vrule height .5pt width 6pt depth 0pt}}\nolimits}

\def\eps{\mathop{\varepsilon}}

\def\Ece{\mathop{\mathcal{E}_{\varepsilon}}}

\def\Hc{\mathop{\mathcal{H}}}

\def\s{\sigma}

\def\Om{\Omega}
\def\om{\omega}

\def\p{\partial}

\def\dm{d_{\overline{M}}}

\def\eps{\mathop{\varepsilon}}

\def\Om{\Omega}
\def\om{\omega}
\def\p{\partial}

\newcommand{\Rc}[1]{\text{Ric}_{#1}} 
\DeclareMathAlphabet{\mathscr}{OT1}{pzc}{m}{it}

\begin{document} 

\title{\textbf{Multiplicity-$1$ minmax minimal hypersurfaces in manifolds with positive Ricci curvature}}
\author{Costante Bellettini\\
University College London}
\date{}

\maketitle

\begin{abstract}
We address the one-parameter minmax construction, via Allen--Cahn energy, that has recently lead to a new proof of the existence of a closed minimal hypersurface in an arbitrary compact Riemannian manifold $N^{n+1}$ with $n\geq 2$ (see Guaraco's work \cite{Gua}). We obtain the following multiplicity-$1$ result: if the Ricci curvature of $N$ is positive then the minmax Allen--Cahn solutions concentrate around a multiplicity-$1$ hypersurface, that may have a singular set of dimension $\leq n-7$. This result is new for $n\geq 3$ (for $n=2$ it is also implied by the recent work \cite{ChoMan} by Chodosh--Mantoulidis). The argument developed here is geometric in flavour and exploits directly the minmax characterization of the solutions. An immediate corollary is that every compact Riemannian manifold $N^{n+1}$ with $n\geq 2$ and positive Ricci curvature admits a two-sided closed minimal hypersurface, possibly with a singular set of dimension at most $n-7$. This existence result also follows from multiplicity-$1$ results developed within the Almgren--Pitts framework, see \cite{KMN}, \cite{MarNev2}, \cite{Ram}, \cite{Zhou3}). 
\end{abstract}

\section{Introduction}
The existence of a closed minimal hypersurface in an arbitrary compact Riemannian manifold $N^{n+1}$ ($n\geq 2$) is a fundamental result that dates back to the early 80s: the minmax procedure developed by Almgren and Pitts \cite{Alm} \cite{Pit} leads to a stationary integral varifold whose support is, by the regularity/compactness theory of Schoen--Simon--Yau \cite{SSY} and Schoen--Simon \cite{SS}, a smoothly embedded hypersurface except possibly for a singular set of codimension $7$ or higher. This existence result has been recently reproven, using an alternative and more classical minmax approach based on an Allen--Cahn approximation scheme (summarized briefly in Section \ref{recall_Gua}), through the combined efforts of Guaraco, Hutchinson, Tonegawa, Wickramasekera \cite{Gua}, \cite{HT}, \cite{Ton}, \cite{TonWic}, \cite{Wic}. As in the former approach, one obtains the minimal hypersurface as an integral varifold, that turns out to be smooth away from a singular set of codimension $\geq 7$ (moreover, its regular part has Morse index at most one by \cite{Hie}, \cite{Gas}). In general, in both approaches, the hypersurface may have several connected components, each a priori with an integer multiplicity $\geq 1$.

In the Almgren--Pitts framework, Ketover--Marques--Neves \cite{KMN} show (relying also on \cite{Zhou2}) that, when $N^{n+1}$ is orientable with positive Ricci curvature and $2\leq n\leq 6$, the minimal hypersurface is two-sided and has multiplicity $1$. This result is extended to $n\geq 7$ by Ram\'{\i}rez-Luna in \cite{Ram} (relying on \cite{Zhou1}).
Zhou \cite{Zhou3} obtains multiplicity-$1$ and two-sidedness in the case $2\leq n\leq 6$ for bumpy metrics and for metrics with positive Ricci, more generally for one- or multi-parameter minmax, thereby confirming a well-known conjecture of Marques--Neves \cite[1.2]{MarNev2} (see also \cite[Addendum]{MarNev2}).

Marques--Neves's conjecture has a natural counterpart for minmax constructions via Allen--Cahn. For $n=2$ it follows from Chodosh--Mantoulidis \cite{ChoMan} (valid more generally for solutions with bounded Morse index, not necessarily minmax solutions) that the minimal surface obtained by the (one- or multi-parameter) Allen-Cahn minmax is two-sided and has multiplicity $1$, in the case of bumpy metrics and in the case of metrics with positive Ricci.

The multiplicity issue is ubiquitous in variational problems and its resolution has far-reaching consequences, specifically (but not only) in minmax constructions, see e.g.~\cite{ChoMan}, \cite{MarNev2}, \cite{Zhou3} and, in yet another minmax approach (for surfaces in arbitrary codimension), \cite{PigRiv}. 
We obtain here the following multiplicity-$1$ result (new for $n\geq 3$), which applies to the Allen--Cahn one-parameter minmax construction in \cite{Gua}. 

\begin{thm}
\label{thm:mult1_Ricci7}
Let $N$ be a compact Riemannian manifold of dimension $n+1$ with $n\geq 2$ and with positive Ricci curvature. 
Then the Allen--Cahn minmax in \cite{Gua} yields on $N$ a multiplicity-$1$ smooth minimal hypersurface $M$ with $\text{dim}\left(\overline{M}\setminus M\right)\leq n-7$. 
\end{thm}

While the two frameworks are different in spirit (see also Remark \ref{oss:compare_boundaries_path}), Theorem \ref{thm:mult1_Ricci7} can be viewed as the Allen--Cahn counterpart of the Almgren--Pitts results contained in \cite{KMN}, \cite{Ram}, \cite{Zhou2}, \cite{Zhou1}. 
Implicit in the multiplicity-$1$ conclusion of Theorem \ref{thm:mult1_Ricci7}, we have an alternative proof of the following existence result for \textit{two-sided} minimal hypersurfaces (which is known from \cite{KMN} and \cite{Zhou3} for $2\leq n\leq 6$ and \cite{Ram} for $n\geq 7$).

\begin{thm}
\label{thm:two-sided}
In any compact Riemannian manifold of dimension $n+1$ with $n\geq 2$ and with positive Ricci curvature there exists a smooth \textit{two-sided} minimal hypersurface $M$ with $\text{dim}\left(\overline{M}\setminus M\right)\leq n-7$.
\end{thm}

\begin{oss}
\label{oss:otherconsequences2}
It also follows easily that $M$ in Theorem \ref{thm:mult1_Ricci7} is connected.
\end{oss}

\subsection{Strategy}

To obtain the multiplicity-$1$ result we exploit directly the minmax characterization (rather than finite index properties). More precisely, we do not analyse the Allen--Cahn solutions constructed in \cite{Gua}, that concentrate on the minimal hypersurface; we only retain the minmax value that they achieve and prove the following result, from which Theorem \ref{thm:mult1_Ricci7} is easily deduced.

\begin{thm}
\label{thm:compare2M_Ricci7}
Let $N$ be a compact Riemannian manifold of dimension $n+1$, $n\geq 2$, and with positive Ricci curvature. Let $M \subset N$ be any smooth minimal hypersurface such that $\text{dim}\left(\overline{M}\setminus M\right)\leq n-7$, $M$ is stationary in $N$, and for every $x\in\overline{M}$ there exists a geodesic ball in $N$ centred at $x$ in which $M$ is stable. Then the minmax value $c_{\eps}$ obtained by \cite{Gua} (for $\eps<1$) satisfies 
$$\limsup_{\eps \to 0} c_{\eps}<2{\Hc}^n(M).$$ 
\end{thm}

\begin{oss}
The assumptions on $M$ in Theorem \ref{thm:compare2M_Ricci7} are valid for any minimal hypersurface produced by the minmax in \cite{Gua} (using \cite{HT}, \cite{Ton}, \cite{TonWic}, \cite{Wic}). Then it is readily checked that Theorem \ref{thm:mult1_Ricci7} follows from Theorem \ref{thm:compare2M_Ricci7}.
\end{oss}

\begin{oss}
\label{oss:otherconsequences}
It is not hard to check that, under the assumptions of Theorem \ref{thm:mult1_Ricci7}, the area of the minmax hypersurface is less than or equal to that of an arbitrary two-sided minimal hypersurface in $N$ that has the properties listed for $M$ in Theorem \ref{thm:compare2M_Ricci7}.
\end{oss}

\begin{oss}
Part of the arguments developed here does not make use of the Ricci positive assumption. In fact, analogues of Theorems \ref{thm:mult1_Ricci7} and \ref{thm:two-sided} in the case in which $2\leq n\leq 6$ and $N$ is endowed with a bumpy metric will appear in \cite{B}. 
\end{oss}

We will now outline the proof of Theorem \ref{thm:compare2M_Ricci7}. Given $M$ as in Theorem \ref{thm:compare2M_Ricci7}, the idea is to produce, for all sufficiently small $\eps$, a continuous path in $W^{1,2}(N)$ that joins the constant $-1$ to the constant $+1$ and such that the Allen--Cahn energy evaluated along the path stays below $2\mathcal{H}^n(M)$ by a fixed positive amount independent of $\eps$ (determined only by geometric properties of $M \subset N$). Since this is one of the admissible paths for the minmax construction in \cite{Gua}, the inequality in Theorem \ref{thm:compare2M_Ricci7} must hold.

\medskip

The construction of the path is geometric in flavour and employs classical tools (coarea formula, semi-linear parabolic theory). For simplicity, in this introduction we illustrate it mainly in the case $2\leq n\leq 6$, so that $M$ is smooth and closed. We think of $M$ with multiplicity $2$ as an immersed two-sided hypersurface, namely its double cover $\tilde{M}$ with the standard projection. This immersion, that we denote by $\iota:\tilde{M}\to N$, is minimal and unstable (by the positiveness of the Ricci curvature). It is possible to find a (sufficiently small) geodesic ball $B\subset M$ such that the lack of stability still holds for deformations that do not move $B$ (this follows by a capacity argument). We then find a deformation of $\iota$ that is area-decreasing on some time interval $[0,t_0]$. This deformation is depicted in the top row of Figure \ref{fig:cut_B_deform}. (We can choose the initial speed of the deformation to be non-negative on $\tilde{M}$, therefore the deformation ``pushes away from $M$''.) We denote by $2\mathcal{H}^n(M)-\tau$ the area of the immersion at time $t_0$, for some $\tau>0$.
If we cut out $B$ from $M$ we are left with an immersion with boundary, namely $\left.\iota\right|_{\tilde{M}\setminus \iota^{-1}(B)}$. We can restrict the previous deformation to $\left.\iota\right|_{\tilde{M}\setminus \iota^{-1}(B)}$, obtaining an area-decreasing deformation (at fixed boundary) on the time interval $[0,t_0]$. This time the area changes from $2\mathcal{H}^n(M) -2\mathcal{H}^n(B)$ to $2\mathcal{H}^n(M) -2\mathcal{H}^n(B)-\tau$. This deformation is depicted in middle row of Figure \ref{fig:cut_B_deform}.
Now we proceed to close the hole at $B$ continuously (bottom row of Figure \ref{fig:cut_B_deform}), reaching, say in in time $1$, the same immersion depicted in the top-right picture of Figure \ref{fig:cut_B_deform}. (In view of the forthcoming construction, it is helpful to think of closing the hole at $B$ by inserting a weighted copy of $B$ and letting the weight increase from $0$ to $2$. Abusing language, we will talk in this introduction of immersions also to indicate these ``weighted immersions''.) The area increases from $2\mathcal{H}^n(M) -2\mathcal{H}^n(B)-\tau$ to $2\mathcal{H}^n(M)-\tau$. Therefore, in going from the middle-left picture to the bottom-right picture of Figure \ref{fig:cut_B_deform}, we have produced a ``path of immersions'' along which the value of the area stays stricly below $2\mathcal{H}^n(M)$, at least by $\min\{\tau,2\mathcal{H}^n(B)\}$, a fixed amount that only depends on the geometry of $M\subset N$.

This path of immersions is then ``reproduced at the Allen--Cahn level'', i.e.~replaced by a continuous path $\gamma:[0,{t_0}+1]\to W^{1,2}(N)$. Each function in the image if this curve is a suitable ``Allen--Cahn approximation'' of the corresponding immersion. To construct this, one fits one-dimensional Allen--Cahn solutions in the normal bundle to the immersion, respecting multiplicities: at points with multiplicity $1$ and $2$ we will use respectively the top and bottom profiles in Figure \ref{fig:profiles_1d}. The image of the immersion corresponds to points where the function transitions between $-1$ and $+1$, with a double transition for points of multiplicity $2$. The operation of closing the hole at $B$ can be reproduced at the Allen--Cahn level thanks to the multiplicity-$2$ assumuption on $B$: in the normal direction to $B$, the profile of the function goes from being constantly $-1$ to looking like the bottom picture in Figure \ref{fig:profiles_1d}. Moreover, the operation is continuous in $W^{1,2}(N)$. (Working in $W^{1,2}$, hypersurfaces are naturally diffused, and continuous weights are allowed: this justifies the geometric operation of closing the hole by increasing the weight of $B$ from $0$ to $2$, as mentioned above\footnote{In a similar spirit, when we will write an Allen--Cahn approximations of an immersion with boundary, there will be no sharp transition of multiplicity at the boundary: the weight will instead continuously decrease to $0$ in a neighbourhood of the boundary of the hypersurface.}.) The construction of $\gamma$ is done for all sufficiently small $\eps$ (the parameter of the Allen--Cahn energy) and, moreover, for all sufficiently small $\eps$ the Allen--Cahn energy all along $\gamma$ is a close approximation of the area of the corresponding immersions; therefore, for all sufficiently small $\eps$, the energies stay below $2\mathcal{H}^n(M)$ by a fixed ``geometric'' amount $\approx \min\{\tau,2\mathcal{H}^n(B)\}$. 

We now consider $\gamma(0)$ and $\gamma(t_0+1)$ (respectively the Allen--Cahn approximations of the immersions in the middle-left and top-right picture of Figure \ref{fig:cut_B_deform}). For the latter, we use a (negative) Allen--Cahn gradient flow (to which we add a small forcing term). We build a mean convex barrier (by writing a suitable Allen--Cahn approximation of $\iota$), that sits below $\gamma(t_0+1)$. Thanks to this, we show that the flow deforms $\gamma({t_0}+1)$ continuously into a stable Allen--Cahn solution, which has to be the constant $+1$ by the Ricci-positive assumption. Along this flow, the Allen--Cahn energy is controlled by the initial bound $\approx 2\mathcal{H}^n(M) -\tau$.
The function $\gamma(0)$ is $\approx +1$ close to $M\setminus B$ and $\approx -1$ away from a tube around $M\setminus B$: this function can be explicitly connected to the constant $-1$ (continuously in $W^{1,2}$) with approximately decreasing Allen--Cahn energy. (A close geometric operation to think of is to give weight $2$ to $M\setminus B$ and let the weight decrease continuously to $0$.) 
Reversing the latter path, then composing it with $\gamma$ and finally with the path obtained via the flow, we produce the promised continuous path in $W^{1,2}(N)$ that joins $-1$ to $+1$ and has the desired energy control.

\begin{figure}[h]
\centering
\includegraphics[scale=0.4]{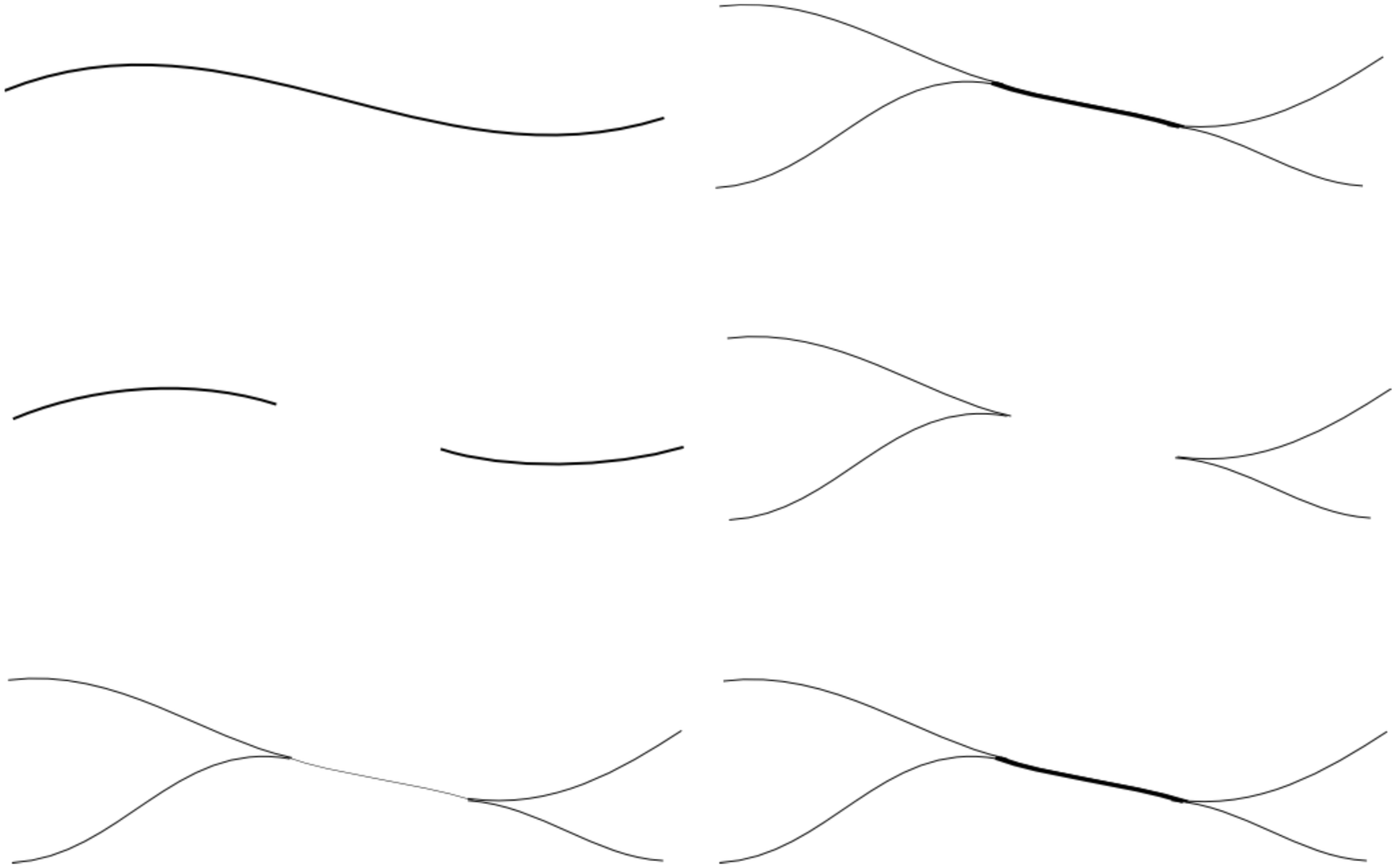}
\put (-260.5, 230){\small{$\iota:\tilde{M}\to N$}}
\put (-270.5, 219){\small{(double cover of $M$)}}
\put (-291.5, 157){\small{$\left.\iota\right|_{\tilde{M}\setminus \iota^{-1}(B)}$}}
\put (-241.5, 157){\small{(make a hole}}
\put (-241.5, 147){\small{at $B$)}}
\put (-105.5, 238){\small{push away from $M$}}
\put (-90.5, 226){\small{keeping $B$ fixed}}
\put (-105.5, 148){\small{push away from $M\setminus B$}}
\put (-236.5, 24){\small{close the hole at $B$ continuously}}
\put (-236.5, 15){\small{(increasing the density)}}
\caption{Cut, deform, fill in. The path of ``immersions'' in the second and third row reaches the same immersion depicted in the top-right picture.}
\label{fig:cut_B_deform}
\end{figure}

\medskip

We stress that the functions $\gamma(t)$, $t\in[0,{t_0}+1]$, that we call ``Allen--Cahn approximations'' of the corresponding immersions, are not solutions of the Allen--Cahn equation, they only realize the ``correct'' energy value. In fact, for $t\neq  {t_0}+1$, we do not even analyse the Allen--Cahn first variation of $\gamma(t)$. (We do so, on the other hand, for $\gamma({t_0}+1)$ and this is important for the gradient flow argument.) The loss of information on the first variation is compensated by the ad hoc structure of the Allen--Cahn approximation: its level sets are by construction sheets over the given immersed hypersurface, so that the Allen--Cahn energy is an effective approximation of area (by the coarea formula) and the geometric information can be translated to the Allen--Cahn level.

\medskip

We digress to comment briefly on the operation of connecting $\gamma(0)$ to the constant $-1$. We could in fact use an Allen--Cahn flow for this step, by first slightly deforming $\gamma(0)$ into another function (with a similar profile, so that it still approximates $2|M\setminus B|$, but with a more effective first variation) and then running the Allen-Cahn flow, that deforms this function to the constant $-1$. We do not argue in this way, since we are able to produce an explicit deformation of $\gamma(0)$ to $-1$, which is elementary and straightforward. We stress, however, that the deformation that we exhibit mimics the Allen--Cahn flow, and is therefore a regularized version of the Brakke flow that starts at $2|M\setminus B|$ and vanishes instantaneously. While the Brakke flow creates a discontinuity in space-time, at the Allen--Cahn level we gain continuity (and the flow reaches $-1$ in time $O(\eps|\log\eps|)$). As we mentioned above, an intuitive geometric counterpart of the deformation connecting $\gamma(0)$ to $-1$, is the one that continuously decreases the weight of $M\setminus B$ from $2$ to $0$ in time $O(\eps|\log\eps|)$. 

A similar remark applies to the portion of path that ``closes the hole at $B$''. If we were to take the $\eps\to 0$ limit we would see that $2|B|$ appears instantaneously and we would go (in a discontinuous way) from the immersion in the middle-right picture to the one in the bottom-right picture of Figure \ref{fig:cut_B_deform}. Again, at the Allen--Cahn level we gain continuity by increasing the weight of $B$ from $0$ to $2$ in time $O(\eps|\log\eps|)$.

\begin{figure}[h]
\centering
\includegraphics[scale=0.45]{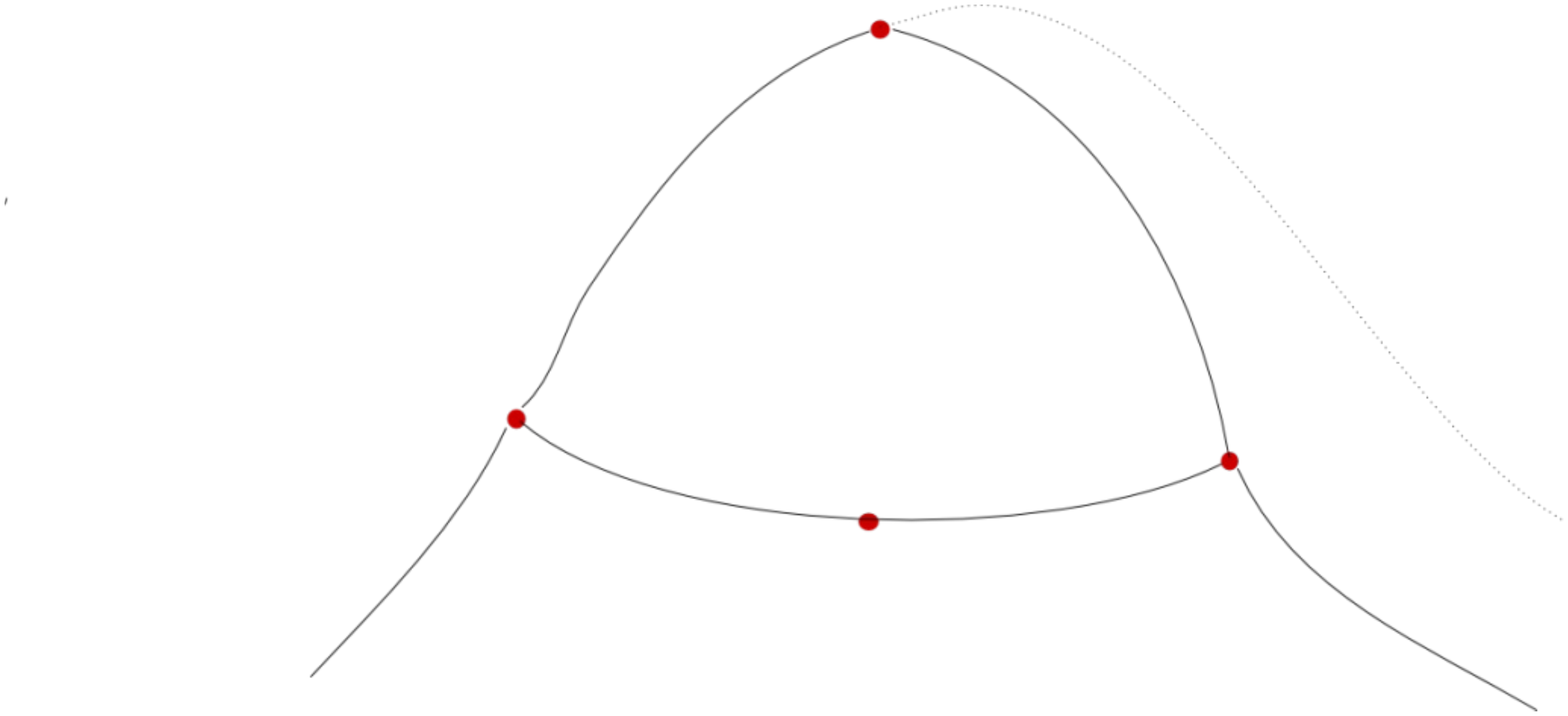}
\put (-199.5, 275){\small{$G^{\eps}_0\approx 2|M|$}}
\put (-251.5, 239){\small{make}}
\put (-251.5, 229){\small{a hole}}
\put (-251.5, 219){\small{at $B$}}
\put (-286.5, 179){\small{$\gamma(0)=f$}}
\put (-286.5, 167){\small{$=g_0$}}
\put (-75.5, 165){\small{$\gamma({t_0}+1)=g_{t_0+1}$}}
\put (-165.5, 146){\small{$g_{t_0}$}}
\put (-200.5, 155){\small{$g_t$}}
\put (-240.5, 143){\small{push away}}
\put (-240.5, 133){\small{from $M\setminus B$}}
\put (-125.5, 152){\small{$g_{t_0+s}$}}
\put (-125.5, 140){\small{close hole}}
\put (-125.5, 130){\small{at $B$}}
\put (-125.5, 120){\small{continuously}}
\put (-90.5, 220){\small{push away from $M$}}
\put (-90.5, 208){\small{keeping $B$ fixed}}
\put (-306.5, 107){\small{$\to -1$}}
\put (-15.5, 105){\small{$\to +1$}}
\caption{Lowering the peak (landscape for the Allen--Cahn energy). The same labels as in Figure \ref{fig:cut_B_deform} are used, to denote deformations that reproduce those in Figure \ref{fig:cut_B_deform}.}
\label{fig:lower_peak}
\end{figure}

\medskip

We emphasise the following point of view on the construction of the path that was sketched above. Consider $\iota:\tilde{M} \to N$: we exhibit two one-sided deformations that decrease area and that can be reproduced for the Allen--Cahn approximations. One (from the top-left to the middle-left picture of Figure \ref{fig:cut_B_deform}) has the geometric effect of removing $2|B|$. The other (from the top-left to the top-right picture of Figure \ref{fig:cut_B_deform}) is a deformation of $\iota$ as an immersion, induced by an initial velocity compactly supported away from $B$. We will denote by $G^{\eps}_0$ in Section \ref{level_sets} the Allen-Cahn approximation of $\iota$. Then, with reference to Figure \ref{fig:lower_peak}, and using the notation $\gamma(0)$, $\gamma(t_0+1)$ respectively for the Allen--Cahn approximations of the immersions in the middle-left and top-right picture of Figure \ref{fig:cut_B_deform}, the two deformations just described, implemented at the Allen--Cahn level, correspond respectively to ``going from $G^{\eps}_0$ to $\gamma(0)$'' and  ``going from $G^{\eps}_0$ to $\gamma({t_0}+1)$''. The two deformations are linearly independent, as the former acts on a compact set containing $B$ while the latter in the complement of this compact set. Note that it may well be that $\iota$ is an immersion with Morse index $1$ (for example, the double cover of an equator of $\RP^3$). The area-decreasing deformation that removes $B$ is clearly not a deformation of $\iota$ as an immersion; it can be reproduced as a continuous deformation at the Allen--Cahn level, thanks to the fact that we have multiplicity $2$ on $B$, so that the profile of $G^{\eps}_0$ in the normal bundle to $B$ looks like the bottom one in Figure \ref{fig:profiles_1d}; this profile can be connected continuously to the constant $-1$ with controlled energy. 

The function $\gamma(0)$ (that will be denoted by $f=g_0$ in Section \ref{all_paths}) can be connected to the constant $-1$ and the function $\gamma({t_0}+1)$ can be connected to the constant $+1$, as described in the sketch given earlier. We thus have a ``recovery path'' for the value $2\mathcal{H}^n(M)$: this path connects $-1$ to $+1$ (passing through $G^{\eps}_0$) and the maximum of the Allen--Cahn energy along this path is $\approx 2\mathcal{H}^n(M)$. What we achieve is to deform this path in the sorroundings of $G^{\eps}_0$, exploiting the information that we have gained on the landscape, specifically we deform the portion between $\gamma(0)$ and $\gamma({t_0}+1)$. From $\gamma(0)$ we use a deformation that reproduces the one in the middle row of Figure \ref{fig:cut_B_deform}. By doing this we reach a function $g_{t_0}$ (notation as in Section \ref{all_paths}). Now we close the hole continuously, replicating the deformation in the bottom row of Figure \ref{fig:cut_B_deform}, reaching the function $\gamma({t_0}+1)$ (that will be denoted by $g_{{t_0}+1}$ in Section \ref{all_paths}). We have thus found a path $\gamma:[0,{t_0}+1]\to W^{1,2}(N)$, from $\gamma(0)$ to $\gamma({t_0}+1)$, that lowers the peak, compared to the initial ``recovery path''. This follows thanks to the fact that the Allen--Cahn energy is a close approximation of the area of the corresponding immersion, so we inherit the estimates that we had for the path of immersions that joins the middle-left picture to the bottom-right picture of Figure \ref{fig:cut_B_deform}.

This shows that the landscape around $G^{\eps}_0$ is reminiscent of one where the Morse index is $\geq 2$. \textit{However} $G^{\eps}_0$ is not a stationary point for the Allen--Cahn energy. In fact, we never need to compute the Allen--Cahn first or second variation along these deformations, it suffices to know that the Allen--Cahn energy at $\gamma(0)$ and $\gamma({t_0}+1)$ is strictly less than its value at $G^{\eps}_0$ (by a fixed amount indepednent of $\eps$).

\begin{oss}
\label{oss:compare_boundaries_path}
It is natural to ask whether the path from $-1$ to $+1$ produced in the earlier sketch can be imitated (e.g.~in an Almgren--Pitts framework) by a one-parameter family of boundaries in $N$. For the portion $\gamma:[0,t_0+1]\to W^{1,2}(N)$, rather than increasing the weight of $B$ from $0$ to $2$ (which cannot be done in the class of boundaries) one can argue by doubling $M\setminus B$ and inserting a small neck at $B$, then pushing this hypersurface away from $M$ without moving the neck (and decreasing the area), then closing the neck. (An operation of this type is analysed in \cite{KMN}. To avoid confusion, we point out that for our path $\gamma$, the nodal sets $\{\gamma_t=0\}_{t\in[t_0,t_0+1]}$ are not cylindrical necks.) It is conceivable that one could then use mean-curvature-flow to drift away from $M$ until extinction time and thus imitate, by using boundaries, the portion of path from $\gamma(t_0+1)$ to $+1$. (The use of a flow for this purpose does not appear to have been investigated in the literature. Generally speaking, gradient flows may be easier to use in the Allen--Cahn framework, since the parabolic problem is semilinear, has long-time existence and singularities do not appear.) For the portion of path that goes from $\gamma(0)$ to the constant $-1$, the spirit of the Allen--Cahn deformation is again different than a deformation of boundaries (compare with \cite{KMN}, \cite{Zhou2}), since its geometric analogues are either a continuous weight-decrease from $2$ to $0$ or a Brakke flow that instantaneously makes $M\setminus B$ disappear. The Allen--Cahn framework allows a very straightforward way to produce this portion of path. (Some extra challenges have to be overcome in \cite{KMN}, for example the catenoid estimate.)
\end{oss}

For $n\geq 7$, we still employ the idea illustrated in low dimensions. Its implementation, however, is rendered somewhat harder by the presence of the singular set: standard tubular neighbourhoods and Fermi coordinates for $M$ (that are essential to fit one-dimensional Allen--Cahn profiles in the normal bundle to $M$) are not available. While the geometric ideas remain the same as in the low-dimensional case, we need to additionally study certain analytic properties. Denote by $\dm:N\to [0,\infty)$ the distance function to $\overline{M}$. The value $\dm(x)$ is always realized by a geodesic (possibly more than one) from $x$ to a \textit{smooth} point of $\overline{M}$. This allows to analyse the cut-locus of $\dm$ (restricting to $\{\dm<\text{inj}(N)\}$), following \cite{ManteMennu}, and obtain $n$-rectifiability properties for it. This leads (for the moment) to the existence of a suitable replacement for Fermi coordinates, which becomes the usual one on any compact subset of $\tilde{M}$. Denote by $\iota:\tilde{M} \to N$ the immersion given by the standard projection from the double cover of $M$. We choose $K\subset \tilde{M}$ compact (sufficiently large) and a geodesic ball $B\subset \iota(K)$ (sufficiently small) so that $\iota:\tilde{M}\to N$ admits a deformation as an immersion that decreases area and only moves $K \setminus \iota^{-1}(B)$. (This is analogous to what we did in the lower dimensional case, except that this time we additionally need a deformation that does not move $M$ close to the singular set.) The set $K$ will play the role that was of $\tilde{M}$ in the low-dimensional case. Around $\iota(K)$ we define Allen--Cahn approximations of suitable immersions by fitting one-dimensional Allen--Cahn profiles in the normal bundle. Away from $\iota(K)$, we use the level sets of $\dm$ to complete the definition of the desired Allen--Cahn approximations and create (as in the low-dimensional case) a continuous path $\gamma:[0,{t_0}+1]\to W^{1,2}(N)$ with controlled energy.\footnote{The continuity of the path, the energy bounds and the mean-convexity of $G^{\eps}_0$ (see subsequent lines) ultimately rest on the fact that almost every level set of $\dm$ is almost everywhere smooth, with mean curvature pointing away from $M$. These properties can be obtained on almost every level set by standard arguments. The almost everywhere information is sufficient for our purposes, because in the Allen--Cahn framework hypersurfaces are ``diffused''. (For contrast, in the case of boundaries of Caccioppoli sets, all level sets of $\dm$ must be analysed.)} 
Exploiting further the $n$-rectifiability of the cut-locus, we analyse the singular part of $\Delta \dm$ and (using also the Ricci-positive condition) we obtain that, restricting to $\{\dm< \text{inj}(N)\}$, the distributional Laplacian of $\dm$ is a positive Radon measure.
This translates into a mean convexity property for the Allen--Cahn approximation $G^{\eps}_0$ of $\iota:\tilde{M}\to N$. With a suitable smoothing operation, we obtain from $G^{\eps}_0$ a smooth barrier $m$ that is still mean-convex for the negative Allen--Cahn gradient flow (as for $2\leq n\leq 6$, we add a small forcing term). By employing $m$ we produce the part of the path that connects $\gamma({t_0}+1)$ to the constant $+1$. 

\subsection{Structure of the paper (and remarks for $n\leq 6$)} 
Except for properties of the distance function borrowed from \cite{ManteMennu} (in Section \ref{distance} we point out the relevant modifications needed), the proof is self-contained.

After the preliminary Section \ref{preliminaries}, we begin the proof of Theorem \ref{thm:compare2M_Ricci7}, which we write for $n\geq 7$, assuming the existence of a singular set $\overline{M}\setminus M$ of dimension $\leq n-7$. While the underlying ideas are the same for all dimensions, the proof becomes considerably shorter and more straightforward in the absence of singular set, in particular when $n\leq 6$. In detail, Sections \ref{distance} and \ref{level_sets}, in which we study the distance function to $\overline{M}$ and its level sets, can be omitted when $\overline{M}=M$ and one can use standard facts about tubular neighbourhoods of smooth closed hypersurfaces. In Section \ref{instability} we identify a large unstable region $2|K\setminus B|$ and in Section \ref{immersions} the immersions that will be relevant for the construction of the path. The compact set $K$ that we need to work with in Sections \ref{instability} and \ref{immersions} can be replaced simply by $\tilde{M}$ when $\overline{M}=M$, and in this case the definitions of the Allen--Cahn approximations of the relevant immersions given in Section \ref{all_paths} become simpler. In Section \ref{reach_1} we construct a barrier $m$ by suitably mollifying a Lipschitz function $G^{\eps}_0$, which is defined from the level sets of $\dm$ and is an Allen--Cahn approximation of $\iota:\tilde{M}\to N$. This convolution procedure (described in Appendix \ref{mollifiers}) ensures smoothness and mean-convexity of $m$, which is important for our arguments. If $\overline{M}=M$, $G^{\eps}_0$ is already smooth and mean-convex and no smoothing is needed, so Appendix \ref{mollifiers} and part of Section \ref{reach_1} can be omitted. In Section \ref{final_argument_Ricci7} we complete the proof of Theorem \ref{thm:compare2M_Ricci7}, and subsequently of Theorems \ref{thm:mult1_Ricci7} and \ref{thm:two-sided}.

\medskip

\textbf{Acknowledgments}: This work is partially supported by EPSRC under grant EP/S005641/1. I am thankful to Neshan Wickramasekera for introducing me to the Allen--Cahn functional and its geometric impact. I would like to thank Felix Schulze for helpful conversations about parabolic PDEs and mean curvature flow. 
I am grateful to Otis Chodosh for a mini-course on geometric features of the Allen--Cahn equation, held at Princeton University in June 2019, and for related discussions. The insight that I gained at that time proved very valuable when I addressed the problem discussed in this work. These lectures took place while I was a member of the Institute for Advanced Study, Princeton: I gratefully acknowledge the excellent research environment and the support provided by the Institute and by the National Science Foundation under Grant No. DMS-1638352. 

\section{Preliminaries}
\label{preliminaries}
We give a brief summmary of the construction in \cite{Gua}, then introduce the one-dimensional Allen--Cahn profiles that will be needed for our proof.

\subsection{Reminders: Allen--Cahn minmax approximation scheme}
\label{recall_Gua}

We recall the minmax construction in \cite{Gua}. For $\eps\in(0,1)$ consider the functional 
$$\Ece(u)=\frac{1}{2\sigma}\int_N \eps\frac{|\nabla u|^2}{2} + \frac{W(u)}{\eps}$$
on the Hilbert space $W^{1,2}(N)$. Here $W$ is a $C^2$ double well potential with non-degenerate minima at $\pm 1$, for example $W(x)=\frac{(1-x^2)^2}{4}$, suitably modified (for technical reasons) outside $[-2,2]$--- a typical choice, that we also follow, is to impose quadratic growth to $\infty$ (others are also possible), and $\sigma$ is a normalization constant, $\sigma=\int_{-1}^1 \sqrt{W(t)/2}\,dt$. Consider continuous paths in $W^{1,2}(N)$ that start at the constant $-1$ and end at the constant $+1$: this is the class of admissible paths. A ``wall condition'' is ensured and yields the existence of a minmax solution $u_{\eps}$ to $\Ece'(u_{\eps})=0$. Moreover, upper and lower energy bounds are established, uniformly in $\eps$.

In order to produce a stationary varifold, one considers $w_{\eps} = \Phi(u_{\eps})$ as in \cite{HT} (with $\Phi(s)=\int_0^s \sqrt{W(t)/2}\,dt$) and defines the $n$-varifolds 
$$V^{\eps}(A) = \frac{1}{2\sigma}\int_{-\infty}^\infty V_{\{w_{\eps}=t\}}(A) dt.$$
The analysis in \cite{HT} (which only requires the stationarity of $u_{\eps}$ and no assumption on their second variation), together with the upper and lower bounds for $\Ece(u_{\eps})$, gives that $V^{\eps}$ converges subsequentially, as $\eps \to 0$, to an integral $n$-varifold $V \neq 0$ with vanishing first variation.

Thanks to the fact that the Morse index of $u_{\eps}$ is $\leq 1$ for all $\eps$, \cite{Gua} reduces the problem locally in $N$ to one that concerns stable Allen--Cahn solutions, as in \cite{Ton}. For these, the regularity theory of \cite{Wic} and \cite{TonWic} applies and gives that $\spt{V}$ is smoothly embedded away from a possible singular set of dimension $\leq n-7$, i.e.~$V$ is the varifold of integration over a finite set of closed minimal hypersurfaces, each counted with integer multiplicity: $V=\sum_{j=1}^K q_j |M_j|$, with $q_j\in \N$ and $M_j$ minimal and smooth away from a set of dimension $\leq n-7$ ($|M_j|$ denotes the multiplicity-$1$ varifold of integration on $M_j$). In the case $n\leq 6$ all the $M_j$'s are completely smooth. (In the case $\Rc{N}>0$ there is only one connected component, $K=1$, see Remark \ref{oss:connectedness}.)

We point out that, denoting by $\eps_i$ the sequence extracted to guarantee the varifold convergence, $\ca{E}{\eps_i}(u_{\eps_i})\to \|V\|(N)$ in this construction, in other words the Allen--Cahn energy of $u_{\eps_i}$  converges to the mass $\sum_{j=1}^K q_j {\Hc}^n(M_j)$ of $V$. 

\subsection{$1$-dimensional profiles}
\label{one_dim_profiles}

Let $\Het(r)$ denote the monotonically increasing solution to $u''-W'(u)=0$ such that $\lim_{r\to \pm\infty} \Het(r) = \pm 1$, with $\Het(0)=0$. (For the standard potential $\frac{(1-x^2)^2}{4}$ we have $\Het(r)=\tanh\left(\frac{r}{\sqrt{2}}\right)$.) Then also $\Het(-r)$ and $\Het(\pm r+z)$ solve $u''-W'(u)=0$ (for any $z\in \R$). The rescaled function $\He(r)=\Het\left(\frac{r}{\eps}\right)$ solves ${\eps} u''-\frac{W'(u)}{\eps}=0$.

\medskip

\textit{Truncations}. The arguments developed here will involve the construction of suitable Allen--Cahn approximations of certain immersions. For that purpose, we will make use of approximate versions of $\He$. While this introcuces small errors in the corresponding ODEs, it has the advantage that the approximate solutions are constant ($\pm 1$) away from an interval of the form $[-6 \eps |\log\eps|, 6 \eps |\log\eps|]$. An Allen--Cahn approximation of a hypersurface in $N$ requires to fit the $1$-dimensional profiles in the normal direction to the hypersurface and we need to stay inside a tubular neighbourhood, so it is effective to have one-dimensional profiles that become constant before we reach the boundary of the tubular neighbourhood.

\medskip

The cutoff for the heteroclininc $\Het$ is done as follows (this truncation is also used in \cite{ChoMan} and \cite{WaWe}): for $\Lambda=3|\log\eps|$ define

$$\OHet(r) = \chi(\Lambda^{-1} r -1)\Het(r) \pm (1-\chi(\Lambda^{-1}|r|-1)),$$
where $\pm$ is chosen respectively on $r>0$, $r<0$ and $\chi$ is a smooth bump function that is $+1$ on $(-1,1)$ and has support equal to $[-2,2]$. With this definition, $\OHet=\Het$ on $(-\Lambda, \Lambda)$, $\OHet=-1$ on $(-\infty, -2\Lambda]$, $\OHet=+1$ on $[2\Lambda,\infty)$. Moreover $\OHet$ solves  $\|\OHet''-W'(\OHet)\|_{C^2(\R)} \leq C \eps^3$, for $C>0$ independent of $\eps$. (Note also that $\OHet''-W'(\OHet)=0$ away from $(-2\Lambda, -\Lambda) \cup (\Lambda,2\Lambda)$.) To see this it suffices to compute

$$\OHet''(r)=\Lambda^{-2} \chi''(\Lambda^{-1} r -1)\Het(r) + \Lambda^{-1}\chi'(\Lambda^{-1} r -1)\Het'(r) - \Lambda^{-2} \chi''(\Lambda^{-1} r -1) + \chi(\Lambda^{-1} r -1) \Het''(r)$$
and observe that the first and third term give $\Lambda^{-2} \chi''(\Lambda^{-1} r -1)(\Het(r)-1)$ which is non-zero only on $(-2\Lambda, -\Lambda) \cup (\Lambda,2\Lambda)$. On this set we have $\|\Het(r)-1\|_{C^2}\leq c \eps^\alpha$ for some $\alpha\geq 6$ (because $\Lambda=3|\log\eps|$) and $c>0$ depending only on $W$. This can be done by explicit check for the standard potential (because $1-\tanh(r)=\frac{2e^{-2x}}{1+e^{-2x}}$ and we evaluate for $r>-3\log\eps$) and is true whenever $W$ is quadratic around the minima. Therefore on $(-2\Lambda, -\Lambda) \cup (\Lambda,2\Lambda)$ we get $\|\OHet''\|_{C^2} \leq \tilde{c} \eps^3$ for $\eps<1/2$ and $\tilde{c}>0$ depending only on $W$. Moreover on $(-2\Lambda, -\Lambda) \cup (\Lambda,2\Lambda)$

$$\|W'(\OHet)\|_{C^2}\leq  \|W'(\OHet)\|_{C^2}+2\|W\|_{C^3}\|\OHet'\|_{C^1} \leq c {\eps}^{\alpha}$$
because $W$ is quadratic around $\pm 1$. In conclusion $\|\OHet''-W'(\OHet)\|_{C^2(\R)} \leq C {\eps}^3$ for some $C>0$ (depending on $W$).

\noindent \textit{Notation}: For $\eps<1$ we rescale these truncated solutions and let $\OHet^{\eps}(\cdot)=\OHet\left(\frac{\cdot}{\eps}\right)$.

\medskip

\textit{Computation of the Allen-Cahn energy of $\OHet^{\eps}$}. To compute the energy of ${\Het}^{\eps}$, following \cite{Ilm}, we have, for any $q:\R\to \R$,
$$\int_{(a,b)} \frac{\eps}{2} |q'|^2  + \frac{W(q)}{\eps}= \int_a^b \frac{1}{2} \left(\sqrt{\eps}q'-\frac{1}{\sqrt{\eps}}\sqrt{2W(q)}\right)^2 + q'\sqrt{2W(q)}.$$
The first term vanishes when $q=\He$. Let $G$ denote a primitive of $\sqrt{2W(t)}$. For the second term, noting that the integrand is $G(q)'$, we get $G(q(b))-G(q(a))$. In particular,  $\int_{\R} \frac{\eps}{2} |q'|^2  + \frac{W(q)}{\eps}= G(1)-G(-1) = 2\sigma$. 
Using the fact that $\He(-2\eps\Lambda)=-1+O(\eps^2)$, we get for $q=\He$
$$\int_{-\infty}^{-2\eps\Lambda} \frac{\eps}{2} |q'|^2  + \frac{W(q)}{\eps}  = G(-1+O({\eps}^2)) - G(-1)=O({\eps}^4)>0$$
and similarly $\int_{2\eps\Lambda}^{\infty} \frac{\eps}{2} |q'|^2  + \frac{W(q)}{\eps} =O({\eps}^4)>0$. Therefore
$$\int_{-2\eps\Lambda}^{2\eps\Lambda} \frac{\eps}{2} |q'|^2  + \frac{W(q)}{\eps}   =2 \sigma - O({\eps}^4).$$
Recalling the definition of $\OHet^{\eps}$, we have that $$\OHet^{\eps}-\He = (1-\chi(\eps\Lambda^{-1} t -1))(\OHet^{\eps} \pm 1)$$ which is controlled by $O(\eps^2)$ in $C^2$-norm. Therefore
\be
\label{eq:hetenergy}
2\sigma- O({\eps}^2)\leq \int_{-2\eps\Lambda}^{2\eps\Lambda} \frac{\eps}{2} |({\OHet}^{\eps})'|^2  + \frac{W({\OHet}^{\eps})}{\eps}  \leq 2\sigma + O({\eps}^2).
\ee

\begin{figure}[h]
\centering
\includegraphics[scale=0.4]{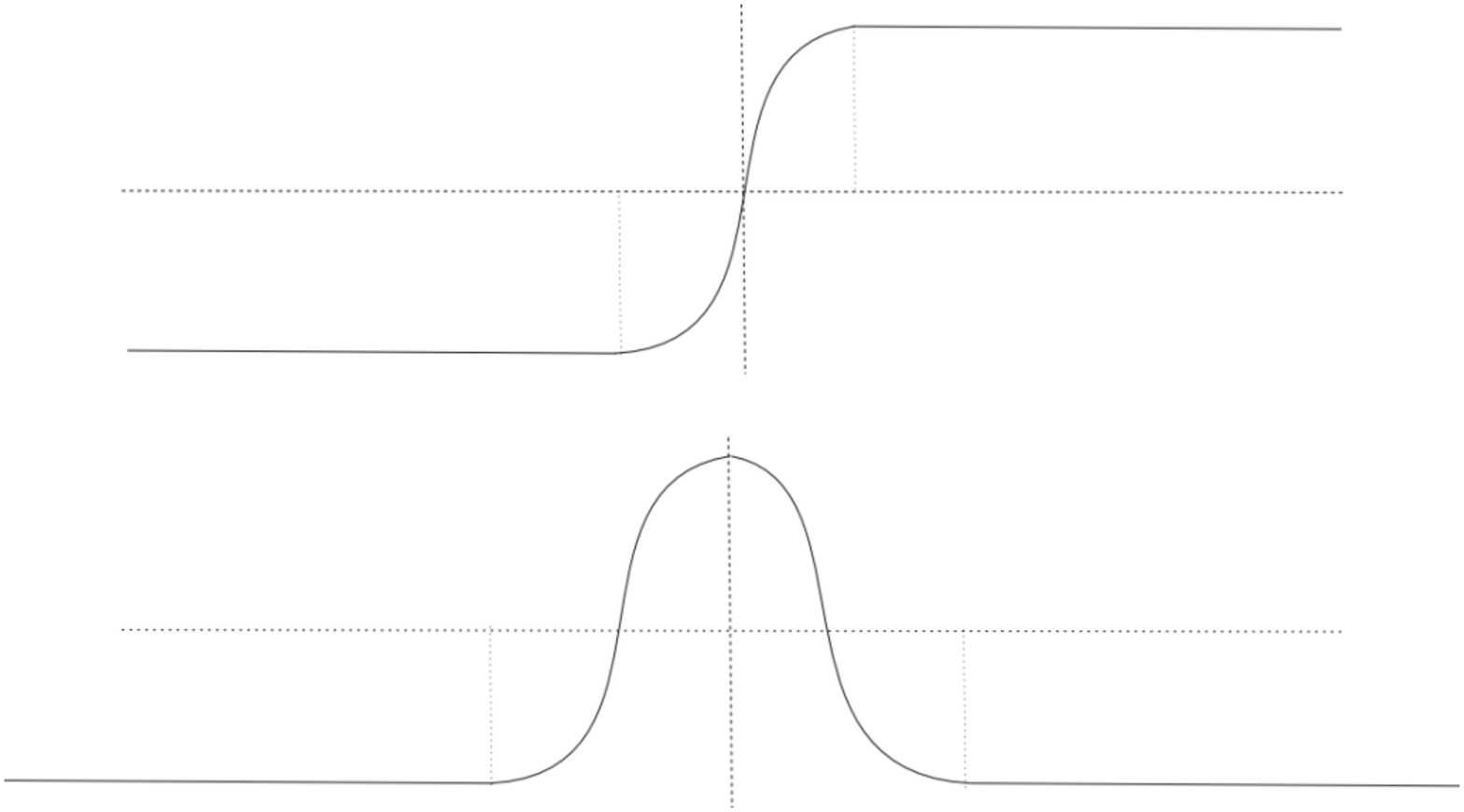}
\put (-189.5, 103){\tiny{$-2\eps\Lambda$}}
\put (-145.5, 103){\tiny{$2\eps\Lambda$}}
\put (-229.5, 103){\tiny{$-4\eps\Lambda$}}
\put (-110.5, 103){\tiny{$4\eps\Lambda$}}
\put (-160.5, 145){\tiny{$+1$}}
\put (-243.5, 69){\tiny{$\equiv -1$}}
\put (-95.5, 69){\tiny{$\equiv -1$}}
\put (-189.5, 203){\tiny{$-2\eps\Lambda$}}
\put (-135.5, 203){\tiny{$2\eps\Lambda$}}
\put (-95.5, 230){\tiny{$\equiv +1$}}
\put (-243.5, 170){\tiny{$\equiv -1$}}
\caption{The smooth functions $\OHet^{\eps}$ (top) and $\Psi=\Psi_0$ (bottom), with $\Lambda=3|\log\eps|$.}
\label{fig:profiles_1d}
\end{figure}

\medskip

\noindent\textit{Families of profiles}.
Define the function $\Psi:\R\to \R$
\be
\Psi(r)=\left\{\begin{array}{ccc}
\OHet^{\eps}(r+2\eps \Lambda) & r\leq 0\\
\OHet^{\eps}(-r+2\eps \Lambda) & r>0
        \end{array}\right. .
\ee
This function is smooth thanks to the fact that all derivatives of $\OHet^{\eps}$ vanish at $\pm 2\eps \Lambda$. Define the following evolution for $t\in [0, \infty)$:
\be
\label{eq:family2}
\Psi_t(r):= \left\{ \begin{array}{ccc}
                      \OHet^{\eps}(r+2\eps \Lambda-t) & r\leq 0 \\
                      \OHet^{\eps}(-r+2\eps \Lambda-t) & r> 0 
                     \end{array}
\right. .
\ee
Note that $\Psi_0=\Psi$ and $\Psi_{t}\equiv -1$ for $t\geq 4\eps\Lambda$. For $t\in(0,4\eps\Lambda)$ the function $\Psi_t$ is equal to $-1$ for $|r|\geq 4\eps\Lambda-t$.
The functions $\Psi_t$ form a family of even, Lipschitz functions and $\Ece(\Psi_t)$ is decreasing in $t$. Indeed, the energy contribution of the ``tails'' ($\pm 1$) is zero (so the energy is finite) and we have $\Ece(\Psi_t)=\Ece(\Psi)-\frac{1}{2\sigma}\int_{-t}^t \eps |\Psi'|^2 + \frac{W(\Psi)}{\eps}$.

The profiles $\Psi_t$, and profiles of the type $\OHet^{\eps}(\cdot -t)$, will be used within our construction to produce Allen--Cahn approximations of relevant immersions (possibly with boundary), having multiplicity $1$ or $2$ on their image.
 
\section{Distance function to $\overline{M}$} 
\label{distance}
Let $N$ be a Riemannian manifold of dimension $n+1$ with $n\geq 2$ and with positive Ricci curvature $\Rc{N}>0$. Let $M\subset N$ be a smooth minimal hypersurface with $\text{dim}\left(\overline{M}\setminus M\right)\leq n-7$, $M$ is stationary in $N$, and $M$ is locally stable in $N$, i.e.~for every point in $\overline{M}$ there exists a geodesic ball centred at the point in which $M$ is stable.
These properties are true for the $\eps\to 0$ varifold limit of finite-index Allen--Cahn solutions on $N$, thanks to the analysis in \cite{Gua}, \cite{HT}, \cite{Ton}, \cite{TonWic}, \cite{Wic}. The stationarity condition implies the existence of tangent cones at every point in $\overline{M}$. A consequence of the deep sheeting theorem in \cite{SS}, \cite{Wic} is that any point of $\overline{M}$ at which one tangent cone is supported on a hyperplane has to be a smooth point.

\medskip

Let $\text{dist}_N$ denote the (unsigned) Riemannian distance on $N$; we will be interested in the function $\dm:N\to [0,\infty)$, $\dm(\cdot)=\text{dist}_N(\cdot, \overline{M})$. Since $N$ is complete, for every $x$ the value $\dm(x)$ is realized by at least one geodesic from $x$ to $\overline{M}$ (Hopf--Rinow). We recall a few facts that are true of the distance to an arbitrary closed set, see \cite[Section 3]{ManteMennu}\footnote{If the closed set is known to be a $C^{1,1}$ submanifold, then the existence of a tubular neighbourhood is guaranteed, in which the nearest point projection is a well-defined map; moreover, if $C^2$ regularity on the submanifold is assumed, Fermi coordinates can be used. In our case, due to the presence of the singular set $\overline{M}\setminus M$, one cannot have a tubular neighbourhood of $\overline{M}$.}. The function $\dm$ is Lipschitz on $N$ (with constant $1$) and locally semi-concave on $N\setminus \overline{M}$, so that its gradient is $BV_{\text{loc}}$ on $N\setminus \overline{M}$ (equivalently, the distributional Hessian of $\dm$ on $N\setminus \overline{M}$ is a Radon measure). We denote by $S_{\dm}$ the subset of $N\setminus \overline{M}$ where $\dm$ fails to be differentiable; $S_{\dm}$ coincides with the set of points in $p\in N\setminus \overline{M}$ for which there exist at least two geodesics from $p$ to $\overline{M}$ whose length realizes $\dm(p)$. The function $\dm$ is $C^1$ on $N\setminus (\overline{M} \cup \overline{S_{\dm}})$ and $S_{\dm}$ is countably $n$-rectifiable (this uses \cite{Alberti}). However, rectifiability is not necessarily true of its closure unless extra hypotheses on the closed set are available; for example, $\overline{S_{\dm}}$ would be countably $n$-rectifiable (even in the $C^{k-2}$ sense) if $\overline{M}$ were a $C^k$ submanifold with $k\geq 3$, thanks to \cite[Section 4]{ManteMennu}. While this statement does not apply immediately in our case due to the presence of the singular set $\overline{M}\setminus M$, the proof in \cite[Section 4]{ManteMennu} can still be carried out without any change by virtue of the following observation, which can also be found in \cite{Gro}, \cite{Zhou1}.

\begin{lem}
 \label{lem:planar_tangent}
Let $x\in N\setminus  \overline{M}$. For any geodesic from $x$ to $\overline{M}$ (whose length realizes $\dm (x)$), we have that the endpoint $y$ on $\overline{M}$ actually belongs to $M$.
\end{lem}

\begin{proof}
Let $\gamma$ be any geodesic from $x$ to $\overline{M}$, let $y$ be its endpoint on $\overline{M}$ and fix a point $z\in N$ that lies in the image of $\gamma$ and such that $\text{dist}_N(z,y)<\text{inj}(N)$. Consider the (open) geodesic ball $B(z)\subset N$ centred at $z$ with radius $\text{dist}_N(z,y)$. Then $\overline{M} \cap B(z)=\emptyset$ (otherwise there would be a shorter curve than $\gamma$ joining $x$ to $\overline{M}$) and $y\in  \overline{M} \cap \p B(z)$. Since the monotonicity formula holds at all points of $\overline{M}$ ($M$ is stationary in $N$), we can blow up at $y$ to obtain tangent cones. Then every tangent cone to $M$ at $y$ has to be supported in a (closed) half space (the complement of the open half space obtained by blowing up $B(z)$ at $y$). By \cite[Ch. 7 Theorem 4.5, Remark 4.6]{Simon} every tangent cone to $M$ at $y$ is the hyperplane tangent to $B(z)$ at $y$, possibly with multiplicity. As pointed out above, the sheeting theorem in \cite{SS}, \cite{Wic} implies that $y$ is a smooth point.
\end{proof}

In other words, any geodesic that realizes the distance to $\overline{M}$ has to end at a smooth point, i.e.~on $M$ (and it meets $M$ orthogonally). This is the key fact that allows to repeat the arguments in \cite[Section 4]{ManteMennu} (as we briefly sketch below) and obtain Proposition \ref{Prop:ManteMennu}. In the rest of this work we will be interested in the set $T_{\om}=\{x\in N: \text{dist}_N(x, \overline{M})<\om\}$, where $\om$ is chosen in $(0,\text{inj}(N))$, therefore we restrict to this open set for our analysis (even though not strictly necessary for this section).

\begin{Prop}[as in \cite{ManteMennu}]
\label{Prop:ManteMennu}
The set $\overline{S_{\dm}}\cap T_\om$ is countably $n$-rectifiable\footnote{Even $C^k$ countably $n$-rectifiable for all $k$, however we will not need this stronger property.}. Moreover, $\nabla \dm \in SBV_{\text{loc}}(T_\om \setminus \overline{M})$ and the singular part (with respect to $\mathcal{H}^{n+1}\res (T_\om \setminus \overline{M})$) of the Radon measure $D^2 \dm \res (T_\om \setminus \overline{M})$ is supported on $\overline{S_{\dm}}\cap (T_\om \setminus \overline{M})$.
\end{Prop}

\begin{oss}
Additionally we have, since $M$ is smooth, that the absolutely continuous part of $D^2 \dm$ has a smooth density with respect to $\mathcal{H}^{n+1}\res \left(T_{\om}\setminus (\overline{M} \cup \overline{S_{\dm}})\right)$. This density coincides with the pointwise Hessian of $\dm$. 
\end{oss}

The rectifiability of $\overline{S_{\dm}}$ implies that at $\mathcal{H}^n$-a.e.~point $x\in \overline{S_{\dm}}$ there exists a measure-theoretic unit normal $\hat{n}_x$ to $\overline{S_{\dm}}$ (rather, two choices of it).
The rectifiability of $\overline{S_{\dm}}$ and the smoothness of $\dm$ on $T_{\om}\setminus (\overline{M} \cup \overline{S_{\dm}})$ imply that the ``Cantor part'' of the $BV_{\text{loc}}$ function $D^2 \dm$ is $0$ (see the sketch below), leaving only the ``jump part'', which is supported on $\overline{S_{\dm}}$. This means that $\nabla \dm$ is $SBV_{\text{loc}}(T_\om \setminus \overline{M})$; moreover, for $\mathcal{H}^n$-a.e.~point $x\in \overline{S_{\dm}}\cap T_{\om}$, the left and right limits in the Lebesgue sense of $\nabla \dm$ at $x$ are well-defined in the two halfspaces identified by the normal (see \cite[Theorem 3.77]{AFP}). 

\medskip

\textit{Sketch: relevant arguments in \cite{ManteMennu}.} Consider the map $F(y,v,t) = \text{exp}_y(tv)$ for $y\in M$ and $v$ a unit vector orthogonal to $M$ at $y$. For fixed $(y,v)$ the curve $F(y,v,t)$ is a geodesic leaving $M$ orthogonally. We will limit ourselves to $t\leq \om$, since we are only interested in $T_\om$. If $t_0$ is sufficiently small (depending on $(y,v)$) the geodesic $t\in [0,t_0]\to F(y,v,t)$ is the minimizing curve between its endpoint $F(y,v,t_0)$ and $\overline{M}$, equivalently, its length is $\dm\left(F(y,v,t_0)\right)$. However for large enough $t_0$ the geodesic may fail to be minimizing, therefore one can consider $\sigma=\sigma_{y,v}\in(0,\om]$ defined as follows: $\sigma=\om$ if $F(y,v,t)$ is minimizing (between its endpoint and $\overline{M}$) for all $t\leq \om$; otherwise, $\sigma$ is chosen in $(0,\om)$ so that $F(y,v,t)$ is minimizing (between its endpoint and $\overline{M}$) for $t\leq \sigma$ and $F(y,v,t)$ is not minimizing if $t>\sigma$. The set of points 
$$\text{Cut}(M)=\{F(y,v,\sigma_{(y,v)}): y\in M, v\in (T_y M)^\perp, |v|=1, \sigma_{(y,v)}<\om\}$$
is the restriction to $T_\om$ of the so-called cut-locus of $M$, and it is a subset of $T_\om \setminus \overline{M}$ whose closure in $T_\om$ does not intersect $M$. 
Recall that the unit sphere bundle of $M$ is just $\tilde{M}$, the oriented double cover of $M$, so we will also write $(y,v)\in \tilde{M}$.

Standard theory of geodesics (e.g.~\cite[Ch.~2, Lemma 4.8 and Ch.~3, Lemma 2.11]{Sakai}, which give the analogue of \cite[Prop. 4.7]{ManteMennu} for $M$) gives that if $x\in \text{Cut}(M)$ then at least one of the following two conditions holds: (a) there exist (at least) two distinct geodesics from $x$ to $M$ that realize $\dm(x)$; (b) the map $F:\tilde{M}\times (0,\om)\to T_\om$ has non-invertible differential at $(y,v,\sigma_{(y,v)})$, where $x=F(y,v,\sigma_{(y,v)})$. Conversely, if (a) or (b) holds, then the geodesic $t\to F(y,v,t)$ cannot be minimal on $t\in[0,t_0]$ when $t_0\in(\sigma_{(y,v)},\om)$. Option (a) is equivalent to $x\in S_{\dm} \cap (T_\om \setminus \overline{M})$ (see \cite[Prop. 3.7]{ManteMennu}).

Using these facts, the arguments of \cite[Proposition 4.8]{ManteMennu} adapt to give that 
$$\overline{S_{\dm}} \cap (T_\om\setminus \overline{M})=\text{Cut}(M),$$
therefore in order to prove the rectifiability in Proposition \ref{Prop:ManteMennu} it suffices (since $S_{\dm}$ is countably $n$-rectifiable, see above) to show that 
$\text{Cut}(M) \setminus S_{\dm}$ is a countably $n$-rectifiable set in $T_\om\setminus \overline{M}$, i.e.~the analogue of \cite[Theorem 4.11]{ManteMennu}. Note that $\overline{S_{\dm}}\cap \overline{M}\subset \overline{M}\setminus M$ is $\mathcal{H}^n$-negligeable, so it does not affect rectifiability. The points in $\text{Cut}(M) \setminus S_{\dm}$ are characterised by the validity of option (b) above, and the arguments in \cite{ManteMennu} are local around the points $(y,v,\sigma_{y,v})\in \tilde{M} \times (0,\om)$, so they apply verbatim to our case.

Once the countable $n$-rectifiability of $\overline{S_{\dm}}$ has been obtained, it follows that $\nabla \dm$ is $SBV_{\text{loc}}(T_\om \setminus \overline{M})$. Indeed, we know to begin with (see above for these statements about the distance to a closed set) that $\nabla \dm$ is in $BV_{\text{loc}}(T_\om \setminus \overline{M})$ and notice that $\dm$ is $C^2$ (even $C^k$ for all $k$) on $T_\om \setminus (\overline{S_{\dm}} \cup \overline{M})$ thanks to the smoothness of $M$. The ``Cantor part'' of the Radon measure $D^2 \dm$ gives $0$ measure to countably $n$-rectifiable sets (see \cite[Prop. 3.92 or Prop. 4.2]{AFP}), in particular it gives $0$ to $\overline{S_{\dm}}$.
The smoothness of $D^2 \dm$ in $T_\om \setminus (\overline{S_{\dm}} \cup \overline{M})$ then implies that there is no ``Cantor part'', i.e.~$\nabla \dm$ is $SBV_{\text{loc}}(T_\om \setminus \overline{M})$. This concludes the sketch of proof of Proposition \ref{Prop:ManteMennu}.
 
 \medskip
 
\begin{oss}[\textit{on the diffeomorphism $F$}]
\label{oss:retract}
We point out a couple of further facts, mainly adapted from \cite{ManteMennu}, for future reference. The level sets of $\dm$ are smooth in the open set $T_\om \setminus (\overline{S_{\dm}} \cup \overline{M})$, thanks to the implicit function theorem, the smoothness of $\dm$ and the invertibility of $F$ on this open set.

The map $F(y,v,t) = \text{exp}_y(tv)$ for $y\in M$ and $v$ a unit vector orthogonal to $M$ at $y$, $t\in(0,\om)$ is a map from $\tilde{M} \times (0,\om)$ into $T_\om$ (since the oriented double cover $\tilde{M}$ of $M$ is defined as the set of $(y,v)$ with $y\in M$, $v$ unit vector normal to $M$ at $y$). Arguing as in \cite[Prop. 4.8]{ManteMennu} we see that the following restriction of $F$ (still denoted by $F$)
\be
\label{eq:diffeo_F}
F:\{((y,v),s):(y,v)\in \tilde{M}, s\in(0,\sigma_{(y,v)})\} \to T_\om\setminus \text{Cut}(M) \setminus \overline{M} 
\ee
is a (smooth) diffeomorphism.\footnote{This diffeomorphism shows, in particular, the following. If $x\in T_\om \setminus (\overline{S_{\dm}} \cup \overline{M})$, then we know that there exists a unique geodesic $\gamma$ from $x$ to $\overline{M}$ and its endpoint $y$ is on $M$ by Lemma \ref{lem:planar_tangent}. Then, by the properties (a), (b) discussed above, no point of $\gamma$ is in $\text{Cut}(M)$ and therefore all points on $\gamma$ except $y$ have the property that they lie in $T_\om \setminus (\overline{S_{\dm}} \cup \overline{M})$.}
This diffeomorphism extends as a continuous map to $\tilde{M}\times \{0\}$ by sending $((y,v),0)$ to $y\in M$ (note that it is a $2-1$ map here, the standard projection from $\tilde{M}$ to $M$). The image of this continuous map is then $T_\om\setminus \text{Cut}(M)\setminus (\overline{M}\setminus M)$. 
Again following verbatim \cite[Prop. 4.8]{ManteMennu}, we also have that the function $\sigma_{(y,v)}$ is continuous on $\tilde{M}$. The diffeomorpshism $F$ in (\ref{eq:diffeo_F}), continuously extended to $\tilde{M}$, provides the natural replacements for Fermi coordinates around $M$ in our situation, where the singular set $\overline{M}\setminus M$ is present. We will write 
$$V_{\tilde{M}}:=\{((y,v),s):(y,v)\in \tilde{M}, s\in[0,\sigma_{(y,v)})\},$$
for the domain of (the extension of) $F$.

Let us take a closer look at the level sets $\Gamma_t=\{x\in T_\om:\dm(x)=t\}$, for $t\in (0,\om)$. By the previous discussion, the smooth hypersurface $\Gamma_t \setminus \overline{S_{\dm}}$ can be restracted smoothly, staying in $T_\om \setminus \overline{S_{\dm}}$ onto a subset of $M$ and at each time the image of the retraction is contained in (a smooth portion of) a level set of $\dm$. In fact, we have a retraction 
$$\left(T_\om\setminus \text{Cut}(M)\setminus (\overline{M}\setminus M)\right) \times [0,1] \to T_\om\setminus \text{Cut}(M)\setminus (\overline{M}\setminus M)$$
explicitly given, using the identification (\ref{eq:diffeo_F}), by (here $q=(y,v)\in\tilde{M}$ and $\sigma_q=\sigma_{(y,v)}$)
$$\begin{array}{ccc}
  R:\{(q,s):q\in \tilde{M}, s\in[0,\sigma_{q})\} \times [0,1] \to \{(q,s):q\in \tilde{M}, s\in[0,\sigma_{q})\}\\
 R(q,s,\alpha)= q+(1-\alpha)s. 
 \end{array}$$
Under the identification (\ref{eq:diffeo_F}), the function $s$ is just $\dm$, so it follows that the retraction preserves level sets of $\dm$.
\end{oss}

 \medskip

We will now analyse the jump part of the Hessian of $\dm:T_\om\setminus \overline{M}\to (0,\infty)$; this will lead to Lemma \ref{lem:jump_part} below. To this end, we perform, for $\mathcal{H}^n$-a.e.~point $x\in \overline{S_{\dm}}$, a blow up of $\dm$ as follows. Using normal coordinates around $x$, for all sufficiently small $\rho>0$ consider the function $d_\rho:B^{n+1}_1(0)\to (0,\infty)$ defined by 
$$d_\rho(y) = \frac{\dm(x+\rho y) - \dm(x)}{\rho}.$$
Then $(\nabla d_\rho)(y) = (\nabla \dm) (x+\rho y)$. Note that $d_\rho$ have Lipschitz-constant $1$ and $d_\rho(0)=0$, therefore we can extract a sequence $\rho_j\to 0$ such that $d_{\rho_j}$ converge in $C^{0,\alpha}$ (for all $\alpha<1$) to a $1$-Lipschitz function $d_x:B^{n+1}_1(0)\to \R$ with $d_x(0)=0$. Recall that for $\mathcal{H}^n$-a.e.~point $x\in \overline{S_{\dm}}$ we have a normal $\hat{n}_x$ and moreover $\nabla \dm$ admits one-sided limits in $L^1$: there exists two constant vectors $a\neq b$ in $\R^{n+1}$ such that 
$$\frac{1}{\rho^{n+1}}\int_{\{z\in B_{\rho}(x): z\cdot \hat{n}_x <0\}} |\nabla \dm -a| \to 0\,\, \text{ and } \,\,\frac{1}{\rho^{n+1}}\int_{\{z\in B_{\rho}(x): z\cdot \hat{n}_x >0\}} |\nabla \dm -b| \to 0$$ as $\rho\to 0$. This is equivalent, by a change of variables, to 

$$\int_{\{z\in B_{1}(0): z\cdot \hat{n}_x <0\}} |\nabla d_{\rho}-a| \to 0,\,\, \int_{\{z\in B_{1}(0): z\cdot \hat{n}_x >0\}} |\nabla d_{\rho}-b| \to 0$$
as $\rho\to 0$. Therefore $\nabla d_{\rho_j}$ converge in $L^1(B_1^{n+1}(0))$ to the function $F_{ab}$ defined to be constant on each of the two half-balls $\{z\in B_1^{n+1}(0): z\cdot \hat{n}_x <0\}$ and $\{z\in B_1^{n+1}(0): z\cdot \hat{n}_x >0\}$, with respective values $a$ and $b$. This function must be the (distributional) gradient of $d_x$. Indeed, for every $v\in C^1_c(B_1(0))$ we have $\int_{B_1(0)} d_x \nabla v = \lim_{j\to \infty}\int_{B_1(0)} d_{\rho_j} \nabla v = - \lim_{j\to \infty} \int_{B_1(0)} \nabla d_{\rho_j} v = - \int_{B_1(0)} F_{ab} v $, where we used, in the two limits, respectively the uniform convergence $d_{\rho_j} \to d_x$ and the $L^1$-convergence $\nabla d_{\rho_j} \to F_{ab}$. The equality obtained expresses the fact that $F_{ab}=\nabla d_x$ and proves that 
$$d_{\rho_j} \to d_x \text{ in } W^{1,1}(B_1(0)) \text{ and in } C^{0,\alpha}(B_1(0)).$$
Recall now that $\dm$ is locally semiconcave, so it has at least an element in the superdifferential, i.e.~there exists a $C^1$ function $\varpi$ in a neighbourhood of $x$ that is $\geq \dm$ and such that $\varpi(x)=\dm(x)$. Performing the same blow up on $\varpi$, we consider the rescalings $\frac{\varpi(x+\rho y) - \varpi(x)}{\rho}$. These functions converge in $C^1(B_1(0))$ to an affine function $\varpi_x$. By uniform convergence, $\varpi_x \geq d_x$ on $B_1(0)$ and $\varpi_x(0)=d_x(0)=0$. 
Recalling that $\nabla d_x = F_{ab}$, we obtain
\be
\label{eq:left_right}
(a - b)\cdot \hat{n}_x \geq 0.
\ee
The jump part of $D(\nabla \dm)$ is characterized as the measure that is absolutely continuous with respect to $\mathcal{H}^n \res \overline{S_{\dm}}$ and with density that is given for $\mathcal{H}^n$-a.e.~$x\in \overline{S_{\dm}}$ by $(b-a) \otimes \hat{n}_x$ (see e.g.~\cite[(3.90)]{AFP}). Taking the trace and using (\ref{eq:left_right}) this implies:

\begin{lem}
\label{lem:jump_part}
Let $\Delta$ denote the Laplace-Beltrami operator on $T_\om \setminus \overline{M}$. The singular (jump) part of $\Delta \dm$ in $T_{\om}\setminus \overline{M}$ is a negative measure (supported on $\overline{S_{\dm}}$).
\end{lem}

\begin{oss}
Sufficient regularity of $\overline{S_{\dm}}$ would actually imply that this jump part is $0$ close to $\overline{M}$, because of the smallness of $\overline{M}\setminus M$ (and the fact that $\overline{S_{\dm}}\cap M=\emptyset$). E.g.~if $\overline{S_{\dm}}$ were a smooth hypersurface, then locally around a point in $\overline{M}\setminus M$ it would separate $M$ into two stationary hypersurfaces (one on each side of $\overline{S_{\dm}}$) whose closures intersect only at $\overline{M}\setminus M$, against the varifold maximum principle \cite{Ilm2}, \cite{Wic}. Lemma \ref{lem:jump_part} is sufficient for our scopes. 
\end{oss}

Next we analyse the absolutely continuous part (with respect to $\mathcal{H}^{n+1}$) of $\Delta \dm$, for $\dm:T_\om\setminus \overline{M}\to (0,\infty)$. By Proposition \ref{Prop:ManteMennu} it suffices to analyse the smooth function $\Delta \dm$ on $T_\om \setminus (\overline{S_{\dm}} \cup \overline{M})$. For this, we will need the Ricci positive condition (which has not been used so far).

\begin{lem}
\label{lem:negative_Laplacian_1}
The function $\dm$ satisfies $\Delta \dm \leq 0$ on $T_\om\setminus (\overline{S_{\dm}}\cup \overline{M})$.
\end{lem}

\begin{proof}
Recall Remark \ref{oss:retract}.
For $x \in T_\om \setminus (\overline{S_{\dm}}\cup \overline{M})$, $\dm(x)$ is realized by the length of a unique geodesic from $x$ to $\pi(x)\in M$ and the level set $\{y\in N\setminus (\overline{S_{\dm}}\cup \overline{M}):\dm(y)=\dm(x)\}$ passing through $x$ is $C^2$ and its scalar mean curvature at $x$ (with respect to the normal that points away from $M$) is $-\Delta \dm(x)$. We are thus in the classical situation in which we look at level sets of the distance function to a smooth submanifold, in this case a geodesic ball $B_r(\pi(x))$ in $M$. This gives the information on the Laplacian in a neighbourhood of $x$. By Riccati's equation \cite[Corollary 3.6]{Gray} we get that in positive Ricci the mean curvature of the level sets $\{y\in N\setminus (\overline{S_{\dm}}\cup \overline{M}):y=\text{exp}_z(t \nu), z\in B_r(\pi(x))\}$ (this is a disk at distance $t$ from $B_r(\pi(x))$), for either of the choices of unit normal $\nu$ on $B_r(\pi(x))$, increases in $t$, hence the Laplacian $\Delta \dm$ is negative on $N\setminus (\overline{S_{\dm}}\cup \overline{M})$.
\end{proof}

In conclusion, from Lemmas \ref{lem:jump_part} and \ref{lem:negative_Laplacian_1}, we have $\Delta \dm \leq 0$ on $T_\om \setminus \overline{M}$ in the sense of distributions. Since $\dm$ is smooth at points in $M$ and $M$ is minimal, we also get that $\Delta \dm$ is $0$ on $M$ and so we can extend the previous conclusion: $\Delta \dm \leq 0$ on $T_\om \setminus (\overline{M}\setminus M)$. We will now extend across $\overline{M}\setminus M$ by a capacity argument.

\begin{Prop}
\label{Prop:negative_Laplacian_2}
Let $N$ be a closed $(n+1)$-dimensional Riemannian manifold with positive Ricci curvature and $M$ a smooth minimal hypersurface as in Theorem \ref{thm:compare2M_Ricci7}. Denote by $\dm$ the distance function to $\overline{M}$ and by $T_\om=\{x\in N:  \dm(x)<\om\}$, where $\om<\text{inj}(N)$. Then $\Delta \dm \leq 0$ on $T_\om$ in the sense of distributions\footnote{A distribution is said to be $\leq 0$ if for every non-negative test function the result is $\leq 0$. A distribution that is $\geq 0$ or $\leq 0$ is necessarily a Radon measure, see e.g.~\cite[Theorem 1.39]{EvGa}.}.
\end{Prop}

\begin{proof}
Let $\delta>0$ be arbitrary and choose $\chi \in C^\infty_c(T_\om)$ to be a function that takes values in $[0,1]$, is identically $1$ in an open neighbourhood of $\overline{M}\setminus M$, identically $0$ away from a (larger) neighbourhood of $\overline{M}\setminus M$ and such that $\int_{T_{\om}} |\nabla \chi|<\delta$ (see \cite[4.7]{EvGa}). Then we have, for $v\in C^\infty_c(T_\om)$,

$$\int_{T_{\om}} (\Delta \dm) v = \int_{T_{\om}} \Delta \dm (1-\chi)v + \int_{T_{\om}} \Delta \dm \, \chi v =$$
\be
\label{eq:capacity1}
= \int_{T_{\om}} \Delta \dm (1-\chi)v - \int_{T_{\om}} \nabla \dm \, \nabla \chi\, v - \int_{T_{\om}} \nabla \dm \nabla v \, \chi.
\ee
For the second term recall that the distribution $\nabla \dm$ is an $L^\infty$ function with $|\nabla \dm|=1$ a.e. and so $|\int_{T_{\om}} \nabla \dm \nabla \chi v|\leq \|v\|_{L^\infty} \left( \int_{T_{\om}}  |\nabla \chi|\right) < \delta \|v\|_{L^\infty}$. This tends to $0$ as $\delta \to 0$.
As $\delta \to 0$, the corresponding $\chi$ will go to $0$ in $L^1$ so the third term will also tend to $0$. 

The distribution $\Delta \dm$ is a priori of order $\leq 1$: 
$|\int_{T_{\om}} (\Delta \dm) v| = |\int_{T_{\om}} \nabla \dm \nabla v|\leq \mathcal{H}^{n+1}(N)\|v\|_{C^1}$.
For the first term in the right-most side of (\ref{eq:capacity1}), observe that $(1-\chi)v \in C^\infty_c(T_\om\setminus (\overline{M}-M))$ and $\Delta \dm$ is a negative Radon measure on this open set by Lemma \ref{lem:negative_Laplacian_1}, so that $\int_{T_{\om}} \Delta \dm (1-\chi)v\leq 0$ if $v\geq 0$ (because $(1-\chi)v\geq 0$ by the choice of $\chi$). We argue as follows: (\ref{eq:capacity1}) holds for all $\delta$ and its last two terms tend to $0$ as $\delta \to 0$, therefore for every $v\in C^\infty_c(T_\om)$ and $v\geq 0$ we have

$$\int_{T_{\om}} (\Delta \dm) v = \lim_{\delta\to 0} \int_{T_{\om}} \Delta \dm (1-\chi)v \leq 0.$$
The distribution $\Delta \dm$ is therefore a negative Radon measure on $T_\om$.

\end{proof}

\section{Level sets of $\dm$}
\label{level_sets}

We consider the level sets $\Gamma_t=\{x:\dm(x)=t\}$, for $t\in[0,\om/2]$ (we fixed an arbitrary $\om\in(0,\text{inj}(N))$); we will obtain that the areas of $\Gamma_t$ are ``essentially'' decreasing in $t$. Further, we will consider an ``Allen--Cahn approximation'' $G^{\eps}_0:N\to \R$ of $\Gamma_{6\eps|\log\eps|}=\Gamma_{2\eps\Lambda}$ defined, for $\eps$ sufficiently small (to ensure $4\eps\Lambda<\om/2$), as follows: 

\be
\label{eq:Gt}
G^{\eps}_0(x) = \left\{\begin{array}{ccc}
                          -1 & \text{ for } x\in N\setminus T_\om \\
\OHet^{\eps}(-\dm(x)+2\eps\Lambda)& \text{ for } x\in T_\om
                         \end{array}\right. .
                         \ee
Since $\OHet^{\eps}$ is constantly $-1$ on $(-\infty, -2\eps\Lambda]$,
the function $G^{\eps}_0$ is constantly $-1$ on $\{x:\dm(x)> 4\eps\Lambda\}$. Moreover, since $\OHet^{\eps}$ is smooth, $G_0^{\eps}$ has the same regularity of $\dm$, i.e.~it is locally Lipschitz, $G^{\eps}_0\in W^{1,\infty}(N)$. Moreover, its gradient (which equals $-(\OHet^{\eps})^\prime(-\dm(x)+2\eps\Lambda)\nabla \dm(x)$ in $T_\om$ and $0$ otherwise) is in $BV(N)$ and its distributional Laplacian $\Delta G^{\eps}_0$ is a Radon measure (see Section \ref{distance}).
Note that the profile of $G^{\eps}_0$ in the normal direction at any point of $M$ is given by the function $\Psi=\Psi_0$ in (\ref{eq:family2}), therefore $G^{\eps}_0$ can also be thought of as an Allen--Cahn approximation of $2|M|$, or equivalently of the immersion $\iota:\tilde{M}\to N$ that covers $M$ twice. The fact that $\Ece(G_0^{\eps})$ is approximately $2|M|$ will be etablished later.

The Allen--Cahn first variation of $G^{\eps}_0$ (which is clearly $0$ outside $T_\om$) can be computed in $T_\om$ as follows:
\be
\label{eq:first_var_Gt}
-(2\s)\ca{E'}{\eps}(G^{\eps}_0)=\eps\Delta G^{\eps}_0 -\frac{W'(G^{\eps}_0)}{\eps} = 
\ee
$$=\eps {\OHet^{\eps}}''(-\dm+2\eps\Lambda) |\nabla \dm|^2 - \eps{\OHet^{\eps}}'(-\dm+2\eps\Lambda) \Delta \dm -\frac{W'(\OHet^{\eps}(-\dm+2\eps\Lambda))}{\eps}=$$
$$=\underbrace{\eps {\OHet^{\eps}}''(-\dm+2\eps\Lambda)-\frac{W'(\OHet^{\eps}(-\dm+2\eps\Lambda))}{\eps}}_{O({\eps}^2)}- \underbrace{\eps{\OHet^{\eps}}'(-\dm+2\eps\Lambda)}_{0\leq\, \cdot\,\leq 3} \underbrace{\Delta \dm}_{\leq 0},$$
in the distributional sense. Since $\Delta \dm$ a Radon measure thanks to Proposition \ref{Prop:negative_Laplacian_2}, we will think of $-\ca{E'}{\eps}(G^{\eps}_0)$ as a Radon measure. (The term $O({\eps}^2)$ in the last line is a Lipschitz function that we interpret as a density with respect to $\mathcal{H}^{n+1}$; the last term is the measure $\Delta \dm$ multiplied by a bounded Lipschitz function.)

\medskip

Denote by $\ca{F}{\eps,\mu}$, for a constant $\mu>0$, the functional on $W^{1,2}(N)$ given by 
$$\ca{F}{\eps,\mu}(u)=\Ece(u)-\frac{\mu}{2\s}\int_N u.$$  
The computation in (\ref{eq:first_var_Gt}) shows that for every $\eps$ there exists $\mu_{\eps}>0$, $\mu_{\eps}\to 0$ as $\eps \to 0$, such that\footnote{A precise choice of $\mu_{\eps}$ will be made in (\ref{eq:choice_of_mu}).} (we need $\mu_{\eps}>2\|O({\eps}^2)\|_{L^\infty}$ where $O({\eps}^2)$ is the first term in the last line of (\ref{eq:first_var_Gt}))
$$-\ca{F'}{\eps,\mu_{\eps}}(G^{\eps}_0)=-\ca{E'}{\eps}(G^{\eps}_0) + \frac{\mu_{\eps}}{2\s}\geq \frac{1}{2}\frac{\mu_{\eps}}{2\s}\mathcal{H}^{n+1}.$$
(The inequality means that the Radon measure on the left minus the Radon measure on the right is a non-negative measure.)
The function $G^{\eps}_0$ will form the starting point for the construction of a barrier for the negative $\ca{F}{\eps,\mu_{\eps}}$-gradient flow in Section \ref{reach_1}.

\medskip

\textit{Areas of $\Gamma_t$.} Since $\overline{S_{\dm}}$ is countably $n$-rectifiable (and thus has Hausdorff dimension $\leq n$ and vanishing $\mathcal{H}^{n+1}$ measure) we get that, for a.e. $t>0$, $\mathcal{H}^n(\overline{S_{\dm}}\cap \Gamma_t)=0$. We will denote by $\Om\subset (0,\om)$ the set with $\mathcal{H}^1(\Om)=0$ such that 
$$t\in (0,\om) \setminus \Om \Rightarrow \mathcal{H}^n(\overline{S_{\dm}}\cap \Gamma_t)=0$$
(and therefore, for $t\notin \Om$, $\Gamma_t$ is a smooth hypersurface away from a $\mathcal{H}^n$-negligeable set). Therefore for $t\in (0,\om) \setminus \Om$ we have $\mathcal{H}^n(\Gamma_t)=\mathcal{H}^n(\Gamma_t \setminus \overline{S_{\dm}})$, i.e.~we only need to compute the area of the smooth part of $\Gamma_t$. Thanks to this, we will compare the area of $\Gamma_t$ to that of $M$ for $t\in (0,\om) \setminus \Om$.

\begin{lem}
\label{lem:Gamma_t}
Let $\Gamma_t=\{x\in N:\dm(x)=t\}$ and $\Om\subset (0,\om)$ as above ($\mathcal{H}^1(\Om)=0$). Then

\noindent (a) for $t\in (0,\om) \setminus \Om$ the set $\Gamma_t$ is a smooth hypersurface away from a set of vanishing $\mathcal{H}^n$-measure and $\mathcal{H}^n(\Gamma_t)< 2 \mathcal{H}^n(M)$;

\noindent (b) the function $t\in (0, \om) \to \mathcal{H}^n(\Gamma_t)$ satisfies for $t_1<t_2$, $t_2 \notin \Om$ ($t_1\in \Om$ is allowed), the inequality $\mathcal{H}^n(\Gamma_{t_2})<\mathcal{H}^n(\Gamma_{t_1})$.
\end{lem}

\begin{proof}
The first part of (a) has already been discussed above. 
Recall the diffeomorphism induced by $F$ in Remark \ref{oss:retract}. Endow $\{(q,s):q\in \tilde{M}, s\in[0,\sigma_{q})\}$ with the pull-back metric (via $F$) from $T_\om\setminus \text{Cut}(M) \setminus \overline{M}$. The metric extend continuously to $\tilde{M}\times \{0\}$ to give the natural metric on $\tilde{M}$. We will thus work in $V_{\tilde{M}}=\{(q,s):q\in \tilde{M}, s\in[0,\sigma_{q})\}$; note that $F^{-1}\left(\Gamma_{t_0}\setminus \overline{S_{\dm}}\right)=\{(x,s)\in V_{\tilde{M}}: s=t_0\}$. Denoting by $\Pi$ the map $\Pi(q,s)=(q,0)$, recall that from the structure of $V_{\tilde{M}}$ we obtain the following. For every $t<t_0$ the set $\{(x,s)\in V_{\tilde{M}}: x\in \Pi\left(F^{-1}\left(\Gamma_{t_0}\setminus \overline{S_{\dm}}\right)\right), s=t\}$ is contained in $F^{-1}(\Gamma_t\setminus \overline{S_{\dm}})$. It is then enough, for (a) and (b), to prove that, if $t_0\notin \Om$ and $t<t_0$, then $\{(x,s)\in V_{\tilde{M}}: s=t_0\}$ has area bounded by $\{(x,s)\in V_{\tilde{M}}: x \in \Pi\left(\{(x,s)\in V_{\tilde{M}}: s=t_0\}\right), s=t\}$.

Let $(x_1, \ldots, x_n, s)$ be local coordinates on $V_{\tilde{M}}$ chosen so that $\frac{\p}{\p x_1}, \ldots ,\frac{\p}{\p x_n}$ form a local frame around a point $x_0\in \tilde{M}$, that is orthonormal at $x_0\in \tilde{M}$, and $\frac{\p}{\p s}$ is the unit speed of the geodesics $\{x=const\}$. Then the Riemannian metric on $V_{\tilde{M}}$ induces an area element $\theta_{s_0}$ for the level set $\{s=s_0\}$ at the point $(x_0,s_0)$. By \cite[Theorem 3.11]{Gray} it satisfies the ODE $\frac{\p}{\p s} \log\theta_s  = -\vec{H}(x_0,s)\cdot \frac{\p}{\p s}$, where $\vec{H}_{(x_0,s)}$ is the mean curvature of the level set at distance $s$ evaluated at the point $(x_0,s)$. (Note that in \cite{Gray} $\theta_s$ denotes the volume element, but since $\frac{\p}{\p s}$ is a unit vector, the area and volume elements are the same.)
By Riccati's equation \cite[Corollary 3.6]{Gray} we find that $H(x_0,s)=\vec{H}_{(x_0,s)}\cdot \frac{\p}{\p s}$ is strictly increasing in $s$, at least at linear rate, thanks to the positiveness of the Ricci curvature, $H_{(x_0,s)}\geq s (\min_N \Rc{N})$. Therefore $\frac{\p}{\p s} \log\theta_s \leq - s (\min_N \Rc{N})$ and we find for $s_0\geq 0$, $t\geq 0$
$$\log\left(\frac{\theta(s_0+t)}{\theta(s_0)}\right) \leq - (\min_N \Rc{N})\int_{s_0}^{s_0+t} s \,ds$$
and therefore
$$\theta_{s_0+t}\leq \theta_{s_0}e^{-\frac{\min_N \Rc{N}}{2}(2s_0 t + t^2)}, \,\,  \text{ for } (x_0,s_0+t)\in V_{\tilde{M}}.$$ 
In particular, $\theta(t)$ is decreasing in $t$. From this (a) and (b) follow by integrating the area element. (Recall that $\int_{\tilde{M}}\theta_0 dx^1 \ldots dx^n = 2\mathcal{H}^n(M)$.) 
\end{proof}

\textit{Allen--Cahn energy of $G_0^{\eps}$.} Thanks to Lemma \ref{lem:Gamma_t} we can control the Allen--Cahn energy of $G_0^{\eps}$ by twice the area of $M$. Indeed, recalling that the energy is $0$ in the complement of $T_{\om/2}$ and that $\nabla G_0^{\eps}$ is parallel to $\nabla \dm$, we use the coarea formula for the slicing function $\dm$ (for which $|\nabla \dm|=1$) and we get

$$\int_{T_{\om/2}} \eps \frac{|\nabla G_0^{\eps}|^2}{2} + \frac{W(G_0^{\eps})}{\eps} = \int_{0}^{\om} \left(\int_{\Gamma_s} \frac{|\nabla G_0^{\eps}|^2}{2} + \frac{W(G_0^{\eps})}{\eps}\right) ds \underbrace{=}_{\text{ by }(\ref{eq:Gt})} $$ 
$$ =\int_{-\om/2+2\eps\Lambda}^{2\eps\Lambda} \left(\int_{\Gamma_{2\eps\Lambda-s}} \eps \frac{({\OHet^{\eps}}'(s))^2}{2} + \frac{W({\OHet^{\eps}}(s))}{\eps}\right) ds \underbrace{\leq}_{\text{Lemma }\ref{lem:Gamma_t}} $$ 
$$\leq   2 \mathcal{H}^n(M) \left(\int_{\R}  \eps \frac{({\OHet^{\eps}}')^2}{2} + \frac{W({\OHet^{\eps}})}{\eps}\right) ,$$
where we used Lemma \ref{lem:Gamma_t} (a)  for a.e.~$s$,
namely $s\notin \Omega$.
By the estimates in (\ref{eq:hetenergy}) we get
\be
\label{eq:G_t_energy}
\Ece(G_0^{\eps})\leq 2 \mathcal{H}^n(M) \, (1+O({\eps}^2)).
\ee

\section{Instability properties of $M$ (choice of $B$)}
\label{instability}

Let $\iota:\tilde{M}\to N$ be the standard projection ($2$--$1$ map) from the oriented double cover of $M$ onto $M$. This is a smooth minimal immersion. Let $\nu$ be a choice (on $\tilde{M}$) of unit normal to the immersion $\iota$. Recall (Remark \ref{oss:retract}) the coordinates $((y,v),s)=(q,s)$ on $V_{\tilde{M}}$, which is diffeomorphic to $T_\om \setminus \overline{S_{\dm}}\setminus (\overline{M}\setminus M)$; here $y\in M$ and $v$ a unit vector orthogonal to $M$ at $y$, or, equivalently, $q=(y,v)\in \tilde{M}$. For every compact set $K\subset \tilde{M}$ there exists $c_K>0$ such that $c_K< \sigma_{(y,v)}$ for all $(y,v)\in K$. This follows from the continuity of $\sigma_q$ on $\tilde{M}$ (Remark \ref{oss:retract}). Choosing $K$ even (i.e.~such that $K$ is the double cover of a compact set $\iota(K)$ in $M$) this means that $\iota(K)$ admits a two-sided tubular neighbourhood of semi-width $c_K$.

\medskip

We will now consider deformations of $\iota$ with initial velocity dictated by a function $\varphi \in C^2_c(\tilde{M})$. For $\varphi \in C_c^2(\tilde{M})$, choose $c_{{\text{supp}\varphi}}$ as above and consider the following one-parameter family of immersions $\iota_t:\tilde{M}\to N$ defined for $t\in(-\delta_0,\delta_0)$, where $\delta_0\in\left(0,\frac{c_{{\text{supp}\varphi}}}{\text{max}\varphi}\right)$:
$$(y,v)\to \text{exp}_{\iota(y)}(t \varphi((y,v)) \nu((y,v))),$$
for $(y,v)\in\tilde{M}$. The first variation of area at $t=0$ is $0$ because $M$ is minimal. The second variation of area at $t=0$ is given by 
\be
\label{eq:second_var}
\int_{\tilde{M}}|\nabla \varphi|^2  d{\Hc}^n-\int_{\tilde{M}}\varphi^2 (|A|^2 +\Rc{N}(\nu,\nu)) d{\Hc}^n,
\ee
where $A$ denotes the second fundamental form of $\iota$, $\nabla$ the gradient on $\tilde{M}$ (with respect to $\mathscr{g}_0$, the Riemannian metric induced by the pull-back from $M$), $\Rc{N}$ the Ricci tensor of $N$ and $\mathcal{H}^n$ is induced on $\tilde{M}$ by $\mathscr{g}_0$ (equivalently, integrate with repsect to $d\text{vol}_{\mathscr{g}_0}$).

\medskip

\begin{lem}[unstable region]
\label{lem:unstableregion_Ricci7}
There exist a geodesic ball $D\subset \subset M$ and $\tilde{\phi} \in C^2_c(\tilde{M})$ with $\tilde{\phi}\geq 0$, such that, writing $\tilde{D}=\iota^{-1}(D)$, the support of $\tilde{\phi}$ is contained in $M\setminus \overline{\tilde{D}}$ and
\begin{equation}
 \label{eq:unstableregion_Ricci7}
 \int_{\tilde{M}}|\nabla \tilde{\phi}|^2 d{\Hc}^n-\int_{\tilde{M}}\tilde{\phi}^2 (|A|^2 +\Rc{N}(\nu,\nu)) d{\Hc}^n < 0.
\end{equation}
\end{lem}

\begin{proof}
The second variation of $M$ is only defined for initial velocities induced by a function with compact support in $M$.
Fix an arbitrary point $b\in M$. Let $\delta>0$ be arbitrary and choose $\rho=\rho_{\delta}\in C^\infty_c(N)$ such that $0\leq \rho \leq 1$, $\rho=1$ in an open neighbourhood of $\{b\}\cup (\overline{M}\setminus M)$, $\rho=0$ in the complement of a (larger) open neighbourhood of $\{b\}\cup (\overline{M}\setminus M)$, and $\int_{N} |\nabla \rho|^2 <\delta$. This is possible because $\{b\}\cup (\overline{M}\setminus M)$ has finite (actually $0$) $\mathcal{H}^{n-2}$-measure (so its $2$-capacity is $0$, see \cite[Section 4.7]{EvGa}). 

Then the function $\varphi(q)=1-\rho(\iota(q))$ is admissible in (\ref{eq:second_var}) and the expression becomes (integrating on $M$)

$$2 \int_{M}|\nabla \rho|^2  d{\Hc}^n-2\int_{M}(1-\rho)^2 (|A_M|^2 +\Rc{N}(\nu,\nu)) d{\Hc}^n.$$
(Note that on $M$ the choice of $\nu$ is in general only permitted up to sign; this suffices for the term $\Rc{N}(\nu,\nu)$ to make sense.) Sending $\delta\to 0$ the second term tends to $-2\int_{M}(|A_M|^2 +\Rc{N}(\nu,\nu)) d{\Hc}^n$ and the first term tends to $0$ (recall that $\int_M |A_M|^2 <\infty$ thanks to the stability assumption) so the above expression converges to a negative number (because $\Rc{N}>0$). Therefore there exists $\delta$ sufficiently small such that  

$$2 \int_{M}|\nabla \rho|^2  d{\Hc}^n-2\int_{M}(1-\rho)^2 (|A_M|^2 +\Rc{N}(\nu,\nu)) d{\Hc}^n<0.$$
We let, for this $\delta$, $\tilde{\phi}(q)=1-\rho(\iota(q))$. Since $1-\rho$ vanishes in a neighbourhood of $b$, there exists a geodesic ball $D$ whose closure is disjoint from $\text{supp}(1-\rho)$ and therefore its double cover $\tilde{D}$ is a positive distance apart from $\text{supp}\tilde{\phi}$.
\end{proof}

\begin{oss}
This lemma uses $n\geq 2$ to argue that $\{b\}$ has codimension $\geq 2$ (for $n=1$ the lemma fails, e.g.~for $\RP^1 \subset \RP^2$). 
\end{oss}

\begin{oss}
\label{oss:two_comp_rho}
By the construction of $\rho$ in \cite{EvGa}, $\rho(x)=0$ when $\text{dist}_N(x,\{b\}\cup (\overline{M}\setminus M))>d_{\delta}$ for some $d_{\delta}\to 0$ as $\delta \to 0$. This means that for $\delta$ sufficiently small the support of $\rho$ has at least two (compact) connected components one of which contains $b$ (and thus $\overline{\tilde{D}}$)  while the union of the others contains an open neighbhourhood $O_1$ of $\overline{M}\setminus M$. Let $O\subset \subset O_1$ be an open set containing $\overline{M}\setminus M$. For $\tilde{\phi}=1-\rho\circ \iota$, we have that the complement of $\text{supp}\tilde{\phi}$ has at least two (open) connected components in $\tilde{M}$, one containing $\overline{\tilde{D}}$ while the other contains $\overline{\iota^{-1}(O)}$. Note that $K=\tilde{M}\setminus \iota^{-1}(O)$ is compact. These facts guarantee that $\tilde{\phi}$ vanishes in a neighbourhood of $\p \tilde{D}$ and of $\p (\iota^{-1}(O))=\p K$, a condition that will be technically useful in Section \ref{immersions}.
\end{oss}

\begin{oss}[choice of $B$]
\label{oss:choiceB_Ricci7}
Choose the ball $B$ in $M$ to be concentric with $D$ and with half the radius. Denote by $R>0$ the radius of $B$. Let $\tilde{B}=\iota^{-1}(B)$: this is the union of two geodesic balls in $\tilde{M}$. 
The choices of $B$ and $\tilde{\phi}$ will be kept until the end.
\end{oss}

\begin{oss}
\label{oss:geom_counterpart_Ricci7}
The geometric counterpart of Lemma \ref{lem:unstableregion_Ricci7} is that the minimal immersion $\iota$ is unstable with respect to the area functional also if we restric to deformations that leave $\tilde{D}$ (and $D$) fixed and that do not move $M$ close to its singular set $\overline{M}\setminus M$. We will be more specific in Section \ref{immersions} below.
\end{oss}

\section{Relevant immersions (choice of $\tau$)}
\label{immersions}

Recall Remark \ref{oss:two_comp_rho}. We will fix the compact subset $K=\tilde{M}\setminus \iota^{-1}(O)$ and will denote by $K_B$ the compact set $K\setminus \tilde{B}$, where $\tilde{B}$ is as in Remark \ref{oss:choiceB_Ricci7}. Note that both $K$ and $K_B$ are even in $\tilde{M}$, i.e.~they are double covers (via $\iota$) of compact subsets of $M$. We have $\text{supp}\tilde{\phi}\subset K_B\subset K$, for $\tilde{\phi}$ chosen in Lemma \ref{lem:unstableregion_Ricci7}. Recall that $\tilde{\phi}$ vanishes in a neighbourhood of $\p K_B$ (and of $\p K$). We will define on $K$ and $K_B$ suitable two-sided immersions into $N$, smooth up to the boundaries $\p K$ and $\p K_B$ (this means that there exist open neighbourhoods of $K$ and $K_B$ to which the immersions can be smoothly extended).

Choose $c_K>0$ such that $c_K<\min_{(y,v)\in K}\sigma_{(y,v)}$ (by the continuity of $\sigma>0$ on $\tilde{M}$ the minimum exists and is positive). We therefore have a well-defined one-sided tubular neighbourhood of $K$ in $V_{\tilde{M}}$, namely $K\times [0,c_K)$, with closure contained in $V_{\tilde{M}}$. Note that there exists an open neighbhourhood of $K$ on which $\sigma_{(y,v)}>c_K$, by continuity of $\sigma$ on $\tilde{M}$. 

Recall that $V_{\tilde{M}}$ is endowed with the Riemannian metric induced by the pull-back from $N$. Let $\Pi_K$ denote the nearest point projection onto $K$ (in coordinates, $\Pi_K(q,s)=(q,0)$). For future purposes, we ensure that $c_K$ above is also suitably small to ensure that, for $x=(q,s)\in K\times [0,c_K)$, then 
\be
\label{eq:JacPi}
\left|\,\,|J\Pi_K|(x)-1\,\,\right|\leq 2 C_{K} s \,\,\text{ and }\,\,\left|\,\,\frac{1}{|J\Pi_K|(x)}-1\,\,\right|\leq 2C_K s,
\ee
where $|J\Pi_K| = \sqrt{(D \Pi_K) (D\Pi_K)^T}$ and the constant $C_K>0$ is the maximum of the norm of the second fundamental form of $\iota:\tilde{M}\to N$ restricted to $K\subset \tilde{M}$. Note that $s$ is just the Riemannian distance of $(q,s)$ to $K$ (and to $\overline{M}$). 

\medskip

Choosing $\tilde{c}_0>0$ and $\tilde{t}_0>0$ sufficiently small, we can ensure that $(q,c+t\tilde{\phi}(q)) \in K\times [0,\frac{c_K}{2})$ for all $t\in [0,\tilde{t}_0]$ and for all $c\in [0,\tilde{c}_0]$. For any such $c,t$ we thus have a smooth  two-sided immersion $q=(y,v)\in \text{Int}(K)\to \text{exp}_{y}\left((c+t \tilde{\phi}(q)) v\right)$ from the interior of $K$ into $N$.

\begin{oss}
 \label{oss:extend_boundary}
 Note that, since $\tilde{\phi}=0$ in a neighbourhood of $\p K$,
the immersion $q=(y,v)\in \text{Int}(K)\to \text{exp}_{y}\left((c+t \tilde{\phi}(q)) v\right)$ agrees with $q=(y,v)\in \text{Int}(K)\to \text{exp}_{y}(c v)$ in a neighbourhood of $\p K$, therefore it extends smoothly to $\p K$. Similarly, $q=(y,v)\in \text{Int}(K_B)\to \text{exp}_{y}\left((c+t \tilde{\phi}(q)) v\right)$ extend smoothly to $\p K_B$ because $\tilde{\phi}=0$ vanishes in a neighbourhood of $\p K$.
\end{oss}

\begin{oss}
 \label{oss:Pi_ct}
\textit{(a)} Again thanks to the fact that $\tilde{\phi}=0$ in a neighbourhood of $\p K$, we have the following technically useful fact. For the two-sided immersion $q=(y,v)\in K\to \text{exp}_{y}\left((c+t \tilde{\phi}(q)) v\right)$, with $c>0$, denote by $\nu$ a choice of unit normal (which extends continuously up to $\p K$) and by $K_{c,t,\tilde{\phi}}$ its image. We can find $\underline{c}>0$ such that, for any $t\in [0,\tilde{t}_0]$ and $c\in [0,\tilde{c}_0]$, the set $\{\text{exp}_{x}(s\nu):s\in (-\underline{c}, \underline{c}), x \in K_{c,t,\tilde{\phi}}\}$ is contained in $K\times [0,c_K)$. By making $\underline{c}$ smaller if necessary, we can also ensure that the set 
$$\{\text{exp}_{x}(s\nu):s\in (-\min\{c, \underline{c}\}, \min\{c, \underline{c}\}), x \in K_{c,t,\tilde{\phi}}\}$$
is a tubular neighbourhood of $K_{c,t,\tilde{\phi}}$, in the sense that it admits a well-defined nearest point projection $\Pi_{c,t}$ onto $K_{c,t,\tilde{\phi}}$. This projection extends smoothly up to the boundary portion $\{\text{exp}_{x}(s\nu):s\in (-\min\{c, \underline{c}\}, \min\{c, \underline{c}\}), x \in \p K_{c,t,\tilde{\phi}}\}$. In fact, close to $\{\text{exp}_{x}(s\nu):s\in (-\min\{c, \underline{c}\}, \min\{c, \underline{c}\}), x \in \p K_{c,t,\tilde{\phi}}\}$ we have that $\Pi_{c,t}$ agrees with the nearest point projection onto $\Gamma_c$.

These properties essentially say that we can work with tubular neighbourhoods of $K_{c,t,\tilde{\phi}}$ without interfering with the complement of $F(K\times [0,c_K))$ and it will be useful when writing Allen--Cahn approximations of the immersions $q=(y,v)\in \text{Int}(K)\to \text{exp}_{y}\left((c+t \tilde{\phi}(q)) v\right)$.
\footnote{More precisely, we can patch the definition of Allen--Cahn approximation given in the tubular neighbourhood of $K_{c,t,\tilde{\phi}}$ (for $c=2\eps\Lambda$ to be chosen) with the function $G^{\eps}_0$ (that is an Allen--Cahn approximation of $\Gamma_{2\eps\Lambda}$).}

\textit{(b)} For notational convenience we redefine $\tilde{c_0}$, by choosing the minimum of $\tilde{c}_0$ specified above and $\underline{c}$ specified in (a). Then we have a well-defined nearest point projection 
$$\Pi_{c,t}: \{\text{exp}_{x}(s\nu):s\in (-c, c), x \in K_{c,t,\tilde{\phi}}\} \to K_{c,t,\tilde{\phi}}$$
for all $c\in (0,\tilde{c_0}]$ and all $t\in [0,\tilde{t_0}]$.
There exists a constant $C_{K,c_0,t_0}>0$, depending only on $c_0$, $t_0$, on the Riemannian metric and on the $C^3$ norms of $\tilde{\phi}$ on $K$ and of $F$, such that

\be
\label{eq:JacPi2}
\left|\,\,|J\Pi_{c,t}|(x)-1\,\,\right|\leq C_{K,c_0,t_0} s \,\,\text{ and }\,\,\left|\,\,\frac{1}{|J\Pi_K|(x)}-1\,\,\right|\leq C_{K,c_0,t_0} s,
\ee
where $|J\Pi_{c,t}|=\sqrt{(D \Pi_{c,t})(D \Pi_{c,t})^T}$ and $s$ is the distance of $x$ to $K_{c,t,\tilde{\phi}}$.
\end{oss}

\begin{oss}
 \label{oss:area_decr_1}
Choosing a suitably small $t_0\leq \tilde{t}_0$, $t_0>0$, we can further ensure that the area of the immersion $q=(y,v)\in \text{Int}(K)\to \text{exp}_{y}\left((t \tilde{\phi}(q)) v\right)$ is strictly decreasing in $t$ on the interval $[0,t_0]$. This follows upon noticing that the first variation (with respect to area) at $t=0$ is $0$ (by minimality of $M$) and the second variation at $t=0$ is negative by Lemma \ref{lem:unstableregion_Ricci7} (see Remark \ref{oss:geom_counterpart_Ricci7}). Remark that the immersions $q=(y,v)\in \text{Int}(K_B)\to \text{exp}_{y}\left((t \tilde{\phi}(q)) v\right)$ (the previous family of immersions restricted to $\text{Int}(K_B)$) have the same area-decreasing property, since $\tilde{\phi}=0$ on $\tilde{D}$. For the latter family of immersions, the area at $t=0$ is $\mathcal{H}^n(K) - \mathcal{H}^n(\tilde{B})\leq 2\mathcal{H}^n(M) - 2\mathcal{H}^n(B)$.
\end{oss}

\begin{lem}
\label{lem:unstable_addcnst_region_Ricci7}
Let $t_0$ be as in Remark \ref{oss:area_decr_1}. There exist $c_0\in (0,\tilde{c}_0]$ and $\tau>0$ such that 

\begin{description}
 \item[(i)] for all $c\in[0,c_0]$ and for all $t\in [0,t_0]$ the area of the immersion
 $$q=(y,v)\in \text{Int}(K_B)\to \text{exp}_{y}\left((c+t \tilde{\phi}(q)) v\right)$$
 is $\leq {\Hc}^n(K)-\frac{3}{4}{\Hc}^n(\tilde{B})=  {\Hc}^n(K)-\frac{3}{2}{\Hc}^n(B)$;
 \item[(ii)] for all $c\in[0,c_0]$ the area of the immersion 
 $$q=(y,v)\in \text{Int}(K)\to \text{exp}_{y}\left((c+t_0\tilde{\phi}(q)) v\right)$$
 is $\leq {\Hc}^n(K)-\tau$.
\end{description}
\end{lem}

\begin{proof}
Let us prove that (i) holds for some $c'_0\in (0,\tilde{c}_0]$ (in place of $c_0$). Argue by contradiction: if not, then there exists $c_i\to 0$ and $t_i \in [0,t_0]$ such that the area of $q\in \text{Int}(K_B)\to \text{exp}_{y}\left((c_i+t_i \tilde{\phi}(q)) v\right)$ is $\geq 2({\Hc}^n(M)-\frac{3}{4}{\Hc}^n(B))$ for all $i$. Upon extracting a subsequence we may assume $t_i \to t \in [0,t_0]$ and by continuity of the area we get that the area of $q\in \text{Int}(K_B)\to \text{exp}_{y}\left((t \tilde{\phi}(q)) v\right)$ is $\geq ({\Hc}^n(K)-\frac{3}{2}{\Hc}^n(B))$. This is however in contradiction with Remark \ref{oss:area_decr_1}, which says that this area is $\leq {\Hc}^n(K)-2{\Hc}^n(B)$.

Let us prove that (ii) holds for some $c''_0\in (0,\tilde{c}_0]$ (in place of $c_0$) and for some $\tau>0$. By Remark \ref{oss:area_decr_1} the area of $q=(y,v)\in \text{Int}(K)\to \text{exp}_{y}\left((t_0\tilde{\phi}(q)) v\right)$ is strictly smaller than ${\Hc}^n(K)$. Denote by $2\tau$ the difference of the two areas. By continuity, there exists $c''_0>0$ such that for all $c\in [0,c''_0]$ the area of the immersion $q=(y,v)\in \text{Int}(K)\to \text{exp}_{y}\left((c+t_0\tilde{\phi}(q)) v\right)$ is smaller than ${\Hc}^n(K)-\tau$.

Choosing $c_0=\min\{c'_0, c''_0\}$ concludes. 
\end{proof}

We will write, in Section \ref{all_paths}, Allen--Cahn approximation of the immersions in Lemma \ref{lem:unstable_addcnst_region_Ricci7}. To that end, we will work in the tubular neighbourhoods specified in Remark \ref{oss:Pi_ct}.

\medskip

\textit{Signed distance $\text{dist}_{K_{c,t,\tilde{\phi}}}$.} To write Allen--Cahn approximation of the immersions in Lemma \ref{lem:unstable_addcnst_region_Ricci7} we will need to use the following notion of signed distance to $K_{c,t,\tilde{\phi}}$. Recall that $\tilde{\phi}\geq 0$ is smooth and $\tilde{\phi}=0$ in a neighbourhood of $\p K$. In the coordinates of $V_{\tilde{M}}$, $K_{c,t,\tilde{\phi}}$ is identified with a graph, namely (for $c\in [0,c_0]$ and $t\in [0,t_0]$)
$$F^{-1}\left(K_{c,t,\tilde{\phi}}\right)=\{(q,s)\in K\times [0,c_K):s=c+t\tilde{\phi}(q)\}.$$
We define, on $K\times (0,c_K)$, the following ``signed distance to $F^{-1}\left(K_{c,t,\tilde{\phi}}\right)$'', for $c>0$. First we decide the sign of the distance: we say that $(q,s)\in K\times (0,c_K)$ has negative distance to $F^{-1}\left(K_{c,t,\tilde{\phi}}\right)$ if $s<c+t\tilde{\phi}(q)$ and positive distance to $F^{-1}\left(K_{c,t,\tilde{\phi}}\right)$ if $s>c+t\tilde{\phi}(q)$. Next we define its modulus. The modulus of the signed distance is the unsigned distance of $(q,s)$ to $F^{-1}\left(K_{c,t,\tilde{\phi}}\right)$ in $K\times (0,c_K)$ (recall that $K\times (0,c_K)$ is endowed with the Riemannian metric pulled back from $N$). Note that if $(q,s) \in F^{-1}\left(K_{c,t,\tilde{\phi}}\right)$ then the distance extends smoothly at $(q,s)$ with value $0$. Also remark that we do not define the unsigned distance on $K\times \{0\}$.
The signed distance just defined descends to a smooth function on $F\left(K\times (0,c_K)\right)\subset N$ that we will denote by $\text{dist}_{K_{c,t,\tilde{\phi}}}$. The set $F\left(K\times (0,c_K)\right)$ is an open tubular neighbourhood of $\iota(K)$ of semi-width $c_K$, with $M$ removed.

\section{Allen--Cahn approximations and paths in $W^{1,2}(N)$}
\label{all_paths}

The overall aim in the sections that follow is to produce, for all sufficiently small $\eps$, a continuous path in $W^{1,2}(N)$ that starts at the constant $-1$, ends at the constant $+1$ and such that $\Ece$ is bounded by $\approx 2\mathcal{H}^{n}(M) -\min\left\{\frac{\mathcal{H}^n(B)}{2},\frac{\tau}{2}\right\}$, where $B$ and $\tau$ were chosen respectively in Remark \ref{oss:choiceB_Ricci7} and Lemma \ref{lem:unstable_addcnst_region_Ricci7} and depend only on geometric data (not on $\eps$). Theorem \ref{thm:compare2M_Ricci7} (and Theorems \ref{thm:mult1_Ricci7}, \ref{thm:two-sided}) will follow immediately once this is achieved.

\subsection{Choice of $\eps$}
\label{choice_eps}
Let $B$ be as in Remark \ref{oss:choiceB_Ricci7} and $c_0$, $t_0$, $\tau$ be as in Lemma \ref{lem:unstable_addcnst_region_Ricci7}. The geometric quantities ${\Hc}^n(B)$ and $\tau$ are relevant in the forthcoming construction. 

In the following sections we are going to exhibit, for every sufficiently small $\eps$, a continuous path in $W^{1,2}(N)$ with $\Ece$ suitably bounded along the whole path. We will specify now an initial choice $\eps<\eps_1$ that permits the construction of the $W^{1,2}$-functions describing the path. When we will estimate $\Ece$ along the path, we will do so in terms of geometric quantities (typically, areas of cetain hypersurfaces, hence independent of $\eps$) plus errors that will depend on $\eps$. For sufficiently small $\eps$, i.e.~$\eps<\eps_2$ for a choice of $\eps_2\leq \eps_1$ to be specified, these errors will be $\leq C (\eps|\log\eps|)$, for some $C>0$ independent of $\eps$; we will not keep track of the constants and will instead write $O(\eps|\log\eps|)$. At the very end (Section \ref{final_argument_Ricci7}), in order to make these errors much smaller than  $\tau$ and $\mathcal{H}^n(B)$, and thus have an effective estimate on $\Ece$ along the path, we may need to revisit the smallness choice: for some $\eps_3$, (possibly $\eps_3\leq \eps_2$) we will get that for $\eps<\eps_3$  the errors can be absorbed in the geometric quantities. Therefore for $\eps<\eps_3$, we will have an upper bound for $\Ece$ along the path that is independent of $\eps$.

Now we choose $\eps_1$. The choices of $\eps_2$, $\eps_3$ will be made as we proceed into the forthcoming arguments. We restrict to $\eps_1<1$, so that the $O(\eps^2)$ controls that we have on the approximated one-dimensional solutions in Section \ref{one_dim_profiles} are valid for all $\eps<\eps_1$. We then require $\eps_1<\frac{1}{e}$ so to have $\eps|\log\eps|$ is decreasing as $\eps$ decreases so that the conditions specified on $\eps_1$ hold also for each $\eps<\eps_1$ and, moreover,

$$6{\eps}_1|\log{\eps}_1|< \frac{c_0}{20}$$ (and implicitly $<\frac{1}{2}\om$). Since the quantity $6{\eps}|\log{\eps}|$ will appear frequently (due to the choice of truncation in Section \ref{one_dim_profiles}), we will use the shorthand notation $\Lambda=3|\log\eps|$, when working at fixed $\eps$.

\subsection{Allen--Cahn approximation of $2(|M|-|B|)$}

Recall the function $G_0^{\eps}:N\to \R$ defined in (\ref{eq:Gt}), which is an Allen--Cahn approximation of $\iota:\tilde{M}\to N$, i.e.~a $W^{1,2}$ function with nodal set close to the image of $\iota$ and such that its Allen--Cahn energy $\Ece(G_0^{\eps})$ is approximately\footnote{In Section \ref{level_sets} we only established an upper bound for $\Ece(G^{\eps}_0)$, and most of the times an upper bound is all that will matter for our Allen--Cahn approximations (although a lower bound in terms of the area of the correspoding immersion is also going to be always valid). In the case of $G^{\eps}_0$, such a lower bound for $\Ece(G^{\eps}_0)$ will be established later.} the area of $\iota$ (i.e.~$\approx 2\mathcal{H}^n(M)$). Due to the fact that we replace hypersurfaces by non-sharp transitions, the function $G_0^{\eps}$ can also be thought of as an Allen--Cahn approximation of $\Gamma_{2\eps\Lambda}$ (that is exactly the nodal set of $G_0^{\eps}$).

\medskip

\textit{Definition of $f$.} We will now ``remove the ball $B$'' from $G_0^{\eps}:N\to \R$. In other words, we will write an Allen--Cahn approximation $f$ of $2 (|M|-|B|)$, or, equivalently, of $\left.\iota\right|_{\tilde{M}\setminus \tilde{B}}$. Always because we have non-sharp transitions, we can think of $f$ also as an Allen--Cahn approximation of $\Gamma_{2\eps\Lambda}$ with two balls removed. Although $f=f^{\eps}$ does depend on $\eps$, we drop the $\eps$ for notational convenience. What is important to keep in mind is that we can perform the contruction of $f$ given below for any $\eps<\eps_1$ and that we will obtain estimates on $\Ece(f^{\eps})$ that are uniform in $\eps$.

To this end, we let $\chi\in C^\infty_c(\tilde{M})$ be smooth and even (i.e.~$\chi(p)=\chi(q)$ if $\iota(p)=\iota(q)$) such that $\chi=1$ on $\tilde{B}$, $|\nabla \chi|\leq \frac{2}{R}$, where $R$ is the radius of $B$, and $\text{supp}\chi \subset \subset \tilde{D}$. Then we define, using coordinates $(q,s)\in K\times [0,c_K)\subset V_{\tilde{M}}$,

\be
\label{eq:G0_B}
G^{\eps}_{0,B}(q,s) = \Psi_{4\eps\Lambda\chi(q)}(s),
\ee
where $\Psi_t$ is as in (\ref{eq:family2}). Since $\chi$ is even, the function $G^{\eps}_{0,B}$ descends to a well-defined function $f$ on $F\left(K\times [0,c_K)\right)$ (this is a tubular neighbourhood of semi-width $c_K$ around $\iota(K)$). Note that $f$ agrees with $G^{\eps}_{0}$ on $F\left((K\setminus \tilde{D})\times [0,c_K)\right)$ and on $F\left((K \times (c_K/2,c_K)\right)$ (on the latter both are equal to $-1$), therefore we extend $f$ to $N$ by setting it equal to $G_0^{\eps}$ on $N\setminus F(K\times [0,c_K))$,
\be
\label{eq:f}
f(x)= \left\{\begin{array}{ccc}
                          G^{\eps}_{0} & \text{ for } x\in N\setminus F\left(K\times [0,c_K)\right) \\
G^{\eps}_{0,B}(F^{-1}(x))& \text{ for } x\in F\left(K\times [0,c_K)\right)
                         \end{array}\right. ,
                         \ee
and obtain that $f$ is $W^{1,\infty}$ on the complement of $F(\overline{\tilde{D}}\times [0,c_K/2])$.
Since $\Psi_t(x)$ is even and Lipschitz on $\R$, see (\ref{eq:family2}), we will in fact conclude that $f$ is $W^{1,\infty}$ on $N$. We only need to check it around points $x\in D$. Let $\chi_0:M\to \R$ be defined by $\chi_0(y)=\chi(\iota^{-1}(y))$; this is a smooth function compactly supported in $D$. In a neighbourhood of $x\in D$ we can choose a small geodesic ball $B_r(x)\subset M$ and use Fermi coordinates $(y,a) \in B_r(x)\times (-c_K,c_K)$. Then in this neighbourhood $f(y,a)=\Psi_{4\eps\Lambda\chi_0(y)}(a)$, hence $f$ is Lipschitz on $B_r(x)\times (-c_K,c_K)$. (The Jacobian factor that measures the distortion of the Riemannian metric from the product metric on $B_r(x)\times (-c_K,c_K)$ is bounded by a constant that only depends on the geometric data $F(K)\subset M \subset N$; therefore it suffices to observe that $\Psi_{4\eps\Lambda\chi_0(y)}(a)$ is Lipschitz with respect to the product metric.) Therefore $f\in W^{1,\infty}(N)$.

\medskip

\textit{Allen--Cahn energy of $f$.} To estimate from above the Allen--Cahn energy of $f$, since $f=G_0^{\eps}$ in the complement of $F\left(D\times [0,c_K]\right)$ and we estimated $\Ece(G^{\eps}_0)$ in (\ref{eq:G_t_energy}), we only need to compute the energy of $f$ on $F\left(D\times [0,c_K]\right)$ (and, similarly, the energy of $G^{\eps}_0$ on $F\left(D\times [0,c_K]\right)$). We can therefore use coordinates $(q,s)$ on $\tilde{D}\times [0,c_K] \subset V_{\tilde{M}}$ as in (\ref{eq:G0_B}) and apply the coarea formula (for the function $\Pi_K (q,s) = (q,0)$, whose Jacobian determinant $|J \Pi_K|$ is computed with respect to the Riemannian metric induced from $N$):
\be
\label{eq:comp_energy_G_0B}
\int_{F(D\times [0,c_K])} \eps\frac{|\nabla f|^2}{2} + \frac{W(f)}{\eps} =\int_{\tilde{B} \times (0,{c_K})} \eps |\nabla G^{\eps}_{0,B}|^2 +  \frac{W(G^{\eps}_{0,B})}{\eps}\,\,\,+
\ee
$$+\int_{\tilde{D}\setminus \tilde{B}} \left( \int_{(0, {c_K})}\frac{1}{|J\Pi_K|}\left( \eps \left|\frac{\p}{\p s} G^{\eps}_{0,B}\right|^2 + \frac{W(G^{\eps}_{0,B})}{\eps} \right)  ds\right) dq +$$ $$+ \int_{(\tilde{D}\setminus \tilde{B})\times (0,{c_K})}  \eps |\nabla_q G^{\eps}_{0,B}|^2 .$$
The notation $\nabla_q$ stands for the gradient projected onto the level sets of $s$. By definition of $G^{\eps}_{0,B}$ we have, in $\tilde{D}\times (0,c_K)$:
$$\nabla_q G^{\eps}_{0,B} = \left. \frac{d}{da}(\Psi_a)(z)\right|_{a=4\eps\Lambda\chi(q)} 4\eps\Lambda\nabla_q \chi$$ 
and since ($R$ denotes the radius of $B$) $|\nabla \chi|\leq \frac{2}{R}$, $|\frac{d}{da}(\Psi_a)(z)| = |\Psi'(|z|+a)| \leq \frac{3}{\eps}$, this implies ($\Lambda=3|\log\eps|$)

\begin{equation}
 \label{eq:tangential_Dirichel_estimate_Ricci7}
\eps|\nabla_q G^{\eps}_{0,B}|^2 \leq \eps \frac{C}{\eps^2} \frac{\eps^2 |\log\eps|^2}{R^2}= \frac{C \eps |\log\eps|^2}{R^2}. 
\end{equation}
(Here $C=(8\cdot 6)^2$.) Since $\tilde{B}, \tilde{D}$, $R$ and $C$ are independent of $\eps$, (\ref{eq:tangential_Dirichel_estimate_Ricci7}) implies that the third term on the right-hand-side of (\ref{eq:comp_energy_G_0B}) can be made arbitrarily small by choosing $\eps$ sufficiently small; this term is $O(\eps^2|\log\eps|^3)$, since the integrand is zero on $(\tilde{D}\setminus \tilde{B})\times (4\eps\Lambda,{c_K})$.
The first term on the right-hand-side of (\ref{eq:comp_energy_G_0B}) vanishes because $G^{\eps}_{0,B}=-1$ on that domain. For the second term on the right-hand-side of (\ref{eq:comp_energy_G_0B}), note that the inner integral only gives a contribution in $[0, 4\eps\Lambda]$ ($G^{\eps}_{0,B}=-1$ on $s\in[4\eps\Lambda,c_K]$). Recalling the bounds on the Jacobian factor $|J\Pi_K|$ given in (\ref{eq:JacPi}) and the energy estimates on the one-dimensional profiles, see (\ref{eq:hetenergy}) and (\ref{eq:family2}), we find  

$$\text{second term on right-hand side of (\ref{eq:comp_energy_G_0B})} \leq $$ $$\leq(1+8\eps\Lambda C_{K}) \int_{\tilde{D}\setminus \tilde{B}} \left(\int_0^{4\eps\Lambda} \frac{1}{2}\eps\left(\Psi_{4\eps\Lambda\chi(q)}^\prime\right)^2 + \frac{W(\Psi_{4\eps\Lambda\chi(q)})}{\eps}\right) dq$$ $$\leq {\Hc}^n(\tilde{D}\setminus \tilde{B})\,(1+8\eps\Lambda C_{K})\Ece( \OHet^{\eps})\leq $$ $$\leq ({\Hc}^n(\tilde{D})-\mathcal{H}^n(\tilde{B}))\,(1+8\eps\Lambda C_K)(2\sigma+O({\eps}^2)).$$
We can thus rewrite (\ref{eq:comp_energy_G_0B}) as a leading term $2\sigma\left({\Hc}^n(\tilde{D})-\mathcal{H}^n(\tilde{B})\right)$ plus errors; for a sufficiently small choice of $\eps_2\leq \eps_1$, for $\eps<\eps_2$, all errors are of the type $O(\eps|\log\eps|)$. We therefore conclude that the following estimate holds for all $\eps<\eps_2$:
$$\int_{F(D\times [0,c_K])} \eps\frac{|\nabla f|^2}{2} + \frac{W(f)}{\eps} \leq  4\sigma ({\Hc}^n(D)-\mathcal{H}^n(B)) + O(\eps|\log\eps|).$$

\medskip

Going back to $G_0^{\eps}$, we can give a lower bound to its energy on $F(D\times [0,c_K])$ with a computation analogous to the one just carried out. With coordinates $(q,s)\in D\times [0,c_K]$ we have that $G_0^{\eps}$ is simply the function $\Psi(s)$ and therefore $|\nabla G_0^{\eps}|$ is given by $\left|\frac{\p}{\p s} \Psi(s)\right|$ (the gradient is parallel to the $\frac{\p}{\p s}$). Using the coarea formula (again\footnote{It would also be possible to use the coarea formula slicing by the distance to $M$, as done in (\ref{eq:G_t_energy}), making use of Lemma \ref{lem:Gamma_t}.} with $\Pi_K$) we get that
$$\int_{F(D\times [0,c_K])} \eps\frac{|\nabla G_0^{\eps}|^2}{2} + \frac{W(G_0^{\eps})}{\eps} =$$
\be
\label{eq:G_0altracoarea}
=\int_{\tilde{D}} \left( \int_{0}^{4\eps\Lambda}\frac{1}{|J\Pi_K|}\left( \eps \left|\frac{\p}{\p s} \Psi(s)\right|^2 + \frac{W(\Psi(s))}{\eps} \right)  ds\right) dq\geq
\ee
$$\geq \mathcal{H}^n(\tilde{D})(1-8\eps \Lambda C_K)(2\sigma+O({\eps}^2)),$$
where we used (\ref{eq:JacPi}) and (\ref{eq:hetenergy}) and (\ref{eq:family2}). The result in (\ref{eq:G_0altracoarea}) is of the form $2\sigma \mathcal{H}^n(\tilde{D})$ plus errors. The errors are of the form $O(\eps|\log\eps|)$ for all $\eps<\eps_2$ for some suitably small choice of $\eps_2\leq \eps_1$. 

\begin{oss}[on the choice of ${\eps}_2$]
We make the choice of $\eps_2$ several times along the construction, always with the scope of making the errors controlled by $C\eps|\log\eps|$ with $C$ independent of $\eps \in (0,\eps_2)$. The specific value $\eps_2$ might change from one instance to the next, but since we make finitely many choices we implicitly assume that the correct $\eps_2$ is the smallest of all. From now on, this remark will apply every time we say that the errors are of the form $O(\eps|\log\eps|)$ for all $\eps<\eps_2$ for some suitably small choice of $\eps_2$. 
\end{oss}

\medskip

In conclusion for all $\eps<{\eps}_2$ we have that 
$$\frac{1}{2\sigma}\int_{F(D\times [0,c_K])} \eps\frac{|\nabla G_0^{\eps}|^2}{2} + \frac{W(G_0^{\eps})}{\eps} - \frac{1}{2\sigma}\int_{F(D\times [0,c_K])} \eps\frac{|\nabla f|^2}{2} + \frac{W(f)}{\eps} \geq$$ 
\be
\label{eq:fromGtof}
\geq 2\mathcal{H}^n(B) - |O(\eps|\log\eps|)|.
\ee
 
Recall that $f$ does depend on $\eps$, although we are not expliciting the dependence for notational convenience, and that we can produce $f$ (as defined above) for every $\eps<\eps_1$. By (\ref{eq:G_t_energy}) and (\ref{eq:fromGtof}), and the fact that $f=G_0^{\eps}$ on $N\setminus F(D\times [0,c_K])$, we conclude that for a sufficiently small choice of $\eps_2\leq \eps_1$, for all $\eps<\eps_2$, the following estimate holds:
\be
\label{eq:energy_f}
\Ece(f) \leq 2\left({\Hc}^n(M)-{\Hc}^n(B)\right)+O(\eps|\log\eps|).
\ee
This says that $f$ is a good\footnote{We only need the upper bound (\ref{eq:energy_f}), however a lower bound of the form $\Ece(f)\geq 2\left({\Hc}^n(\iota(K))-{\Hc}^n(D)\right)-O(\eps|\log\eps|)$ is also easily seen to be valid.}  Allen--Cahn approximation of $2(|M|-|B|)$.
In terms of the immersions of Lemma \ref{lem:unstable_addcnst_region_Ricci7}, $f$ is also an Allen--Cahn approximation of $q=(y,v)\in \text{Int}(K_B)\to \text{exp}_{y}\left(2\eps\Lambda v\right)$ (the nodal set of $f$ contains the image of this immersion with boundary).

\subsection{From $\Ece(-1)=0$ to $2 (|M|-|B|)$}

In this section we construct a continuous path in $W^{1,2}(N)$ that joins $f$ to the constant $-1$, keeping $\Ece$ along the path controlled by $\Ece(f)$.

We begin by introducing the following one-parameter family of functions: for $r\in[0,4\eps\Lambda]$ define

\be
\label{eq:Gr}
Y^{\eps}_r(x) = \left\{\begin{array}{ccc}
                          -1 & \text{ for } x\in N\setminus T_\om \\
\Psi_r(\dm(x))& \text{ for } x\in T_\om
                         \end{array}\right. ,
                         \ee
where $\Psi_r$ is as in (\ref{eq:family2}). Since $\OHet^{\eps}$ is constantly $-1$ on $(-\infty, -2\eps\Lambda]$, the function $Y^{\eps}_t$ is constantly $-1$ on $\{x:\dm(x)> 4\eps\Lambda-r\}$. Moreover, since $\dm$ is Lipschitz on $N$ and $\Psi_r$ is Lipschitz on $\R$, we have $|\Psi_r(\dm(x)) - \Psi_r(\dm(y))|\leq C_{\Psi_r} |\dm(x)-\dm(y)|\leq C_{\Psi_r} C_{\dm} \text{dist}(x,y)$, where $C_{\Psi_r}$, $C_{\dm}$ denote the Lipschitz constants of $\Psi_r$ and $\dm$ respectively. We therefore conclude that $Y^{\eps}_r \in W^{1,\infty}(N)$.

Notice that $Y^{\eps}_0=G^{\eps}_0$. We compute $\Ece(Y^{\eps}_0)$ by using the coarea formula (slicing by the distance function $\dm$, for which $|\nabla \dm|=1$) as we did for $G^{\eps}_0$ (see (\ref{eq:G_t_energy})). We obtain
\be
\label{eq:energy_Yr}
\Ece(Y_r^{\eps})\leq  2 \mathcal{H}^n(M) \left(\frac{1}{2\sigma}\int_{0}^{4\eps\Lambda-r}  \eps \frac{(\Psi_r^\prime)^2}{2} + \frac{W(\Psi_r)}{\eps}\right) \leq   2 \mathcal{H}^n(M) (1+O({\eps}^2)),
\ee
using (\ref{eq:hetenergy}) and the fact that $\int_{0}^{4\eps\Lambda-r}  \eps \frac{(\Psi_r^\prime)^2}{2} + \frac{W(\Psi_r)}{\eps}\leq \int_{0}^{\infty}  \eps \frac{(\Psi_r^\prime)^2}{2} + \frac{W(\Psi_r)}{\eps}=2\sigma+O({\eps}^2)$. Note that $\Ece(G_r^{\eps}) \to 0$ as $r\to 4\eps\Lambda$.

Now we give a lower bound for the energy of $Y_r^{\eps}$ on the domain $F(D\times [0,c_K))$ as we did for $G_0^{\eps}$ in (\ref{eq:G_0altracoarea}), i.e.~using the coarea formula for the function $\Pi_K$. Note that on this domain we can use the coordinates $(q,s)$ on $D\times [0,c_K)$ and the fact that the gradient of $Y_r^{\eps}$ is parallel to $\frac{\p}{\p s}$. We have

$$\int_{F(D\times [0,c_K])} \eps\frac{|\nabla Y_r^{\eps}|^2}{2} + \frac{W(Y_r^{\eps})}{\eps} =$$
\be
\label{eq:energy_Yr_lowerbound}
=\int_{\tilde{D}} \left( \int_{0}^{4\eps\Lambda-r}\frac{1}{|J\Pi_K|}\left( \eps \left|\frac{\p}{\p s} \Psi_r(s)\right|^2 + \frac{W(\Psi_r(s))}{\eps} \right)  ds\right) dq\geq
\ee
$$\geq 2\mathcal{H}^n(D)(1-4\eps \Lambda C_K) \int_{0}^{4\eps\Lambda-r}  \eps \frac{(\Psi_r^\prime)^2}{2} + \frac{W(\Psi_r)}{\eps},$$
where we used (\ref{eq:JacPi}) and the fact that the part in paretheses of the inner integrand is independent of $q$. We therefore conclude, from (\ref{eq:energy_Yr}) and (\ref{eq:energy_Yr_lowerbound}), the following estimate for the Allen-Cahn energy of $Y_r^{\eps}$ in the complement of $F(D\times [0,c_K))$: there exists $\eps_2\leq \eps_1$ sufficiently small such that for all $\eps<{\eps}_2$
\begin{equation}
\label{eq:energy_Yr_complement}
\int_{N\setminus F(D\times [0,c_K])} \eps\frac{|\nabla Y_r^{\eps}|^2}{2} + \frac{W(Y_r^{\eps})}{\eps} \leq 
\end{equation}
$$\leq 2\left(\mathcal{H}^n(M)-\mathcal{H}^n(D)\right) \int_{0}^{4\eps\Lambda-r}  \eps \frac{(\Psi_r^\prime)^2}{2} + \frac{W(\Psi_r)}{\eps} + O(\eps|\log\eps|).$$

\medskip

\textit{Definition of the path $f_r$}. We now define a continuous path $r\in[0,4\eps\Lambda]\to f_r\in W^{1,2}(N)$ as follows. Recalling the definition of $\chi\in C^{\infty}_c(\tilde{M})$ and using coordinates $(q,s)\in \tilde{D}\times [0,c_K)$ we set

$$Y_{r,B}(q,s) = \Psi_{4\eps\Lambda\chi(q)+r}(s),$$
where $\Psi_t$ is as in (\ref{eq:family2}).
The function $f_r:N\to \R$ is then defined by

\be
\label{eq:path-1_Ricci7}
f_r(x)= \left\{ \begin{array}{ccc}
             Y_r^{\eps}(x) & \text{ if } x\in N\setminus F(\tilde{D}\times [0,c_K))\\
             Y_{r,B}(F^{-1}(x)) & \text{ if } x\in F(\tilde{D}\times [0,c_K))
            \end{array}
\right. .
\ee
Note that $f_r$ is well-defined on $D$ since $\chi$ is even.
Remark also that for $r=0$ this function is $f$ and for $r=4\eps\Lambda$ it is the constant $-1$. Moreover, $f_r\in W^{1,\infty}(N)$ for every $r$. To see this, notice that $Y_{r,B}$ is smooth on $\tilde{D}\times (0,c_K)$, so $f_r$ is smooth on $F(\tilde{D}\times (0,c_K))$. Moreover, $f_r\in W^{1,\infty}(N\setminus F(\overline{\tilde{D}}\times [0,c_K]))$ because it agrees with $ Y_r^{\eps}$ on this open set. The smoothness at $F(\overline{\tilde{D}}\times\{c_K\})$ is immediate because $f_r=-1$ in a neighbourhood of $F(\overline{\tilde{D}}\times\{c_K\})$. We thus only need to check that $f_r$ is Lipschitz locally around any point $x\in \overline{D}$. Using Fermi coordinates $(y,a)\in B(x)\times (-\delta,\delta)$, where $B(x)$ is a small geodesic ball in $M$ centred at $x$ and $\delta>0$, we have the following expression for $f$, thanks to the fact that $\Psi_r:\R\to \R$ is even for every $r$: $f_r(y,a)=\Psi_{4\eps\Lambda\chi_0(y)+r}(a)$, where $\chi_0(p)=\chi(F^{-1}(p))$. Since $\Psi_r(z)$ is Lipschitz on $\{(r,z): r\in [0,\infty), z\in \R\}$, and since $\chi_0$ is smooth, we obtain that $f_r \in W^{1,\infty}$ on the chosen neighbourhood of $x$. (As we did in (\ref{eq:f}), we use the fact that being Lipschitz for the product metric on $B(x)\times (-\delta,\delta)$ implies Lipschitzianity for the Riemannian metric induced from $N$.) In conclusion we have $f\in W^{1,\infty}(N)$.

The path $r\in [0,4\eps\Lambda]\to f_r \in W^{1,2}(N)$ is moreover continuous. Let us check the continuity of $\nabla f_r$ in $r$ (with respect to the $L^2$ topology on $N$). For $\nabla f_r$ on $F(\tilde{D}\times [0,c_K))$ we have the following expression, using $(q,s)$-coordinates on $\tilde{D}\times (0,c_K)$:
$$\left(4\eps\Lambda \nabla \chi (q) \Psi^\prime_{0}(s+4\eps\Lambda\chi(q)+ r) ,\Psi^\prime_{0}(4\eps\Lambda\chi(q)+r+s)\right).$$
By continuiuty of translations in $L^p$ and smoothness of $\chi$ we get that $\nabla f_r$ is continuous in $r$ (with respect to the $L^2$-topology, or even $L^p$ for any $p$). Similarly we can argue for $T_\om \setminus F(\tilde{D}\times [0,c_K))$, where $f=Y_{r,0}^{\eps}$ and the gradient is $\Psi^\prime_{0}(r+\dm(x))\nabla \dm(x)$: this changes continuously with $r$ (with respect to the $L^2$-topology, or even $L^p$ for any $p$). Therefore we have that $r\in [0,4\eps\Lambda]\to \nabla f_r \in L^2(N)$ is continuous. The fact that $f_r$ changes continuously in $r$ with respect to the $L^2$ topology is even more straightforward.

\medskip

\textit{Energy along the path.} To estimate $\Ece(f_r)$ we compute the energy on $F(\tilde{D} \times [0,c_K))$ using the coarea formula for $\Pi_K$, similarly to (\ref{eq:energy_f}), in the coordinates $(q,s) \in \tilde{D} \times [0,c_K)$. Notice that $Y_{0,B}^{\eps}(q,s)=-1$ for $q\in B$. Then we obtain
$$\int_{F(\tilde{D} \times [0,c_K))} \eps \frac{|\nabla f_r|^2}{2} + \frac{W(f_r)}{\eps} =$$
$$ =\int_{\tilde{D}\setminus \tilde{B}} \int_0^{c_K} \frac{1}{|J\Pi_K|} \left(\eps \frac{|\Psi_r^\prime(s)|^2}{2} + \frac{W(\Psi_r(s))}{\eps} \right) ds \, dq \,\,    \underbrace{\leq}_{(\ref{eq:JacPi})} $$
$$\leq (1+8\eps\Lambda)\int_{\tilde{D}\setminus \tilde{B}} \int_0^{c_K} \left(\eps \frac{|\Psi_r^\prime(s)|^2}{2} + \frac{W(\Psi_r(s))}{\eps} \right) ds\,  dq\leq $$
$$\leq 2 (1+8\eps\Lambda) \mathcal{H}^n(D\setminus B) \left(\int_{r}^{4\eps\Lambda}  \eps \frac{(\Psi^\prime)^2}{2} + \frac{W(\Psi)}{\eps}\right).$$
Recalling that $f_r = Y_{r,B}^{\eps}$ on $N\setminus F(\tilde{D} \times [0,c_K))$ and by the estimate in ({\ref{eq:energy_Yr_complement}) we conclude that there exists $\eps_2\leq \eps_1$ such that for all $\eps\leq \eps_2$ the following estimates hold for $r\in[0,4\eps\Lambda]$:

\be
\label{eq:energy_path_1_Ricci7}
\Ece(f_r)\leq  2 \left(\mathcal{H}^n(M)-\mathcal{H}^n(B)\right)  \left(\frac{1}{2\sigma}\int_{r}^{4\eps\Lambda}  \eps \frac{(\Psi^\prime)^2}{2} + \frac{W(\Psi)}{\eps}\right)+O(\eps|\log\eps|),
\ee
$$\Ece(f_r)\leq 2 \left(\mathcal{H}^n(M)-\mathcal{H}^n(B)\right) +O(\eps|\log\eps|).$$
(The second follows from the first since the energy of $\Psi$ in parentheses is $\leq 1+O(\eps^2)$.) The second estimate shows the uniform energy control on $r\in[0,4\eps\Lambda]$; the first shows that $\Ece(f_r)\to 0$ as $r\to 4\eps\Lambda$.

\begin{oss}
At least for $n\leq 6$ it is possible to produce a continuous path from $f$ to $-1$, with a similar energy control as in (\ref{eq:energy_path_1_Ricci7}), by employing an alternative argument. One can employ the negative $\Ece$-gradient flow starting at a suitably constructed function $f_2$ that is $W^{1,2}$-close to $f$ and with $\Ece(f_2)\approx \Ece(f)$. The flow will be mean convex and will tend to the constant $-1$, reaching it in time $O(\eps|\log\eps|)$. The $\eps\to 0$ limit of such paths is then the Brakke flow that starts at $2(|M|-|B|)$ and vanishes instantaneously. The family (\ref{eq:path-1_Ricci7}) that we gave in this section mimics exactly this flow, however it is more elementary, even for $n\leq 6$, as we can exhibit the path explicitly (and moreover presents no additional difficulties for $n\geq 7$). Note that the path $f_r$ that we produced also reaches $-1$ in time $O(\eps|\log\eps|)$. 
\end{oss}

\subsection{Lowering the peak}
In this section we construct the next portion of our path, starting at $f$. The immersions in Lemma \ref{lem:unstable_addcnst_region_Ricci7} are particularly relevant, as they provide the geometric counterpart of this portion of the $W^{1,2}$-path: first we use the immersions in (i) of Lemma \ref{lem:unstable_addcnst_region_Ricci7} keeping $c=2\eps\Lambda$ and increasing $t$ from $0$ to $t_0$; then we connect the final immersion just obtained to the one in (ii) of Lemma \ref{lem:unstable_addcnst_region_Ricci7} with $t=t_0$ and $c=2\eps\Lambda$ (in doing so, we ``close the hole at $B$''). The portion of the path that we exhibit in this section is made of Allen--Cahn approximations of the immersions just described. It is this portion of the path that ``lowers the peak'' of $\Ece$, keeping it a fixed amount below $2\mathcal{H}^n(M)$ (thanks to the estimates in Lemma \ref{lem:unstable_addcnst_region_Ricci7}).

\medskip

We will keep using the shorthand notation $\Lambda=3|\log\eps|$. All the functions that we will construct in this section coincide with $G^{\eps}_0$ in the complement of $F(K\times [0,c_K))$. By construction they will in fact agree with $G^{\eps}_0$ in a neighbourhood of $\p F(K \times [0,c_K))$ (guaranteeing a smooth patching) and thanks to Remark \ref{oss:Pi_ct} and since $2\eps\Lambda <c_0/20$ (Section \ref{choice_eps}) we can use tubular neighbourhoods of semi-width $2\eps\Lambda$ around $K_{c,t,\tilde{\phi}}$ for every $c\geq 2\eps\Lambda$ to define Allen--Cahn approximations of the immersions in Lemma \ref{lem:unstable_addcnst_region_Ricci7}.

\medskip

Recall the notation $K_{c,t,\tilde{\phi}}$ from Section \ref{immersions}: it denotes the image via $F:V_{\tilde{M}}\to N$ of the graph $\{(q,s)\in V_{\tilde{M}}: q\in K, s=c+t\tilde{\phi}(q)\}$, for $t\in[0,t_0]$ and $c\in [0,c_0]$. In other words, $K_{c,t,\tilde{\phi}}$ is image of the immersion (smoothly extended up to $\p K$, see Remark \ref{oss:extend_boundary}) $q=(y,v)\in K \to \text{exp}_{y}\left((c+t \tilde{\phi}(q)) v\right)$.
Recall the definition of signed distance provided in Section \ref{immersions} and denote by $\text{dist}_{K_{c,t,\tilde{\phi}}}$ the signed distance to $K_{c,t,\tilde{\phi}}$, well-defined on $F(\tilde{D}\times (0,c_K))$. If $t=0$, then $\text{dist}_{K_{c,0,\tilde{\phi}}}$ extends continuously to $F(\tilde{D}\times [0,c_K))$ with value $-c$ on $F(\tilde{D}\times \{0\})$.

Then the definition of $f$ in (\ref{eq:G0_B})-(\ref{eq:f}), can equivalently be given as follows

$$f(x) =\left\{ \begin{array}{ccc}
\OHet^{\eps}_{4\eps\Lambda\chi_0(\Pi_K(x))}\left(-\text{dist}\left(x,K_{2\eps\Lambda,0,\tilde{\phi}}\right)\right)  & \text{ for } x\in F(K\times [0,c_K))\\
G_0^{\eps}(x) & \text{ for } x\in N\setminus F(\tilde{D}\times [0,c_K)),
                \end{array} \right.$$
where $$\OHet^{\eps}_{s}(\cdot)=\OHet^{\eps}(\cdot -s),$$ $\chi_0=\chi\circ F^{-1}$ and, with a slight abuse of notation, $\Pi_K(x)$ is the nearest point projection of $x$ onto $M$. (In the coordinates of $V_{\tilde{M}}$ we have $\Pi_K(q,s)=(q,0)$, which is the notation used in Section \ref{immersions}; the map on $F(K\times [0,c_K))$ that we are using above should then be $F\circ \Pi_K \circ F^{-1}$, we however denote both the map in $K\times [0,c_K)$ and the map in $F(K\times [0,c_K))$ by the same symbol $\Pi_K$.)

\begin{oss}
 \label{oss:extend_signed_distance}
The signed distance $\text{dist}_{K_{2\eps\Lambda,t, \tilde{\phi}}}(x)$ is defined on  $K  \times (0,c_K)$. 
We point out the following facts. Let $x\in F(K\times \{0\})$ and $x_j\to x$, $x_j\in K  \times (0,c_K)$ (so that the signed distance is negative on $x_j$ for $j$ sufficiently large). Then 
$$\limsup_{j\to \infty} \text{dist}_{K_{2\eps\Lambda,t, \tilde{\phi}}}(x_j) \leq -2\eps\Lambda.$$
Moreover, $\text{dist}_{K_{2\eps\Lambda,t, \tilde{\phi}}}$ extends continuously, with value $-2\eps\Lambda$, to $\left(K\setminus \text{supp}(\tilde{\phi})\right)\times \{0\}$. In particular, this continuous extension is valid on (an open neighbourhood of) $\overline{D}$.
\end{oss}

\textit{Definition of $g_t$.} We construct now the portion of path $t\in [0,t_0]\to g_t \in W^{1,2}(N)$ whose geometric counterpart is given by the immersions in (i) of Lemma \ref{lem:unstable_addcnst_region_Ricci7} with $c=2\eps\Lambda$ and $t\in[0,t_0]$. These immersions ``have a hole at $\tilde{B}$''. We set, for $t\in [0,t_0]$:

\be
\label{eq:gt_path2}
g_t(x)=\left\{\begin{array}{ccc}
G_0^{\eps}(x)\,\,\,\,\,\,\, (\text{see (\ref{eq:Gt})})& \text{ for } x\in  N\setminus F(K\times [0,c_K))\\
\OHet^{\eps}_{4\eps\Lambda\chi_0(\Pi_K(x))}(-\text{dist}_{K_{2\eps\Lambda,t, \tilde{\phi}}}(x)) & \text{ for }  x\in F(K \times (0,c_K))  \cup D  \\
1 & \text{ for }  x\in F((K\setminus \overline{D})\times \{0\})     \\
             \end{array}\right..
\ee
In the second line of (\ref{eq:gt_path2}) we are using the fact that $\text{dist}_{K_{2\eps\Lambda,t, \tilde{\phi}}}$ is well-defined and continuous on $\overline{D}$, with value $-2\eps\Lambda$ (see Remark \ref{oss:extend_signed_distance}). Also note that on $F(\p K\times [0,c_K))$ the definition in the second line agrees with the definition of $G_0^{\eps}$ ($\tilde{\phi}$ vanishes in a neighbourhood of $\p K$, see Section \ref{immersions}) and the same is true on $F(K\times \{c_K\})$ ($g_t=-1$ in a neighbourhood). For $t=0$ we have $g_0=f$, by the expresion of $f$ given earlier in this section.

\textit{$g_t \in W^{1,\infty}(N)$ for each $t$.} Let us check first that $g_t$ it is continuous on $N$ for each $t$. In view of the comments just made, this needs to be checked only at an arbitrary $x$ in $F((\text{Int}(K)\setminus \overline{D})\times \{0\})$. Let $x_j\to x$, then for sufficiently large $j$ we have $x_j\in F((\text{Int}(K)\setminus \overline{D})\times [0,c_K))$. Then $x, x_j \notin \text{supp}\chi_0 \times [0,c_K)$. Therefore by (\ref{eq:gt_path2}) we get $g_t(x_j)=\OHet^{\eps}(-\text{dist}_{K_{2\eps\Lambda,t, \tilde{\phi}}}(x_j))$. Recall Remark \ref{oss:extend_signed_distance}. By continuity of $\OHet^{\eps}$ and the fact that $\OHet^{\eps}(z)=1$ for $z\geq 2\eps\Lambda$ we conclude that $g_t(x_j)\to 1$, hence $g_t$ is continuous at $x$.

To check that $g_t\in W^{1,\infty}(N)$, note first that the definition in the second line of (\ref{eq:gt_path2}) is equal to the one of $G_0^{\eps}$ in a neighbourhood of the boundary of $F(K\times [0,c_K))$. Moreover $g_t$ is smooth on $F(K\times (0,c_K))$ and $G_0^{\eps}$ is $W^{1,\infty}(N)$. These fact imply that $g_t\in W^{1,\infty}(N\setminus F(K\times \{0\}))$, and actually even in a neighbourhood of the boundary of $F(K\times [0,c_K))$. Moreover, for $x\in B$ we have $g_t=-1$ in a neighbourhood of $x$, because  $\chi_0=1$ on $B$ and $\OHet^{\eps}_{4\eps\Lambda}(z)\equiv -1$ for $z\leq 2\eps\Lambda$.

Therefore we only need to check that $g_t$ is locally Lipschitz around points $x\in F((\text{Int}(K)\setminus \tilde{B})\times \{0\})$. We distinguish two cases. If $x\notin \overline{D}$, i.e.~if $x\in F((\text{Int}(K)\setminus \overline{\tilde{D}})\times \{0\})$, then $g_t$ is actually $C^1$ in a neighbourhood of $x$. This is seen by repeating the argument used above (for the continuity of $g_t$ at such point) to prove that $|\nabla g_t (x_j)|\to 0$ (using the fact that ${\OHet^{\eps}}^{\prime}$ is smooth on $\R$ and equal to $0$ on $[2\eps\Lambda,\infty)$). We therefore have: $g_t$ is $C^1$ on $F((\text{Int}(K)\setminus \overline{\tilde{D}})\times (0,c_K))$, $g_t$ extends continuously to $F((\text{Int}(K)\setminus \overline{\tilde{D}})\times \{0\})$ with constant value $1$ and $\nabla g_t$ extends continuously to this set with value $0$. From these facts it follows that the $L^\infty$ function equal to $\nabla g_t$ on $F((\text{Int}(K)\setminus \overline{\tilde{D}})\times (0,c_K))$ is the ditributional derivative of $g_t$ in a neighbourhood of $x$ and therefore $g_t$ is $C^1$ in a neighbourhood of $x$.
In the second case, i.e.~if $x \in \overline{D}\setminus B$, then for a sufficiently small ball $B_\rho(x)\subset M$ we have $\tilde{\phi}=0$ on $F^{-1}(B_\rho(x))$ (because $\text{supp}(\tilde{\phi})$ and $\overline{\tilde{D}}$ are disjoint) and we can use a well-defined system of Fermi coordinates $(y,a)\in B_\rho(x)\times (-c_K, c_K)$. In these coordinates we have $\text{dist}_{K_{2\eps\Lambda,t, \tilde{\phi}}}(F(y,a))=|a|-2\eps\Lambda$ and $g_t(F(y,a))=\OHet^{\eps}_{4\eps\Lambda\chi(y)}(-|a|+2\eps\Lambda)$, which is Lipschitz in the neighbourhood.

\medskip

\textit{The path $t\to g_t$ is continuous.} It suffices to check that the second line in (\ref{eq:gt_path2}) is continuous in $t$. The proof is analogous to the one where we proved that $f_r$ defined in (\ref{eq:path-1_Ricci7}) is continuous in $r$; it can be carried out using the coordinates on $V_{\tilde{M}}$ and the fact that the graph $\{(q,s):q\in K, s=2\eps\Lambda +t\tilde{\phi}(q)\}$ changes smoothly in $t$, hence so does the function $\text{dist}_{K_{2\eps\Lambda,t, \tilde{\phi}}}$. Therefore the path $t\in [0,t_0]\to g_t \in W^{1,2}(N)$ is continuous.

\medskip

\textit{Energy of $g_t$.} To give an upper bound for $\Ece(g_t)$ we first need a lower bound for the energy of $G_0^{\eps}$ on $F(K\times [0,c_K))$. This is analogous to the estimate in (\ref{eq:G_0altracoarea}):

$$\int_{F(K\times [0,c_K])} \eps\frac{|\nabla G_0^{\eps}|^2}{2} + \frac{W(G_0^{\eps})}{\eps} =$$
$$=\int_{K} \left( \int_{0}^{4\eps\Lambda}\frac{1}{|J\Pi_K|}\left( \eps \left|\frac{\p}{\p s} \Psi_0(s)\right|^2 + \frac{W(\Psi_0(s))}{\eps} \right)  ds\right) dq\,\,\,\underbrace{\geq}_{(\ref{eq:JacPi})}$$
$$\geq (1-8\eps \Lambda C_K) \int_{K} \left( \int_{0}^{4\eps\Lambda}\left( \eps \left|\frac{\p}{\p s} \Psi_0(s)\right|^2 + \frac{W(\Psi_0(s))}{\eps} \right)  ds\right) dq\geq$$
$$\geq \mathcal{H}^n(K)(1-8\eps \Lambda C_K)\left( \int_{-2\eps\Lambda}^{2\eps\Lambda}\eps \left|{\OHet^{\eps}}'\right|^2 + \frac{W({\OHet^{\eps}})}{\eps} \right)=$$
\be
\label{eq:G_0_lowerboundK}
=\mathcal{H}^n(K)(1-8\eps \Lambda C_K)(2\sigma+O({\eps}^2)),
\ee
where we used (\ref{eq:hetenergy}). 

From (\ref{eq:G_t_energy}) and (\ref{eq:G_0_lowerboundK}) we obtain that, for some suitably small choice of $\eps_2\leq \eps_1$, for all $\eps<\eps_2$ the following holds for the energy of $G_0^{\eps}$ (and thus also of $g_t$) in $N\setminus F(K\times [0,c_K])$:
\be
\label{eq:G_taltracoarea}
\frac{1}{2\sigma}\int_{N\setminus F(K\times [0,c_K])} \eps\frac{|\nabla G_0^{\eps}|^2}{2} + \frac{W(G_0^{\eps})}{\eps} \leq \mathcal{H}^n(\tilde{M}\setminus K) + O(\eps|\log\eps|).
\ee
We now pass to an estimate for the energy of $g_t$ in $F(K\times [0,c_K])$. For this we will use Fermi coordinates for a tubular neighbourhood of $K_{2\eps\Lambda, t, \tilde{\phi}}$ of semi-width $2\eps\Lambda$. Denote by $(y,a) \in K_{2\eps\Lambda, t, \tilde{\phi}} \times (-2\eps\Lambda,2\eps\Lambda)$ such coordinates and by $\Pi_{2\eps\Lambda,t}$ the nearest point projection from the chosen tubular neighbourhood onto $K_{2\eps\Lambda, t, \tilde{\phi}}$ (see Remark \ref{oss:Pi_ct}). Remark that $g_t=-1$ on $F(\tilde{B}\times [0,c_K))$ so there is no energy contribution in this open set.
The coarea formula (for the function $\Pi_{2\eps\Lambda,t}$) then gives:

$$\int_{F((K\setminus \tilde{B})\times [0,c_K])} \eps\frac{|\nabla g_t|^2}{2} + \frac{W(g_t)}{\eps} =$$
$$=\int_{K_{2\eps\Lambda, t, \tilde{\phi}}\setminus F(\tilde{D}\times [0,c_K))} \left( \int_{-2\eps\Lambda}^{2\eps\Lambda}\frac{1}{|J\Pi_{2\eps\Lambda,t}|}\left( \eps \left|{\OHet^{\eps}}^\prime(a)\right|^2 + \frac{W(\OHet^{\eps}(a))}{\eps} \right)  da\right) dy + $$ 
$$+\int_{K_{2\eps\Lambda, t, \tilde{\phi}}\cap F((\tilde{D}\setminus\tilde{B})\times [0,c_K))} \left( \int_{-2\eps\Lambda}^{2\eps\Lambda(1-2\chi(\Pi_K((y,a))))}\frac{1}{|J\Pi_{2\eps\Lambda,t}|}\left( \eps \left|{\OHet^{\eps}}^\prime(a)\right|^2 + \frac{W(\OHet^{\eps}(a))}{\eps} \right)  da\right) dy\,\leq $$
$$\leq \int_{K_{2\eps\Lambda, t, \tilde{\phi}}\setminus F(\tilde{B}\times [0,c_K))} \left( \int_{-2\eps\Lambda}^{2\eps\Lambda}\frac{1}{|J\Pi_{2\eps\Lambda,t}|}\left( \eps \left|{\OHet^{\eps}}^\prime(a)\right|^2 + \frac{W(\OHet^{\eps}(a))}{\eps} \right)  da\right) dy \,\,\underbrace{\leq}_{(\ref{eq:JacPi2})}$$
$$\leq (1+2\eps \Lambda C_{K,t_0,c_0,\tilde{\phi}}) \mathcal{H}^n (K_{2\eps\Lambda, t, \tilde{\phi}}\setminus F(\tilde{B}\times [0,c_K))) \left(\int_{-2\eps\Lambda}^{2\eps\Lambda} \eps \left|{\OHet^{\eps}}^\prime\right|^2 + \frac{W(\OHet^{\eps})}{\eps}\right) =$$
$$(1+2\eps \Lambda C_{K,t_0,c_0,\tilde{\phi}}) \mathcal{H}^n (K_{2\eps\Lambda, t, \tilde{\phi}}\setminus F(\tilde{B}\times [0,c_K)))(2\sigma+O({\eps}^2) ).$$
Therefore for some suitably small choice of $\eps_2\leq \eps_1$, for all $\eps<\eps_2$ the following holds 

$$\frac{1}{2\sigma}\int_{F(K\times [0,c_K])} \eps\frac{|\nabla g_t|^2}{2} + \frac{W(g_t)}{\eps} \leq  \mathcal{H}^n (K_{2\eps\Lambda, t, \tilde{\phi}}\setminus F(\tilde{B}\times [0,c_K))) + O(\eps|\log\eps|).$$
Note that $K_{2\eps\Lambda, t, \tilde{\phi}}\setminus F(\tilde{B}\times [0,c_K))$ is the image of $K_B$ via the immersion in (i) of Lemma \ref{lem:unstable_addcnst_region_Ricci7} when $c=2\eps\Lambda$. Using Lemma \ref{lem:unstable_addcnst_region_Ricci7} in the last estimate and putting it together with (\ref{eq:G_taltracoarea}) we finally obtain that, for some suitably small choice of $\eps_2\leq \eps_1$, for all $\eps<\eps_2$ the following estimate holds for all $t\in [0,t_0]$:

\be
\label{eq:energy_gt_Ricci7}
\Ece(g_t)\leq 2\left(\mathcal{H}^n(M) -\frac{3}{4}\mathcal{H}^n(B)\right) + O(\eps|\log\eps|).
\ee

\medskip

\textit{Definition of $g_{t_0+r}$: ``closing the hole at $B$''.} We have constructed a continuous path $t\in [0,t_0]\to g_t\in W^{1,2}(N)$ with $g_0=f$ and with $\Ece$ uniformly controlled by (\ref{eq:energy_gt_Ricci7}). The next portion of the path will start from $g_{t_0}$ and will ``close the hole at $B$''. On the geometric side, we are starting at the immersion in Lemma \ref{lem:unstable_addcnst_region_Ricci7} (i) with $c=2\eps\Lambda$ and $t=t_0$, and ending at the immersion in Lemma \ref{lem:unstable_addcnst_region_Ricci7} (ii) with $c=2\eps\Lambda$ and $t=t_0$.
We define for $r \in [0,1]$

\be
\label{eq:gtr_path3}
g_{t_0+r}(x)=\left\{\begin{array}{ccc}
G_0^{\eps}(x)\,\,\,\,\,\,\, (\text{see (\ref{eq:Gt})})& \text{ for } x\in  N\setminus F(K\times [0,c_K))\\
\OHet^{\eps}_{4\eps\Lambda(1-r)\chi_0(\Pi_K(x))}(-\text{dist}_{K_{2\eps\Lambda,t, \tilde{\phi}}}(x)) & \text{ for }  x\in F(K \times (0,c_K))  \cup D  \\
1 & \text{ for }  x\in F((K\setminus \overline{D})\times \{0\})     \\
             \end{array}\right..
\ee
Note that $g_{t_0+r}=g_{t_0}$ when $r=0$ (justifying the notation). Moreover $g_{t_0+r}(x)=g_{t_0}(x)$ for $r\in [0,1]$ and $x\in F(\text{supp}(\chi) \times [0,c_K))$. In other words, we are only changing the values of $g_{t_0}$ in the set $F\left(\tilde{D}\times [0,c_K)\right)$ (equivalently, introducing Fermi coordinates centred at $D$, the set $D \times (-c_K,c_K)$).

The fact that $g_{t_0+r}\in W^{1,\infty}(N)$ for every $r\in [0,1]$ follows by repeating the arguments used for $g_t$, where the only part that has to be altered is the local expression of $g_{t_0+r}$ around points of $D$. Using Fermi coordinates $(y,a)$ with $y\in D$, $a\in (-c_K,c_K)$ we get $g_{t_0+r}(F(y,a))=\OHet^{\eps}_{4\eps\Lambda(1-r)\chi(y)}(-|a|+2\eps\Lambda)$, which is Lipschitz. Notice that this is the domain in $N$ where we are ``closing the hole'': when $r=1$ the expression just obtained becomes $g_{t_0+1}(F(y,a))=\OHet^{\eps}_{0}(-|a|+2\eps\Lambda)=\Psi_0(a)$ and so

\be
\label{eq:gt1_path3}
g_{t_0+1}(x)=\left\{\begin{array}{ccc}
G_0^{\eps}(x)\,\,\,\,\,\,\, (\text{see (\ref{eq:Gt})})& \text{ for } x\in  N\setminus F(K\times [0,c_K))\\
\OHet^{\eps}(-\text{dist}_{K_{2\eps\Lambda,t, \tilde{\phi}}}(x)) & \text{ for }  x\in F(K \times (0,c_K))\\
1 & \text{ for }  x\in F(K \times \{0\})     \\
             \end{array}\right..
\ee
Note also that $r\in [0,1]\to g_{t_0+r}\in W^{1,2}(N)$ is a continuous path (with a proof as the ones used earlier for $g_t$ and $f_r$).

\medskip

\textit{Energy of $g_{t_0+r}$.} We use the coarea formula as we did to reach (\ref{eq:energy_gt_Ricci7}). We get

$$\int_{F(K\times [0,c_K])} \eps\frac{|\nabla g_{t_0+r}|^2}{2} + \frac{W(g_{t_0+r})}{\eps} =$$
$$=\int_{K_{2\eps\Lambda, t_0, \tilde{\phi}}} \left( \int_{-2\eps\Lambda}^{2\eps\Lambda(1-2(1-r)\chi(\Pi_K(y,a)))}\frac{1}{|J\Pi_{2\eps\Lambda,t_0}|}\left( \eps \left|{\OHet^{\eps}}^\prime(a)\right|^2 + \frac{W(\OHet^{\eps}(a))}{\eps} \right)  da\right) dy\,\leq $$
$$\leq \int_{K_{2\eps\Lambda, t_0, \tilde{\phi}}} \left( \int_{-2\eps\Lambda}^{2\eps\Lambda}\frac{1}{|J\Pi_{2\eps\Lambda,t_0}|}\left( \eps \left|{\OHet^{\eps}}^\prime(a)\right|^2 + \frac{W(\OHet^{\eps}(a))}{\eps} \right)  da\right) dy \,\,\underbrace{\leq}_{(\ref{eq:JacPi2})}$$
$$\leq (1+2\eps \Lambda C_{K,t_0,c_0,\tilde{\phi}}) \mathcal{H}^n (K_{2\eps\Lambda, t_0, \tilde{\phi}}) \left(\int_{-2\eps\Lambda}^{2\eps\Lambda} \eps \left|{\OHet^{\eps}}^\prime\right|^2 + \frac{W(\OHet^{\eps})}{\eps}\right) =$$
$$=(1+2\eps \Lambda C_{K,t_0,c_0,\tilde{\phi}}) \mathcal{H}^n (K_{2\eps\Lambda, t_0, \tilde{\phi}})(2\sigma+O({\eps}^2) ).$$
Therefore for some suitably small choice of $\eps_2\leq \eps_1$, for all $\eps<\eps_2$ the following holds 

$$\frac{1}{2\sigma}\int_{F(K\times [0,c_K])} \eps\frac{|\nabla g_{t_0+r}|^2}{2} + \frac{W(g_{t_0+r})}{\eps} \leq  \mathcal{H}^n (K_{2\eps\Lambda, t_0, \tilde{\phi}}) + O(\eps|\log\eps|).$$
Note that $K_{2\eps\Lambda, t_0, \tilde{\phi}}$ is the image of $K$ via the immersion in (ii) of Lemma \ref{lem:unstable_addcnst_region_Ricci7} when $c=2\eps\Lambda$. Using Lemma \ref{lem:unstable_addcnst_region_Ricci7} in the last estimate and putting it together with (\ref{eq:G_taltracoarea}) we finally obtain that, for some suitably small choice of $\eps_2\leq \eps_1$, for all $\eps<\eps_2$ the following estimate holds for all $r\in [0,1]$:

\be
\label{eq:energy_gtr_Ricci7}
\Ece(g_t)\leq 2 \mathcal{H}^n(M) -\tau + O(\eps|\log\eps|).
\ee

\subsection{Connect to $+1$}
\label{reach_1}

To conclude the construction of our path, we will connect $g_{t_0+1}$ to the constant $+1$ by means of a (negative) gradient flow. To this end, we will produce a suitable barrier $m$, constructed from $G_0^{\eps}$. First we check that
$$G_0^{\eps}\leq g_{t_0+1} \,\,\,\text{ on } N.$$ 
To see this, recall that $G_0^{\eps}= g_{t_0+1}$ on $N\setminus F(K\times [0,c_K))$, so we only need to compare the two functions on $F(K\times [0,c_K))$. 
On this domain we use coordinates $(q,s)\in K\times [0,c_K)$. Use the following temporary notation: $H(x)=\OHet^{\eps}(-x)$, $T= \{(q,s):q\in K, s=2\eps\Lambda+t_0\tilde{\phi}(q)\}$ and $d(q,s)=\text{dist}_{K_{2\eps\Lambda, t_0, \tilde{\phi}}}(F(q,s))$. Equivalently, the latter signed distance is $\text{sgn}_{(q,s)}\text{dist}((q,s), T)$, where $\text{dist}$ is the Riemannian distance (induced from $N$) and $\text{sgn}_{(q,s)}=-1$ on $\{(q,s):q\in K, 0<s<2\eps\Lambda+t_0\tilde{\phi}(q)\}$ and $\text{sgn}_{(q,s)}=+1$ on $\{(q,s):q\in K, 2\eps\Lambda+t_0\tilde{\phi}(q)\leq s<c_K\}$. 
Then $G_0^{\eps}(q,s) = H(s-2\eps\Lambda)$ and $g_{t_0+1}(q,s)=H(d(q,s))$. If $\text{sgn}_{(q,s)}=-1$ then the Riemannian distance to $T$ is $\geq 2\eps\Lambda -s$, because $T$ lies above $\{s=2\eps\Lambda\}$. Similarly, if $\text{sgn}_{(q,s)}=+1$ then the Riemannian distance to $T$ is $\leq s-2\eps\Lambda$. Therefore in either case we have $d(q,s)\leq s-2\eps\Lambda$. This implies (since $H$ is decreasing) that $G_0^{\eps}(q,s)\leq  g_{t_0+1}(q,s)$.

\medskip

We are going to work with the ``modified'' Allen--Cahn energy 
$$\ca{F}{\eps,\mu_{\eps}}(u)=\Ece(u) - \frac{\mu_{\eps}}{2\s} \int_N u\, d\mathcal{H}^{n+1},$$
where $\mu_{\eps}>0$ tends to $0$ as $\eps\to 0$. The role of $\mu_{\eps}$ is that of a forcing term, to ensure that the flow ``moves in the desired direction'' and is moreover ``mean-convex''. There is flexibility on the choice of $\mu_{\eps}$; we fix the following (note that in Section \ref{level_sets} we only required $\mu_{\eps}>|O(\eps^2)|$ in order to obtain (\ref{eq:first_var_Gt})):
\be
\label{eq:choice_of_mu}
\mu_{\eps}=\eps|\log\eps|.
\ee

We are now ready to contruct the barrier. 
\begin{lem}
\label{lem:barrier_m} 
For all sufficiently small $\eps$ there exists a smooth function $m:N\to \R$ ($m=m^{\eps}$) such that $m<g_{t_0+1}$ and $-(2\s)\ca{F'}{\eps,\mu_{\eps}}(m)=\eps\Delta m -\frac{W'(m)}{\eps}+\mu_{\eps}>0$.
\end{lem}

\begin{proof}
From (\ref{eq:first_var_Gt}) we have for all sufficiently small $\eps$ that

$$-(2\s)\ca{F'}{\eps,\mu_{\eps}}(G^{\eps}_0)\geq \frac{\mu_{\eps}}{2} \mathcal{H}^{n+1},$$
where $\mu_{\eps}>0$ is defined in (\ref{eq:choice_of_mu}). Recall that this inequality means that (the positive Radon measure) $-(2\s)\ca{F'}{\eps,\mu_{\eps}}(G^{\eps}_0)$ minus $\frac{\mu_{\eps}}{2} \mathcal{H}^{n+1}$ is a positive measure. 

For $\rho>0$ consider the function $G^{\eps}_0-\rho$. Then $\Delta(G^{\eps}_0-\rho)=\Delta G^{\eps}_0$ and $W'(G^{\eps}_0-\rho)$ converges uniformly on $N$ to $W(G^{\eps}_0)$ as $\rho\to 0$. Therefore we can find a sufficiently small $\rho_0>0$ (depending on $\eps$, in fact we may choose $\rho_0\approx \eps^2$) such that for all sufficiently small $\eps$ we have
\be
\label{eq:g-rho}
-(2\s)\ca{F'}{\eps,\mu_{\eps}}(G^{\eps}_0-\rho_0)\geq \frac{\mu_{\eps}}{3}\mathcal{H}^{n+1}.
\ee
Let $C_N$ be the constant in Lemma \ref{lem:second_derivative}. We are going to work with $\eps$ sufficiently small to ensure (in addition to the previous conditions identified so far in this proof) that $2\eps (2\s) C_N <\mu_{\eps}/20$. From now we work at fixed $\eps$ (satisfying the smallness conditions just imposed).

Let $\eta_{\delta}$ be the mollifiers defined in Appendix \ref{mollifiers} for $\delta<\delta_0$, where $\delta_0>0$ depends only on the geometry of $N$. Then the (smooth) function $-(2\s)\ca{F'}{\eps,\mu_{\eps}}(G^{\eps}_0-\rho_0) \star \eta_{\delta}$ defined in (\ref{eq:def_mollify_Radon}) is positive for all $\delta$, more precisely for all sufficiently small $\delta$ (one needs $\frac{\mu_{\eps}}{12}>O(\delta^2)$, where $O(\delta^2)$ appears in (\ref{eq:int_eta}))
\be
\label{eq:wrong_order_conv}
\left(-(2\s)\ca{F'}{\eps,\mu_{\eps}}(G^{\eps}_0-\rho_0)\right) \star \eta_{\delta} \geq \frac{\mu_{\eps}}{4}.
\ee
This follows from (\ref{eq:g-rho}) and (\ref{eq:int_eta}), (\ref{eq:def_mollify_Radon}). 
We now mollify $(G^{\eps}_0-\rho_0)$ as in (\ref{eq:def_mollify_function}). We have $|G^{\eps}_0-\rho_0|<2$, since $|G^{\eps}_0|\leq 1$. From Lemma \ref{lem:convolution_W12} (a) we obtain that the functions $(G^{\eps}_0-\rho_0) \star \eta_{\delta}$ converge uniformly on $N$ to $(G^{\eps}_0-\rho_0)$ as $\delta\to 0$.
Therefore (for $\delta$ sufficiently small $-2< (G^{\eps}_0-\rho_0)\star\eta_{\delta}< 2$ since the same bound holds for $G^{\eps}_0-\rho_0$)
$$\left\|W'((G^{\eps}_0-\rho_0) \star \eta_{\delta}) - W'(G^{\eps}_0-\rho_0)\right\|_{C^0(N)}\leq$$ $$\leq \|W''\|_{C^0([-2,2])} \|(G^{\eps}_0-\rho_0) \star \eta_{\delta} - (G^{\eps}_0-\rho_0)\|_{C^0(N)}\to 0$$
as $\delta\to 0$. The function $W'(G^{\eps}_0-\rho_0)$ belongs to $W^{1,\infty}(N)$, therefore by Lemma \ref{lem:convolution_W12} (b) we get $\left\|W'(G^{\eps}_0-\rho_0) \star \eta_{\delta} - W'(G^{\eps}_0-\rho_0)\right\|_{C^0(N)}\to 0$. By the triangle inequality we therefore have
\be
\label{eq:exchange_conv_order_W}
\left\|W'((G^{\eps}_0-\rho_0) \star \eta_{\delta}) - W'(G^{\eps}_0-\rho_0)  \star \eta_{\delta}\right\|_{C^0(N)}\to 0
\ee
as $\delta\to 0$. By Lemma \ref{lem:second_derivative} there exists $C_N$ (depending only on the geometry of $N$) such that for all $\delta<\delta_0$ we have $\|\Delta((G^{\eps}_0-\rho_0) \star \eta_{\delta}) - \Delta(G^{\eps}_0-\rho_0) \star \eta_{\delta}\|_{L^\infty(N)} \leq C_N \|G^{\eps}_0-\rho_0\|_{L^\infty(N)} \leq 2 C_N$. Therefore the modulus of the difference of the following two (smooth) functions

$$\eps\Delta((G^{\eps}_0-\rho_0) \star \eta_{\delta}) -\frac{W'(G^{\eps}_0-\rho_0)  \star \eta_{\delta})}{\eps}+\mu_{\eps} \,\,\text{ and } \,\,\left(-(2\s)\ca{F'}{\eps,\mu_{\eps}}(G^{\eps}_0-\rho_0)\right) \star \eta_{\delta}$$
is at most $2\eps C_N + O_{\delta}(1)$, where the infinitesimal of $\delta$ is given by the norm in (\ref{eq:exchange_conv_order_W}) plus $O(\delta^2) \mu_{\eps} $.
Recall (\ref{eq:wrong_order_conv}) and the smallness condition imposed on $\eps$. Then for sufficiently small $\delta$, writing $m=(G^{\eps}_0-\rho_0) \star \eta_{\delta}$, we have  

\be
\label{eq:first_var_m}
\eps \Delta m - \frac{W'(m)}{\eps}+\mu_{\eps}\geq \frac{\mu_{\eps}}{5}.
\ee
Finally note that for sufficiently small $\delta$ we also have $m<g_{t_0+1}$, since $G_0^{\eps}-\rho_0<g_{t_0+1}$ and $(G^{\eps}_0-\rho_0) \star \eta_{\delta}$ converges uniformly to $G^{\eps}_0-\rho_0$ as $\delta\to 0$ (Lemma \ref{lem:convolution_W12}).
\end{proof}

\begin{oss}(choice of $\eps_2$, again)
We will assume that Lemma \ref{lem:barrier_m} is valid for all $\eps< \eps_2$, where once again we change the choice of $\eps_2$ if necessary.
\end{oss}

\textit{Flow from $m$.} We consider now the negative gradient flow of $(2\s)\ca{F}{\eps,\mu_{\eps}}$, with initial condition given by the smooth function $m$, i.e.~the solution $m_t$ to the PDE

\be
\label{eq:F_flow}
\left\{\begin{array}{ccc}
\eps \frac{\p}{\p t} m_t  = \eps \Delta m_t - \frac{W'(m_t)}{\eps}+\mu_{\eps}\\
m_0 = m
       \end{array}\right.,
\ee
where $\Delta$ is the Laplace--Beltrami operator on $N$. This semilinear parabolic problem has a solution for $t\in [0,\infty)$ and $m_t \in C^\infty(N)$ for all $t>0$, as we will now sketch. 

Short-time existence and uniqueness for a weak solution in $W^{1,2}(N)$ are valid by standard semilinear parabolic theory (rewrite the problem as an integral equation, then use a fixed point theorem). To see why we get global existence in our case, integrate (\ref{eq:F_flow}) on any interval $[0,T]$ on which the weak solution is defined: we get
\begin{equation}
 \label{eq:integrate_flow_m}
 \eps\int_0^T \left(\int_N \left|\frac{\p}{\p t} m_t\right|^2\right) dt +\frac{\eps}{2} \int_N \left|\nabla m_T\right|^2= 
\end{equation}
$$= \frac{\eps}{2} \int_N \left|\nabla m_0\right|^2 - \frac{1}{\eps} \int_N \left(W(m_T)-\eps\mu_{\eps} m_T\right)+ \frac{1}{\eps} \int_N \left(W(m_0)-\eps\mu_{\eps} m_0\right).$$
With our choice of $W$ that is quadratic on $(\pm 2,\pm \infty)$ we can ensure that $\frac{W(u)}{\eps}-\mu_{\eps}u$ is bounded below. Then (\ref{eq:integrate_flow_m}) gives a priori bounds $\int_N \left|\nabla m_t\right|^2\leq C_{m_0,\eps, W}$ independently of $t\in[0,T]$. Again from (\ref{eq:integrate_flow_m}), moving the term $\frac{1}{\eps} \int_N \left(W(m_t)-\eps\mu_{\eps} m_t\right)$ to the left-hand-side and recalling $|u|^2\leq C_{W,\eps} \max\{2,\frac{W(u)}{\eps}-\mu_{\eps}u\}$, we also get an a priori $L^2$-bound on $m_t$. In conclusion 
\begin{equation}
\label{eq:apriori_bounds_parabolic}
 \|m_t\|_{W^{1,2}(N)}\leq C, \,\,\,\,  \int_0^T \left(\int_N \left|\frac{\p}{\p t} m_t\right|^2\right) \leq C,
\end{equation}
with $C$ independent of $t$. This first bound in (\ref{eq:apriori_bounds_parabolic}) provides the assumption under which short-time existence can be iterated to lead global existence for a weak solution to $(\ref{eq:F_flow})$ in $W^{1,2}(N)$. 

Writing the PDE in the form $\frac{\p}{\p t} m_t - \Delta m_t =\frac{1}{\eps^2} (W(m_t)-\eps\mu_{\eps}m_t)$, we treat the right-hand-side as the non-homogeneous term $f$ of a linear PDE. The problem is subcritical thanks to the quadratic growth of $W$: Sobolev embeddings give $u_t\in L^{2^*}(N)$ for all $t$, with $2(n+1)/((n+1)-2)=2^*>2$, so $f:[0,\infty)\to L^p(N)$ with $p>1$. Bootstrapping gives $C^\infty$ regularity\footnote{First we can use regularity theory in $L^p$ for $p>1$ for the linear parabolic PDE $\frac{\p}{\p t} m_t - \Delta m_t =f_t$ combined with Sobolev embeddings, to obtain $m_t\in W^{2,p}(N)$ for all $p<\infty$. Then differentiating the PDE we get $\frac{\p}{\p t} (\nabla m_t) - \Delta (\nabla m_t) =\frac{1}{\eps^2} (W'(m_t)-\eps\mu_{\eps})\nabla m_t$ and we can again treat this as a linear PDE with right-hand-side of the type $[0,\infty)\to L^p(N)$ with $p>1$, therefore $\nabla m_t\in W^{2,p}(N)$ for some $p>1$. Iterating, we get $m_t\in W^{k,p}$ for all $k\in \N$, $p<\infty$, hence smoothness.}.

\begin{lem}[\textit{mean convexity of $m_t$}]
 \label{lem:mean_convex}
The positivity condition $-(2\s)\ca{F'}{\eps,\mu_{\eps}}(m_t)=\eps \Delta m_t - \frac{W'(m_t)}{\eps} + \mu_{\eps}> 0$ holds for all $t\geq 0$. 
\end{lem}

\begin{proof}
For notational convenience, we write for this paragraph $F_t=\eps \Delta m_t - \frac{W'(m_t)}{\eps} + \mu_{\eps}$ (right-hand-side of the first line in (\ref{eq:F_flow})). By the previous discussion, $F_t$ is smooth on $N$ for all $t\in [0,\infty)$. Differentiating $F_t=\eps \Delta m_t - \frac{W'(m_t)}{\eps} + \mu_{\eps}$ (and using $\eps \p_t m_t=F_t$) we get the evolution of $F_t$, given by $ \p_t F_t =   \Delta F_t - \frac{W''(m_t)}{\eps^2} F_t$. So $F_t$ solves $ \p_t \gamma =   \Delta \gamma - \frac{W''(m_t)}{\eps^2} \gamma$, and the constant $\gamma=0$ is also a solution to the same PDE. The condition $F_t>0$ is therefore preserved by the maximum principle, since it holds at $t=0$ by Lemma \ref{lem:barrier_m}.
\end{proof}

Lemma \ref{lem:mean_convex} implies in particular that $m_t:N\to \R$ is increasing in $t$ (since $\p_t m_t = -\frac{2\s}{\eps}\ca{F'}{\eps,\mu_{\eps}}(m_t)>0$), therefore $\lim_{t\to \infty} m_t = m_\infty$ is well-defined pointwise on $N$.
The $W^{1,2}(N)$-norm of $m_t$ is bounded uniformly in $t$ by (\ref{eq:apriori_bounds_parabolic}), therefore $m_t \to m_\infty$ in $W^{1,2}$-weak. Moreover $\|W'(m_t)\|_{W^{1,2}(N)}$ is also uniformly bounded in $t$, since $|\nabla (W'(m_t))| = |W''(m_t)| |\nabla m_t|\leq \|W''\|_{C^0([-2,2])}|\nabla m_t|$ (one can chech that $-2\leq m_t\leq 2$ for all $t$ by the maximum principle). Therefore $W'(m_t)\to W'(m_\infty)$ in $W^{1,2}$-weak. By the second bound in (\ref{eq:apriori_bounds_parabolic}) we have $L^1$-summability in time, on $t\in(0,\infty)$, for $\left\|\frac{\p}{\p t} m_t\right\|_{L^2(N)}$ and therefore there exists $t_j\to \infty$ such that the function $\frac{\p}{\p t} m_t:N\to \R$ has $L^2(N)$-norm that tends to $0$ along the sequence $t_j$.
These facts imply that the weak formulation of the PDE in (\ref{eq:F_flow}) passes to the limit as $t_j\to \infty$ and gives that $m_\infty$ solves $-\ca{F'}{\eps,\mu_{\eps}}=0$ in the weak sense. Standard elliptic theory (or passing parabolic estimates for $m_t$ to the $t\to \infty$ limit) then show that $m_\infty \in C^{\infty}$ solves $-\ca{F'}{\eps,\mu_{\eps}}(m_\infty)=0$ in the strong sense.

\begin{lem}[\textit{stability of $m_\infty$}]
 \label{lem:m_infty_stable}
The limit $m_\infty$ of the flow $m_t$ (as $t\to \infty$) is a stable solution of $\ca{F'}{\eps,\mu_{\eps}}=0$. 
\end{lem}

\begin{proof}
This is a consequence of the ``mean convexity'' of $m_t$ (Lemma \ref{lem:mean_convex}) and of the maximum principle.  We give the explicit argument. Recall from the previous discussion that $m_\infty$ is stationary, i.e.~$\ca{F'}{\eps,\mu_{\eps}}(m_\infty)=0$. Also recall that the second variation at $u:N\to \R$ of the functional $(2\s)\ca{F}{\eps,\mu}$ (for a constant $\mu$) on the test function $\phi$ is given by the quadratic form $Q(\phi,\phi)=\int_N \eps|\nabla \phi|^2 + \frac{W''(u)}{\eps}\phi^2$ (the term involving $\mu$ disappears because it is linear) and the Jacobi operator is given by $-\eps \Delta \phi + \frac{W''(u)}{\eps}\phi$.

Let $\rho_1$ be its first eigenfunction (with respect to the energy $(2\s)\ca{F}{\eps,\mu_{\eps}}$), then $\rho_1$ is (strictly) positive and smooth on $N$. Consider, for $s\in(-\delta,\delta)$ (for some small positive $\delta$), the functions $m_\infty - s\rho_1$. Then their first variation satisfies
$$\frac{\p}{\p s} \left(-(2\s)\ca{F'}{\eps,\mu_{\eps}}(m_\infty -s \rho_1)\right) = -\eps \Delta \rho_1 + \frac{W''(m_\infty-s\rho_1)}{\eps}\rho_1.$$
If $m_{\infty}$ were unstable, then the first eigenfunction would satisfy $-\eps\Delta \rho_1 +\frac{W''(m_\infty)}{\eps}\rho_1=\lambda_1 \rho_1$ for some $\lambda_1<0$ and therefore
$$\left.\frac{\p}{\p s}\right|_{s=0} \left(-(2\s)\ca{F'}{\eps,\mu_{\eps}}(m_\infty -s \rho_1)\right)=\lambda_1\rho_1<0$$
on $N$. Then we could choose $s_0>0$ sufficiently small so that 
\be
\label{eq:one-sided_minimization1}
-(2\s)\ca{F'}{\eps,\mu_{\eps}}(m_\infty -s \rho_1) = \eps \Delta (m_\infty -s \rho_1) - \frac{W'((m_\infty -s \rho_1) )}{\eps}+\mu_{\eps} <0
\ee
on $N$ for $s\in [0,s_0]$. Note that $m_\infty -s \rho_1$ is smooth on $N$. 

Since $-(2\s)\ca{F'}{\eps,\mu_{\eps}}(m_t)>0$, at any $t\in[0,\infty)$ we have $m_t>m_0$, in particular $m_\infty>m_0$. Choose $s$ sufficiently small so that $s<s_0$ and $m_\infty -s \rho_1 > m_0$. Let $\tau>0$ be the first time for which $m_\tau$ has a point $x$ such that $m_\tau(x)= (m_\infty -s \rho_1)(x)$.  Then $m_\infty -s \rho_1- m_\tau$ is a smooth non-negative function on $N$ with a minimum at $x$, so $\Delta (m_\infty -s \rho_1)(x)\geq \Delta m_\tau(x)$. Moreover we have $W'(m_\infty -s \rho_1) = W'(m_\tau)$ at $x$. Recalling that $\eps \Delta m_\tau - \frac{W'(m_\tau)}{\eps}+\mu_{\eps}>0$ on $N$ (preservation of mean convexity) we get $\eps \Delta (m_\infty -s \rho_1) (x) - \frac{W'((m_\infty -s \rho_1) (x))}{\eps}+\mu_{\eps}>0$, contradicting (\ref{eq:one-sided_minimization1}).
\end{proof}

\begin{Prop}
\label{Prop:ACconstant_Ricci}
If $\Rc{N}>0$ then any stable solution to $\ca{F'}{\eps,\mu}=0$ on $N$ must be a constant (here $\mu$ can be any constant.)
\end{Prop}

\begin{proof}
Let $u$ be a stable solution to $\ca{F'}{\eps,\mu}(u)=0$. We test the stability inequality $Q(\cdot, \cdot)\geq 0$ on a test function of the form $|\nabla u|\phi$ for $\phi\in C^2(N)$. We get (this follows from the expression of $Q$ given above by using Bochner's identity, see \cite{ChoMan}, \cite{Ton})
$$\int_{N\setminus \{|\nabla u|=0\}} \left(|A_{\eps}|^2 + \text{Ric}_N\left(\frac{\nabla u}{|\nabla u|}, \frac{\nabla u}{|\nabla u|}\right) \right) \eps |\nabla u|^2 \phi^2\leq \int_N \eps |\nabla u|^2 | \nabla \phi|^2,$$
where $A_{\eps}$ stands for the Allen--Cahn second fundamental form of $u$ and $|A_{\eps}|^2=|D^2 u|^2-|\nabla|\nabla u||^2\geq 0$. We plug in $\phi=1$ so the positiveness of $\Rc{N}$ gives $\nabla u \equiv 0$.
\end{proof} 

Lemma \ref{lem:m_infty_stable} and Proposition \ref{Prop:ACconstant_Ricci} give that $m_\infty$ is a constant. There exist exactly two constant solutions of $\ca{F'}{\eps,\mu}=0$: for this, the constant $k$ must satisfy $W'(k) = \eps \mu k$ (and therefore $W(k)\approx c_w^2 \eps^2 \mu^2$), so one constant is slighly larger than $-1$ and the other is slighly larger than $+1$ when $\eps$ is sufficiently small (both solutions are trivially checked to be stable). In our case, since $m_\infty>m_0$ and $m_0>1/2$ on an open neighbourhood of $\overline{M}$, we conclude that $m_\infty$ is the constant slightly larger than $+1$, which we will denote by $k_{\mu_{\eps}}$:

\begin{equation}
 \label{eq:m_infty_constant}
m_\infty \equiv k_{\mu_{\eps}}.
\end{equation}

\medskip

\textit{Flow from $g_{t_0+1}$.} We are now ready to consider the negative $(2\s)\ca{F}{\eps,\mu_{\eps}}$-gradient flow $\{h_t\}$ starting at $h_0=g_{t_0+1}$. We first make the initial datum smooth, by considering mollifiers $\eta_\delta$ for $\delta\in(0,\overline{\delta}]$ as in Appendix \ref{mollifiers} and $\overline{\delta}$ sufficiently small to preserve the strict inequality with $m$, i.e.~to ensure $g_{t_0+1}\star \eta_{\delta} >m$  for $\delta\in(0,\overline{\delta}]$. The family 

\be
\label{eq:mollified_path}
\delta\in (0,\overline{\delta}] \to g_{t_0+1}\star \eta_{\delta}\in W^{1,2}(N)
\ee
is continuous in $\delta$ and extends by continuity at $\delta=0$ with value $g_{t_0+1}$ (see Remark \ref{oss:continuity_in_delta}). Continuity is also valid for $\delta\in (0,\overline{\delta}] \to g_{t_0+1}\star \eta_{\delta}\in C^0(N)$. As a consequence, $\Ece (g_{t_0+1}\star \eta_{\delta})$ varies continuously with $\delta$ and therefore, upon choosing $\overline{\delta}$ possibly smaller, we also have, in addition to (\ref{eq:mollified_path}) and to $g_{t_0+1}\star \eta_{\overline{\delta}} >m$, that the following holds for all $\delta\in(0,\overline{\delta}]$,

\be
\label{eq:energy_mollified_path}
\Ece (g_{t_0+1}\star \eta_{\delta})\leq \Ece(g_{t_0+1})+\frac{1}{4}\tau.
\ee

\medskip

We now let $h_0 = g_{t_0+1}\star \eta_{\overline{\delta}}$ be the initial condition for the negative $(2\s)\ca{F}{\eps,\mu_{\eps}}$-gradient flow:
\be
\label{eq:F_flow_h}
\left\{\begin{array}{ccc}
\eps \frac{\p}{\p t} h_t  = \eps \Delta h_t - \frac{W'(h_t)}{\eps}+\mu_{\eps}\\
h_0 = g_{t_0+1}\star \eta_{\overline{\delta}}
       \end{array}\right. .
\ee
By the maximum principle, since $m_0<h_0$, the two flows (\ref{eq:F_flow}) and (\ref{eq:F_flow_h}) preserve $m_t<h_t$ for all $t$.\footnote{We have smoothed the initial data in order to use basic linear parabolic theory to obtain smoothness at all times and thus use the classical maximum principle. The other option is to use $g_{t_0+1}$ as initial condition and prove that it becomes smooth after a short time.}
Since $g_{t_0+1}\leq 1$ by construction, we also have $h_0<k_{\mu_{\eps}}$, therefore $h_t <k_{\mu_{\eps}}$ for all $t>0$ by the maximum principle. On the other hand we saw that $m_t \to k_{\mu_{\eps}}$ as $t\to \infty$, therefore (with smooth convergence, in particular we have continuity in $t$ for $t\in [0,\infty]\to h_t \in W^{1,2}(N)$)

\begin{equation}
\label{eq:flow_to_k}
h_t \to k_{\mu_{\eps}}\,\, \text{ as } \,\, t\to \infty . 
\end{equation}

\textit{Evaluation of $\Ece$ on the path $h_t$}.
Let us estimate the value of $\Ece$ along this path. For this, note that $\ca{F}{\eps,\mu_{\eps}}$ is decreasing along the flow $\{h_t\}$, therefore $\Ece(h_t)\leq \Ece(h) + 2\frac{\mu_{\eps}}{2\s} {\Hc}^{n+1}(N)$ for all $t$ (where we used $h_t<2$ for all $t$). This implies that $\Ece$ is bounded above indepedently of $\eps$; more precisely, recalling that $\Ece(h_0)\leq 2{\Hc}^n(M)-\tau + O(\eps|\log\eps|)$, we can absorbe $\frac{\mu_{\eps}}{\s} {\Hc}^{n+1}(N)$ in the error $O(\eps|\log\eps|)$ for $\eps$ sufficiently small. In other words we obtain, for $\eps_2\leq \eps_1$ sufficiently small, the upper bound 
\be
\label{eq:energy_path_3_Ricci7}
\Ece(h_t)\leq 2{\Hc}^n(M)-\frac{3}{4}\tau+O(\eps|\log\eps|)
\ee
for all $t$ and for all $\eps<\eps_2$.

\medskip

To complete the path, we will now connect $h_\infty=k_{\mu_{\eps}}$ to $+1$ by a negative $(2\s)\Ece$-gradient flow:
\be
\label{eq:F_flow_k}
\left\{\begin{array}{ccc}
\eps \frac{\p}{\p t} k_t  = \eps \Delta k_t - \frac{W'(k_t)}{\eps}+\mu_{\eps}\\
k_0 = k_{\mu_{\eps}}
       \end{array}\right. .
\ee
Since $(2\s)\ca{E'}{\eps}(k_{\mu_{\eps}}) = - \mu_{\eps}<0$ it follows that mean-convexity is preserved and $k_t$ decreases in $t$. Moreover $+1$ is a solution and thus acts as a lower barrier. In conclusion the negative $(2\s)\Ece$-gradient flow starting at $k_{\mu_{\eps}}$ tends to $+1$. (One can check that in fact $k_t$ is a constant for each $t$.)
The energy $\Ece(k_t)$ is decreasing, so the same upper bound that we had in (\ref{eq:energy_path_3_Ricci7}) holds:
\begin{equation}
\label{eq:energy_path_4_Ricci7}
\Ece(k_t)\leq 2{\Hc}^n(M)-\frac{3}{4}\tau+O(\eps|\log\eps|)
\end{equation}
for all $t$ and for all $\eps<\eps_2$.

\section{Conclusion of the proof of Theorems \ref{thm:compare2M_Ricci7}, \ref{thm:mult1_Ricci7} and \ref{thm:two-sided}}
\label{final_argument_Ricci7}

In the previous sections we exhibited (given $M$ as in Theorem \ref{thm:compare2M_Ricci7}, which also fixed $B$ and $\tau$ by Remark \ref{oss:choiceB_Ricci7} and Lemma \ref{lem:unstable_addcnst_region_Ricci7}) for all sufficiently small $\eps$ (namely $\eps<\eps_2$) the following five continuous paths in $W^{1,2}(N)$: (\ref{eq:path-1_Ricci7}) reversed, (\ref{eq:gt_path2}), (\ref{eq:gtr_path3}), (\ref{eq:mollified_path}), (\ref{eq:F_flow_h}). In the order just given, these paths have matching endpoints, therefore their composition in the same order provides a continuous path in $W^{1,2}(N)$ for all $\eps<\eps_2$, that starts at the constant $-1$ and ends at the constant $+1$ and such that
$$\Ece \text{ along this path is} \leq 2{\Hc}^n(M)-\min\left\{\frac{3}{4}\tau, \frac{3{\Hc}^n(B)}{4}\right\} + O(\eps|\log\eps|),$$
thanks to (\ref{eq:energy_path_1_Ricci7}), (\ref{eq:energy_gtr_Ricci7}), (\ref{eq:energy_mollified_path}), (\ref{eq:energy_path_3_Ricci7}), (\ref{eq:energy_path_4_Ricci7}). 
Choosing $\eps_3$ sufficiently small to ensure that $\eps<\eps_3\Rightarrow|O(\eps|\log\eps|)|\leq \min\left\{\frac{\tau}{4}, \frac{\mathcal{H}^n(B)}{4}\right\}$ the above bound gives, for all $\eps<\eps_3$, that the maximum of $\Ece$ on the path is at most $2{\Hc}^n(M)-\min\left\{\frac{\tau}{2}, \frac{{\Hc}^n(B)}{2}\right\}$.

The path is in the admissible class for the minmax construction in \cite{Gua}, therefore the maximum on this specific path controls from above the minmax value $c_{\eps}$ achieved by the index-$1$ solution $u_{\eps}$ obtained from \cite{Gua} (for all $\eps<\eps_3$). Summarising, for every $M\subset N$ as in Theorem \ref{thm:compare2M_Ricci7} there exist ${\eps}_3>0$, $\tau>0$ and $B\subset M$ (non-empty) such that for all $\eps<{\eps}_3$
\be
\label{eq:energy_u_eps}
c_{\eps}=\Ece(u_{\eps})\leq 2{\Hc}^n(M)-\min\left\{\frac{\tau}{2}, \frac{{\Hc}^n(B)}{2}\right\}.
\ee
This concludes the proof of the strict inequality in Theorem \ref{thm:compare2M_Ricci7}. 

\medskip

For Theorem \ref{thm:mult1_Ricci7} it suffices to observe that the integral varifold $V$ produced in \cite{Gua} is (thanks to \cite{Wic}, \cite{TonWic}) such that each connected component of $\text{reg}_V$ (the smoothly embedded part of $\spt{V}$) has the properties needed so that it can be used in place of $M$ in Theorem \ref{thm:compare2M_Ricci7}, or in (\ref{eq:energy_u_eps}) above; moreover, the mass $\|V\|(N)$ of $V$ is $\lim_{{\eps}_i\to 0} c_{{\eps}_i}$ (see Section \ref{recall_Gua}). Letting $M$ be any connected component of $\text{reg}_V$ and denoting by $\theta\in \N$ its (constant) multiplicity, using (\ref{eq:energy_u_eps}) we get $\theta \mathcal{H}^n(M)\leq \|V\|(N)<2\mathcal{H}^n(M)$. This implies $\theta=1$ and the multiplicity assertion in Theorem \ref{thm:mult1_Ricci7} is proved.

\medskip

The fact that the minimal hypersurface is two-sided then follows immediately, since under multiplicity-$1$ convergence (and by the lower energy bounds in \cite{Gua}) we have that $u_{{\eps}_i}\to u_\infty$ in $BV(N)$, where $u_\infty$ is a non-constant function that takes values in $\{-1,+1\}$ and, moreover, $V$ is the multiplicity-$1$ varifold associated to the reduced boundary of the set (of finite perimeter) $\{u_\infty =+1\}$ (there is no ``hidden boundary'' in the limit). We therefore have a global normal on $\text{reg}_V$ (the interior- or the exerior-pointing normal for $\p \{u_\infty =+1\}$). Theorem \ref{thm:two-sided} is therefore proved.

\begin{oss}
 \label{oss:connectedness}
Note that $\text{reg}_{V}$ has to be connected, since each connected\footnote{We point out that connectedness of $\text{reg}_{V}$ and of $\spt{V}$ are in fact the same thing by the varifold maximum principle.} component of it is unstable (because it is two-sided and $\Rc{N}>0$) and therefore the Morse index of $\text{reg}_{V}$ is at least the number of its connected components. On the other hand, by multiplicity-$1$ convergence (or by \cite{Hie}, \cite{Gas}) the Morse index of $\text{reg}_{V}$ is $\leq 1$. An alternative argument for the connectedness, that does not rely on two-sidedness, can be given by means of the maximum principle for stationary varifolds (\cite{Ilm} \cite{Wic}) and the Frankel property\footnote{The proof of the Frankel property can be adapted because we have local stability for $V$ and so the shortest geodesic between two connected components of $\spt{V}$ must have endpoints on the smooth parts (not on the singular set), by the same reasoning used in Lemma \ref{lem:planar_tangent}, see also \cite[Theorem 2.10]{Zhou1}.} for $\Rc{N}>0$ (using the regularity results \cite{TonWic}, \cite{Wic}).
\end{oss}

\appendix
\section{Mollifiers}
\label{mollifiers}

We explain in detail the mollification procedure used in Section \ref{reach_1}. For this appendix, notation is reset. Let $(N,g)$ be a closed Riemannian manifold of dimension $n+1$ and $f:N\to \R$ in $W^{1,\infty}(N)$. We are going to produce, for every $\delta>0$ sufficienly small, a smooth function $f_\delta:N\to \R$ such that $f_\delta \to f$ strongly in $W^{1,2}(N)$ as $\delta\to 0$ (even $W^{1,p}$ for every $p<\infty$, but we will not need this). The function $f_\delta$ is defined as a convolution $f\star \eta_\delta$, for a suitable mollifier $\eta_\delta$. Moreover we will check that, if additionally $\nabla f \in BV(N)$, then we have, for all $\delta$ sufficiently small, that $\Delta \, f_\delta = (\Delta f)\star \eta_\delta + E_\delta$, where $(\Delta f)\star \eta_\delta$ is the convolution of the Radon measure $\Delta f$ with the mollifier $\eta_\delta$ and hence it is identified with its (smooth) density with respect to $\mathcal{H}^{n+1}$, and $E_\delta$ is a smooth function bounded in $L^\infty$ by a constant that only depends on $N$. It would not suffice for our scopes in Section \ref{reach_1} to have a convolution procedure that gives $\Delta \, f_\delta \to \Delta f$ as measures, therefore we give an ad hoc contruction here.

\medskip

We begin with the definitions. The standard smooth mollifier on $\R$ is $\eta(x)=e^{-\frac{1}{1-x^2}}$ for $|x|< 1$, and $\eta(x)=0$ for $|x|\geq 1$. In the following, $\delta<\text{inj}(N)$. We then let $\eta_\delta:N\times N\to \R$ be defined as

$$\eta_\delta(x,y)=\left\{\begin{array}{ccc}
\frac{1}{c_n} \frac{1}{\delta^{n+1}} \eta\left(\frac{d(x,y)}{\delta}\right)& \text{ for } d(x,y)< \delta\\
0 &  \text{ for } d(x,y)\geq \delta
                          \end{array}\right.;$$
here $d$ is the Riemannian distance on $N$ (note that in the first line $y$ is in the geodesic ball centred at $x$ with radius $\delta$) and $c_n=\int_{B_1^{n+1}(0)} \eta(|x|)d\mathcal{L}^{n+1}=(n+1)\om_{n+1}\int_0^1 \eta(s) s^n ds$, where the integration is with respect to the Lebesgue $(n+1)$-dimensional measure. Therefore for every $x$, using normal coordinates centred at $x$, the function $\eta_\delta \circ \exp_x$ integrates to $1$ in the ball of radius $\delta$ in the tangent space to $N$ at $x$, endowed with the Euclidean metric. Moreover, denoting by $B_\delta(x)$ the geodesic ball centred at $x$, there exist $\delta_0< \text{inj}(N)$ and $C_N>0$ such that, for all $x\in N$ and for all $\delta\leq \delta_0$, we have
\begin{equation}
\label{eq:int_eta}
\int_{B_\delta(x)} \eta_\delta(x,y) d\mathcal{H}^{n+1}(y)=1+O(\delta^2),
\end{equation}
where $|O(\delta^2)|\leq C_{N} \delta^2$. The constant $C_N$ depends ony on the curvature of $N$, more precisely on the maximum of the modulus of the sectional curvature (recall that $N$ is compact). 

\textit{Choice of $\delta_0$.} The sectional curvatures of $N$ are bounded in modulus since $N$ is compact. Recalling Riccati's equation and the Bishop-G\"unther inequalities (see the final inequality in the proof of \cite[Theorem 3.17]{Gray}, combined with \cite[(3.23)]{Gray} in the case $P=\{x\}$) there exist $\delta_0< \text{inj}(N)$ and $C_N>0$ such that for all $x\in N$ and for $\delta\leq \delta_0$ we have 
$$|\mathcal{H}^{n}(\p B_\delta(x))-(n+1)\om_{n+1}\delta^n|\leq C_{N} (n+1)\om_{n+1}\delta^{n+2},$$
where $\om_{n+1}$ is the Euclidean volume of the unit ball in $R^{n+1}$. At the same time we can also ensure the following (see \cite[(3.35)]{Gray}, or also \cite[Lemma 12.1]{Gray2}). For all $x\in N$ and for $\delta\leq \delta_0$, denoting by $H_{x,\delta}$ the mean curvature function on the geodesic sphere of radius $\delta$ around the point $x$ (with respect to the outward-pointing normal, hence $H_{x,\delta}\leq 0$) we have ($-\frac{n}{\delta}$ is the Euclidean mean curvature of the sphere of radius $\delta$ in $\R^{n+1}$)
$$\left|H_{x,\delta}+\frac{n}{\delta}\right| \leq C_N \delta \,\,\,\,\, \text{ on } \p B_\delta(x).$$

\textit{Proof of (\ref{eq:int_eta}).} This follows by using the coarea formula in $B_\delta(x)$ for the function $d(x,\cdot)$, for which $|\nabla d(x,\cdot)|=1$. By the choice of $\delta_0$ above we have $C_N>0$ such that for all $x\in N$ and for $s\leq \delta_0$, $|\mathcal{H}^{n}(\p B_s(x))-(n+1)\om_{n+1}s^n|\leq C_{N} (n+1)\om_{n+1}s^{n+2}$. Then using the coarea formula we get $\int_{B_\delta(x)} \eta_\delta(x,y)  d\mathcal{H}^{n+1}(y)=\frac{1}{c_n} \frac{1}{\delta^{n+1}} \int_0^\delta \mathcal{H}^n(\p B_s(x)) \eta\left(\frac{s}{\delta}\right) ds\leq \frac{1}{c_n} \frac{1}{\delta^{n+1}} (n+1)\om_{n+1}\int_0^\delta s^n \eta\left(\frac{s}{\delta}\right) ds + \frac{1}{c_n} \frac{1}{\delta^{n+1}}C_{N} (n+1)\om_{n+1} \int_0^\delta s^{n+2} \eta\left(\frac{s}{\delta}\right) ds$ and using $s^2\leq \delta^2$ in the second term we have $\int_{B_\delta(x)} \eta_\delta(x,y)  d\mathcal{H}^{n+1}(y)\leq \frac{1}{c_n}  (n+1)\om_{n+1}\int_0^1 t^n \eta\left(t\right) dt + \delta^2 \frac{1}{c_n} \frac{1}{\delta^{n+1}}C_{N} (n+1)\om_{n+1} \int_0^\delta s^{n} \eta\left(\frac{s}{\delta}\right) ds = 1+C_N \delta^2$. For the other inequality, namely $\int_{B_\delta(x)} \eta_\delta(x,y)  d\mathcal{H}^{n+1}(y)\geq 1-C_N \delta^2$, one proceeds similarly.

\medskip

From now on we take $\delta\leq \delta_0$. The convolution of an $L^\infty$ function $f:N\to \R$ with $\eta_\delta$ is the function $f\star \eta_\delta:N\to \R$ defined as follows: 

\be
\label{eq:def_mollify_function}
(f\star \eta_\delta)(x)=\int_{N} f(y)\eta_\delta(x,y) d\mathcal{H}^{n+1}(y).
\ee
This is a smooth function thanks to the smoothness of $\eta_\delta$ in $(x,y)$. Note that we have chosen a convolution kernel that does not integrate exactly to $1$, however (\ref{eq:int_eta}) suffices to ensure:

\begin{lem}
\label{lem:convolution_W12}
Let $f\in W^{1,\infty}(N)$. Then 

\noindent (i) $f\star \eta_\delta\to f$ uniformly on $N$;

\noindent (ii) $f\star \eta_\delta\to f$ in $W^{1,2}(N)$.
\end{lem}

\begin{proof}[proof of Lemma \ref{lem:convolution_W12} (i)]
For all $x$ we have $\int_N |f(y)-f(x)| \eta_\delta(x,y)d\mathcal{H}^{n+1}(y) = \int_N |f(y)-f(x)| \frac{\eta_\delta(x,y)}{1+O(\delta^2)} d\mathcal{H}^{n+1}(y) +  \int_N |f(y)-f(x)| \frac{O(\delta^2)}{1+O(\delta^2)} \eta_\delta(x,y)d\mathcal{H}^{n+1}(y)$, where $O(\delta^2)$ is the function in (\ref{eq:int_eta}). The first term is bounded by $L_f \int_N |x-y| \frac{\eta_\delta(x,y)}{1+O(\delta^2)} d\mathcal{H}^{n+1}(y)\leq L_f \delta$, where $L_f$ is the Lipschitz constant of $f$. The second term is bounded in absolute value by $\tilde{C}_N \|f\|_{C^0(N)} \delta^2$ for all sufficiently small $\delta$. Therefore $\int_N |f(y)-f(x)| \eta_\delta(x,y)d\mathcal{H}^{n+1}(y)$ tends to $0$ uniformly in $x$. Then we compute, recalling (\ref{eq:int_eta}),

$$(f\star \eta_\delta)(x) -f(x)= \int_{N} f(y)\eta_\delta(x,y) d\mathcal{H}^{n+1}(y) - \int_N f(x) \frac{\eta_\delta(x,y)}{1+O(\delta^2)} d\mathcal{H}^{n+1}(y)=$$
$$=\int_{N} (f(y)-f(x))\eta_\delta(x,y) d\mathcal{H}^{n+1}(y)+\int_N f(x) \frac{O(\delta^2)}{1+O(\delta^2)}\eta_\delta(x,y) d\mathcal{H}^{n+1}(y).$$
The last term is bounded in absolute value by $\tilde{C}_N \|f\|_{C^0(N)} \delta^2$ for all sufficiently small $\delta$. Therefore
\be
\label{eq:uniform_convergence_molli}
(f\star \eta_\delta) \to f \,\, \text{ uniformly on } N.
\ee
\end{proof}

\begin{proof}[proof of Lemma \ref{lem:convolution_W12} (ii)]
We can choose a finite cover of $N$ by geodesic balls of radius $\delta_0$ in which we fix a local orthonormal frame. In each ball $U\subset N$, we let $\{v_\ell\}_{\ell=1}^{n+1}$ denote the $g$-orthonormal frame. We can make the non-restrictive assumption that the collection of open sets $\tilde{U}$ obtained by setting $\tilde{U}=\{x\in U:\text{dist}(x,\p U)\geq \delta_0/2\}$ still constitutes a finite cover of $N$. Our final aim is to prove that for each $U$ and for every $\ell$ we have $\int_{\tilde{U}} \left|(\nabla (f\star \eta_\delta) -\nabla f)\cdot v_\ell\right|^2 \to 0$ as $\delta\to 0$. There are only finitely many open sets $\tilde{U}$, so this implies that $\int_N \left|\nabla (f\star \eta_\delta) - \nabla f\right|^2 \to 0$. (Here $|\quad|$ stands for the $g$-norm, $\nabla$ for the metric gradient and $\cdot$ for the $g$-scalar product of vectors.) 

We divide the proof in two parts. In step 1 we will show that, writing $v$ for one of the $v_\ell$, we have $(\nabla f\cdot v)\star \eta_\delta \to (\nabla f\cdot v)$ in $L^2(\tilde{U})$ (by the choice of $\tilde{U}$, these convolutions can be defined by staying inside $U$ for $\delta<\delta_0/2$). 
In step 2 we will prove that $(\nabla f \cdot v)\star \eta_\delta - \nabla (f\star \eta_\delta) \cdot v$
tends to $0$ in $L^\infty({\tilde{U}})$. The two steps together then give 
$$\int_{\tilde{U}} |\nabla (f\star \eta_\delta) \cdot v - \nabla f \cdot v|^2 \to 0$$
as $\delta\to 0$, which is our aim. 

Step 1. The first observation is that if $q\in L^\infty(N)$ then 

\be
\label{eq:almost_everywhere_Lebesgue}
\text{ for $\mathcal{H}^{n+1}$-a.e.~$x$ we have }\,\, \int_N |q(y)-q(x)| \eta_\delta(x,y)d\mathcal{H}^{n+1}(y) \to 0 \,\,\text{ as }\,\,\delta\to 0.
\ee
This follows by writing, as done for Lemma \ref{lem:convolution_W12} (i), $\int_N |q(y)-q(x)| \eta_\delta(x,y)d\mathcal{H}^{n+1}(y) = \int_N |q(y)-q(x)| \frac{\eta_\delta(x,y)}{1+O(\delta^2)} d\mathcal{H}^{n+1}(y) +  \int_N |q(y)-q(x)| \frac{O(\delta^2)}{1+O(\delta^2)} \eta_\delta(x,y)d\mathcal{H}^{n+1}(y)$. The second term tends to $0$ as argued earlier. The first term tends to $0$ if $x$ is a Lebesgue point of $q$ (hence for almost all $x$). Then we have, with $\nabla f \cdot v$ in place of $q$:

$$\int_{\tilde{U}} |((\nabla f\cdot v)\star \eta_\delta)(x) -(\nabla f\cdot v)(x)|^2 d\mathcal{H}^{n+1}(x)=$$
$$\int_{\tilde{U}} \left|\int_{U} \left((\nabla f\cdot v)(y) - (\nabla f\cdot v)(x)\right)\eta_\delta(x,y) d\mathcal{H}^{n+1}(y)+\right.$$ $$+\left.(\nabla f\cdot v)(x)\int_U \frac{O(\delta^2) \eta_\delta(x,y) }{1+O(\delta^2)}d\mathcal{H}^{n+1}(y) \right|^2 d\mathcal{H}^{n+1}(x) \underbrace{\leq}_{|a+b|^2\leq 2a^2 +2b^2} $$
$$2 \int_{\tilde{U}} \underbrace{\left|\int_{U} \left((\nabla f\cdot v)(y) - (\nabla f\cdot v)(x)\right)\eta_\delta(x,y) d\mathcal{H}^{n+1}(y)\right|^2}_{\to 0 \text{ by (\ref{eq:almost_everywhere_Lebesgue}) for a.e.~$x$}} d\mathcal{H}^{n+1}(x) + \tilde{C}_N \|\nabla f\|_{L^\infty(N)}^2 \delta^4 .$$
(In the last term, we have included $\mathcal{H}^{n+1}(\tilde{U})\leq \mathcal{H}^{n+1}(N)$ in the constant $\tilde{C}_N$.) The braced integrand in the first term tends to $0$ for a.e.~$x$ by (\ref{eq:almost_everywhere_Lebesgue}). Moreover, the braced expression is bounded for every $x$ by $4\|\nabla f\|_{L^{\infty}(N)}^2 (1+O(\delta^2))^2$, which is summable on $N$. Hence we can use dominated convergence to conclude that the first term tends to $0$ as $\delta \to 0$. The second tends to $0$ a well, therefore we conclude that
$$(\nabla f\cdot v)\star \eta_\delta \to (\nabla f\cdot v) \,\,\text{ in } L^2(\tilde{U}).$$ 

\textit{Step 2.} We compute the difference between the two (smooth) functions $(\nabla f \cdot v)\star \eta_\delta$ and $\nabla (f\star \eta_\delta) \cdot v$ and prove that it goes to $0$ uniformly on $\tilde{U}$. We work in normal coordinates centred at an arbitrary point $O\in \tilde{U}$, namely in the ball $D=\{x\in \R^{n+1}:|x|<\delta_0/2\}$, with exponential map $\text{exp}_O:D\to B_{\delta_0/2}(O)\subset U$. We will evaluate the difference of the two functions at $O$, making sure that the result does not depend on $O$. Since we are interested in $\nabla (f\star \eta_\delta) \cdot v$, we need to let $x$ vary in a neighbourhood of $O$ before evaluating the derivative, therefore we will assume $x\in \{x\in \R^{n+1}:|x|<\delta_0/4\}$ and $\delta<\delta_0/4$, so that $y$ stays in $D$.

We use the customary notation $g_{ij}$ for the metric coefficients, $\sqrt{|g|}$ for the volume density induced by $g$. 
We denote by $h$ the Lipschitz function on $D$ given by $f\circ \text{exp}_O:D\to \R$ and by $\rho:D\times D\to \R$ the mollifier $\rho(x,y)=\eta_\delta(\text{exp}_O(x), \text{exp}_O(y))$, for an arbitrary $\delta<\frac{\delta_0}{4}$. We point out that $\rho(0,y)=\frac{1}{c_n\delta^{n+1}}\eta\left(\frac{|y|}{\delta}\right)$ because we are in normal coordinates, where $|\cdot|$ denotes the Euclidean length. We write $\nabla_g$ to denote the metric gradient in $D$, $(\nabla_g)^i=g^{ij}\p_{x_j}$. Let $v_\ell$ be represented, in the chart, by $\sum v_{\ell}^{j}\p_{j}$. We fix an arbitrary $\ell$ and let $v=(v^1, \ldots, v^{n+1})=(v_\ell^1, \ldots, v_\ell^{n+1})$. We will write $\cdot$ between two vectors to denote the scalar product induced by $g$, so $\nabla_g h \cdot v = \sum g_{ij} g^{ia}\p_{x_a}h\, v^j=\delta_j^a \p_{x_a}h\, v^j=\p_{x_j}h \,v^j$ ($=dh(v)$). 
We restrict to $x\in D_{{\delta_0}/4}$ and we compute for $\delta<\frac{\delta_0}{4}$ the coordinate expression for $\nabla (f\star \eta_\delta) \cdot v$ (integration is in $dy$ unless otherwise specified):

\be
\label{eq:first_partial_conv}
\p_{x_j} \left(\int_D h(y)\rho(x,y)\sqrt{|g|}(y) dy \right) v^j(x) =
\ee
$$= v^j(x) \underbrace{\int_D h(y)\p_{x_j} (\rho(0,y-x))\sqrt{|g|}(y)}_{I}+\underbrace{v^j(x) \int_D h(y)\p_{x_j} (\rho(x,y)-\rho(0,y-x))\sqrt{|g|}(y)}_{II}.$$
Working on the first term, and using the notation $\rho(0,\cdot)=\rho_0(\cdot)$, we have
$$I= -\int_D h(y) (\p_{x_j} \rho_0)(y-x)\sqrt{|g|}(y)\,\,\underbrace{=}_{y-x=z}\,\, -\int_D  h(x+z) (\p_{x_j} \rho_0)(z)\sqrt{|g|}(x+z) dz=$$
$$= \int_D  (\p_{x_j} h)(x+z) \rho_0(z)\sqrt{|g|}(x+z)dz +  \underbrace{\int_D  h(x+z) \rho_0(z) (\p_{x_j}\sqrt{|g|})(x+z)dz}_{III}=$$
$$\underbrace{=}_{x+z=y}\,\, \int_D  (\p_{x_j} h)(y) \rho_0(y-x)\sqrt{|g|}(y) dy + III=  $$
$$=\int_D  (\p_{x_j} h)(y) \rho(x,y)\sqrt{|g|}(y) +\underbrace{\int_D  (\p_{x_j} h)(y) (-\rho(x,y)+\rho(0,y-x))\sqrt{|g|}(y)}_{IV}+ III.  $$
Consider the first term from the last line, recalling that $v^j(x)$ multiplies $I$ in (\ref{eq:first_partial_conv}):
$$v^j(x)\int_D  (\p_{x_j} h)(y) \rho(x,y)\sqrt{|g|}(y) =$$ 
$$= \underbrace{\int_D  v^j(y)(\p_{x_j} h)(y) \rho(x,y)\sqrt{|g|}(y)}_{((\nabla f \cdot v)\star \eta_\delta)(\text{exp}_O(x))} + \underbrace{\int_D (v^j(x)-v^j(y)) (\p_{x_j} h)(y) \rho(x,y)\sqrt{|g|}(y)}_{V}.$$
Having rewritten the first term in the last line, we evaluate (\ref{eq:first_partial_conv}) at $x=0$ to obtain 
$$(\nabla (f\star \eta_\delta ) \cdot v)(O) - ((\nabla f \cdot v)\star \eta_\delta)(O)= 
\left.V\right|_{x=0} +v^j(0)\, \left.IV\right|_{x=0} +v^j(0)\,\left.III\right|_{x=0} + \left.II\right|_{x=0}.$$ 
It is immediate that $\left.IV\right|_{x=0}=0$. In $V$ we have $\rho(0,y)=0$ for $d(0,y)=|y|\geq \delta$, therefore $|v^j(0)-v^j(y)|\leq C |y|$ for some constant $C$ that depends on derivatives of $v$ in $U$ and can be thus chosen independently of $U$ (there are finitely many $U$'s) and of $v_\ell$ (finitely many smooth vector fields). We therefore get that 
$\left.V\right|_{x=0}$ is bounded in modulus by $C \|\nabla f\|_{L^{\infty}} \delta (1+O(\delta^2))\leq C'  \|\nabla f\|_{L^{\infty}} \delta$ for some $C'$ that depends only on the choices of charts and vector fields.
In $\left.III\right|_{x=0}$, the integrand is non-zero only for $|z|\leq \delta$. Let $\tilde{C}_N>0$ be an upper bound for the modulus of the second derivatives of the volume element in a normal coordinate system of radius $\delta_0$ centred at an arbitrary point in $N$ (such a constant exists by the compactness of $N$, the smoothness of the metric and the fact that $\delta_0<\text{inj}(N)$). Recalling that in normal coordinates the metric coefficients have vanishing first derivatives at $0$, we get that $\left.|III|\right|_{x=0}\leq C\|f\|_{C^0} \delta$ for all $\delta\leq \delta_0$, with a constant $C$ that only depends on the geometric data. For $II$, recall that $\rho(x,y)=\frac{1}{c_n\delta^{n+1}}\eta\left(\frac{d(x,y)}{\delta}\right)$, where $d$ is the Riemannian distance (induced by $g$); so for each $y$ we have $\p_{x_j} \rho(\cdot,y)=\frac{1}{c_n\delta^{n+2}}\eta'\left(\frac{d(\cdot,y)}{\delta}\right)\p_{x_j} d(\cdot, y)$. On the other hand $\rho(0,y-x)=\frac{1}{c_n\delta^{n+1}}\eta\left(\frac{|y-x|}{\delta}\right)$ so for each $y$ we have $\p_{x_j} \rho(0,y-\cdot)=\frac{1}{c_n\delta^{n+2}}\eta'\left(\frac{|y-\cdot|}{\delta}\right)\p_{x_j} |y-\cdot|$. At $\cdot=0$ we have, for every $y\neq 0$, $\p_{x_j} |y-\cdot| = \p_{x_j} d(\cdot, y) =-\frac{y_j}{|y|}$, because we are in normal coordinates, and $d(0,y)=|y|$. Therefore $\left.II\right|_{x=0}=0$.

We have therefore proved that $(\nabla (f\star \eta_\delta ) \cdot v)(O) - ((\nabla f \cdot v)\star \eta_\delta)(O)\leq C\delta$ for $C$ independent of $O$ and so $|(\nabla f \cdot v)\star \eta_\delta-(\nabla (f\star \eta_\delta)\cdot v)|\to 0$ uniformly on $\tilde{U}$. 
\end{proof}

\begin{oss}
\label{oss:continuity_in_delta}
Also note that $\delta \in (0,\delta_0] \to (f\star \eta_\delta) \in W^{1,2}(N)$
is continuous, since $\eta_\delta$ changes smoothly with $\delta$ (in fact, this curve is differentiable on $(0,\delta_0)$). Similarly, $\delta \in (0,\delta_0] \to (f\star \eta_\delta) \in C^0(N)$ is continuous.
\end{oss}

Next we are going to be interested in $\Delta (f\star \eta_\delta)$ under the additional assumption on $f$ that $\nabla f\in BV(N)$. Here $\Delta$ denotes the Laplace-Beltrami operator. Recall that $f\star \eta_\delta$ is smooth, so $\Delta (f\star \eta_\delta)$ is smooth on $N$. We shall compare this function with $(\Delta f) \star \eta_\delta$, where $\Delta f$ is a Radon measure. For a Radon measure $\mu$ on $N$ we define the (smooth) function $\mu\star \eta_\delta:N\to \R$ as follows:

\be
\label{eq:def_mollify_Radon}
(\mu\star \eta_\delta)(x) = \int \eta_\delta(x,y) d\mu(y).
\ee

\begin{lem}
 \label{lem:second_derivative}
Let $f\in W^{1,\infty}(N)$ with $\nabla f \in BV(N)$. There exists $C_N>0$ depending only $N$ (in fact, on the curvature of $(N,g)$) such that, for all $\delta<\delta_0$,
$$\|(\Delta f) \star \eta_\delta - \Delta (f\star \eta_\delta)\|_{L^\infty(N)}\leq C_N \|f\|_{L^\infty(N)}.$$ 
\end{lem}

\begin{proof}
We work in a normal system of coordinates centred at an arbitrary $O\in N$. Let $D$ be the ball centred at $0\in \R^{n+1}$ of radius $\delta_0$, with $\text{exp}_O:D \to B_{\delta_0}(O)$ denoting the exponential map. 
We keep notation as in the proof of step 2 of Lemma \ref{lem:convolution_W12} (ii), in particular we set $\rho(x,y)=\eta_\delta(\text{exp}_O(x), \text{exp}_O(y))$ and $\rho_0(\cdot)=\rho(0, \cdot)$. The Laplace-Beltrami operator $\Delta$ is, in the coordinate chart, $\frac{1}{\sqrt{|g|}} \frac{\p}{\p x_i}\left(\sqrt{|g|}g^{ij}\frac{\p}{\p x_j}\right)$, so $\Delta f$ is $\frac{1}{\sqrt{|g|}} \frac{\p}{\p x_i}\left(\sqrt{|g|}g^{ij}\frac{\p h}{\p x_j}\right)$, where $h=f\circ \exp_O$. We compute $\Delta (f\star \eta_\delta)(x)$ in the normal chart, keeping $x\in \{|\cdot|\leq \delta_0/2\}$ and $\delta<\delta_0/2$ so that $y\in D$ in the following computations. Differentiating (with respect to the $x$-variables),

\be
\label{eq:second_partial_conv}
\Delta \left(\int_D h(y)\rho(x,y)\sqrt{|g|}(y) dy \right)=
\ee
$$= \underbrace{\int_D h(y)\Delta (\rho_0(y-x))\sqrt{|g|}(y)}_{I}+\underbrace{\int_D h(y)\Delta (\rho(x,y)-\rho(0,y-x))\sqrt{|g|}(y)}_{II},$$
where derivatives are taken in $x$ and integration is in $dy$. We compute, for each $y$:

$$\Delta (\rho_0(y-x)) = -\frac{1}{\sqrt{|g|(x)}} \p_{x_i}\left(\sqrt{|g|(x)}g^{ij}(x)\right)\left(\p_{x_i}\rho_0\right)(y-x) + g^{ij}(x) (\p^2_{x_i x_j} \rho_0)(y-x);$$
$$(\Delta \rho_0)(y-x) = \frac{1}{\sqrt{|g|(y-x)}} \p_{x_i}\left(\sqrt{|g|(y-x)}g^{ij}(y-x)\right)\left(\p_{x_i}\rho_0\right)(y-x) + g^{ij}(y-x) (\p^2_{x_i x_j} \rho_0)(y-x).$$
Therefore

\be
\label{eq:diff_coeff_g}
\Delta (\rho_0(y-x))-(\Delta \rho_0)(y-x) =\left(g^{ij}(x)- g^{ij}(y-x) \right)(\p^2_{x_i x_j} \rho_0)(y-x)-
\ee
$$\left(\frac{1}{\sqrt{|g|(x)}} \p_{x_i}\left(\sqrt{|g|(x)}g^{ij}(x)\right)+\frac{1}{\sqrt{|g|(x-y)}} \p_{x_i}\left(\sqrt{|g|(x-y)}g^{ij}(x-y)\right)\right)\left(\p_{x_i}\rho_0\right)(y-x) $$
and we can rewrite $I$ as follows (so that in the second term we will be able to use (\ref{eq:diff_coeff_g})):

$$I=\int h(x+z) (\Delta \rho_0)(z)\sqrt{|g|}(x+z) dz+ \underbrace{\int h(y) \left(\Delta (\rho_0(y-x))-(\Delta \rho_0)(y-x)\right) \sqrt{|g|}(y)}_{III}.$$ 
We want to evaluate at $x=0$. Let $\tilde{\rho}_0=\rho_0\circ \text{exp}_O^{-1}$ and recall that $f=h\circ \text{exp}_O^{-1}$, then the first term on the right-hand-side of the last equality, evaluated at $x=0$, is $\int_N  f\, \Delta \tilde{\rho}_0 \, d\mathcal{H}^{n+1}$. Integrating by parts we rewrite is as $\int_N  \Delta f\, \tilde{\rho}_0\, d\mathcal{H}^{n+1}$ and we get

$$\left.I\right|_{x=0}=\int \Delta h(z) \rho_0(z)\sqrt{|g|}(z) dz+ \left.{III}\right|_{x=0}=$$ 
$$ =\int \Delta h(y) \rho(0,y)\sqrt{|g|}(y) dy+ \left.{III}\right|_{x=0}=((\Delta f)\star \eta_\delta)(O) + \left.{III}\right|_{x=0}.$$
Recall (\ref{eq:second_partial_conv}); the statement of Lemma \ref{lem:second_derivative} will therefore follow by estimating $\left.{II}\right|_{x=0}$ and $\left.{III}\right|_{x=0}$, taking care that the estimates should be independent of $O$. For $\left.{III}\right|_{x=0}$, we use (\ref{eq:diff_coeff_g}) and the following two facts. Firstly, $\int \p^2_{ij} \rho_0(y) dy = \frac{1}{c_n \delta^2} \left(\int_{B_1}\p^2_{ij}\eta_1(0,y)dy\right)$ and $\int \p_{i} \rho(y) dy = \frac{1}{c_n \delta} \left(\int_{B_1} \p_i \eta_1(0,y)dy\right)$ (the two integrals on the right-hand-sides depend only on $\eta$, which is fixed, so they can be absorbed into constants). Secondly, since we are in normal coordinates, $g^{ij}(0)=\delta^{ij}$, $\p_{x_k} g_{ij} = 0$ at $0$ for all $k$; since $N$ is compact, there exists a constant $C_{N,\delta_0}$ such that in any normal system of coordinates centred at a point of $N$ and with radius $\delta_0$ ($<\text{inj}(N)$), the second derivatives of the metric coefficients are bounded in modulus by $C_{N,\delta_0}$. Therefore

\noindent $\left|g^{ij}(0)- g^{ij}(-y) \right|\leq C_{N,\delta_0} |y|^2$ and $\left| \frac{1}{\sqrt{|g|(-y)}} \left.\p_{x_i}\right|_{x=0}\left(\sqrt{|g|(x-y)}g^{ij}(x-y)\right)\right|\leq C_{N,\delta_0}|y|$.
Using these two facts in (\ref{eq:diff_coeff_g}), and noting that $|y|\leq \delta$ on the set where the integrand of $\left.{III}\right|_{x=0}$ does not vanish, we get that $\left.{III}\right|_{x=0}$ is bounded in modulus by $\|f\|_{L^\infty} C_{n,N,\delta_0} \|\eta\|_{C^2(\R)} = C \|f\|_{L^\infty}$, with $C$ depending only on fixed geometric data.

\noindent For $\left.{II}\right|_{x=0}$ we need to compare, for each $y\neq 0$, $\Delta (\rho(x,y))$ and $\Delta (\rho_0(y-x))$, both evaluated at $x=0$. Let us write $m_\delta (\cdot)= \frac{1}{c_n\delta^{n+1}} \eta\left(\frac{\cdot}{\delta}\right)$, $m_\delta:\R\to \R$. Then, denoting by $d$ the distance induced by $g$ and by $|\quad|$ the Euclidean distance, by $|\quad|_g$ the vector length for $g$, and by $\nabla$ the $g$-gradient, we get for each $y\neq 0$ (derivatives with respect to $\cdot$)
$$\Delta (\rho(\cdot,y)) = \Delta \left( m_\delta(d(\cdot,y)) \right)= m_\delta^{''}(d(\cdot,y))|\nabla d(\cdot,y)|_g^2 + m_\delta^{'}(d(\cdot,y))\Delta d (\cdot,y),$$
$$\Delta (\rho_0(y-\cdot)) = \Delta \left(m_\delta(|y-\cdot|) \right)= m_\delta^{''}(|y-\cdot|)\left|\,\nabla |y-\cdot|\,\,\right|_g^2 + m_\delta^{'}(|y-\cdot|)\Delta |y-\cdot|.$$
Evaluating at $\cdot=0$ we note that $m_\delta^{''}(d(0,y))|\nabla d(0,y)|_g^2 = m_\delta^{''}(|y|)|\nabla |y||_g^2$ and $m_\delta^{'}(d(0,y)) = m_\delta^{'}(|y|)$, because in normal coordinates we have $d(0,y)=|y|$ and $\nabla |y-\cdot|=\nabla d(\cdot,y)=-\frac{y}{|y|}$ at the point $\cdot=0$ (for any chosen $y\neq 0$). We therefore need to compare, for any $y\neq 0$, $\Delta d(\cdot,y)$ and $\Delta |y-\cdot|$ at $0$. The former is the opposite of the mean curvature at $0$ of a geodesic sphere centred at $y$ with radius $d(0,y)$ (as usual, we compute the scalar mean curvature with respect to the outward-pointing normal to the sphere). On the other hand, recall that computing $\Delta$ at $0$ is the same as computing the Euclidean Laplacian, therefore $-\Delta |y-\cdot|$ at $0$ is the Euclidean mean curvature at $0$ of a Euclidean sphere centred at $y$ with radius $|y|$, hence $\Delta |y-\cdot|=-\frac{n}{|y|}$ at $0$. The difference $\Delta d(\cdot,y)-\Delta |y-\cdot|$ is therefore bounded in modulus by $C_N |y|$, thanks to the initial choice of $\delta_0$. Since we can take $|y|\leq \delta$ in $\left.{II}\right|_{x=0}$ (because $\rho=0$ otherwise) we can estimate $\left.\left(\Delta (\rho(x,y))-\Delta (\rho_0(y-x))\right)\right|_{x=0}$ in modulus by $\|m_\delta^{'}\|_{L^\infty} C_N \delta\leq \frac{1}{\delta^{n+1}}C_{\eta} C_N$; integrating on $\{|y|\leq \delta\}$ we get that $\left.{II}\right|_{x=0}$ is bounded in modulus by $C_\eta C_N \|f\|_{L^\infty}$. 

We have thus obtained $|(\Delta f) \star \eta_\delta - \Delta (f\star \eta_\delta)|(O)\leq C \|f\|_{L^\infty}$ with $C$ depending only on $(N,g)$ and on the fixed entities $\delta_0$, $\eta$. The arbitraryness of $O$ gives the result.
\end{proof}

\end{document}